\documentclass[12pt, reqno]{amsart}
\usepackage{enumerate}
\usepackage{amsfonts}
\usepackage{amsmath}
\usepackage{color}
\usepackage{hyperref}

\usepackage{graphicx}

\usepackage{tikz}
\usetikzlibrary{arrows,graphs,shapes,positioning}
\usetikzlibrary{decorations.markings}

\tikzstyle arrowstyle=[scale=1.2]

\usepackage{geometry} 

\newtheorem{thm}{Theorem}[section]

\newtheorem{cor}[thm]{Corollary}
\newtheorem{lem}[thm]{Lemma}
\newtheorem{prop}[thm]{Proposition}
\newtheorem*{prob*}{Problem}
\newtheorem*{thm*}{Theorem}
 \theoremstyle{definition}
\newtheorem{defn}[thm]{Definition}

\newtheorem{rem}[thm]{Remark}

\numberwithin{equation}{section}

\newcommand{\bo}{\mathcal{O}}

\newcommand{\Epsilon}{\mathcal{E}}

\newcommand{\R}{\mathbb R}

\newcommand{\eps}{\varepsilon}

\DeclareMathOperator{\diag}{diag}

\newcommand{\re}{\mathop{\mathrm{Re}}}
\newcommand{\im}{\mathop{\mathrm{Im}}}
\newcommand{\Tr}{\mathop{\mathrm{Tr}}}

\newcommand{\eins}{\leavevmode\hbox{\small1\kern-3.8pt\normalsize1}}
\newcommand{\e}{\,\mbox{e}}

\newcommand{\abs}[1]{\left\vert#1\right\vert}

\begin{document}
\title{\bf{Singular Value Statistics For the  Spiked Elliptic Ginibre Ensemble}}

\author{Dang-Zheng Liu}
\address{Key Laboratory of Wu Wen-Tsun Mathematics, CAS, School of Mathematical Sciences, University of Science and Technology of China, Hefei 230026, P.R.~China}
 \email{dzliu@ustc.edu.cn}

\author{Yanhui Wang}
\address{School of Mathematics and Statistics, Henan University, Henan, 475001, P.R.~China}
\email{yhwang@henu.edu.cn}

\keywords{Spiked Wishart matrix,  elliptic ensemble, Pfaffian  point processes, Fredholm Pfaffian, Tracy-Widom distribution}

\begin{abstract}
The complex elliptic Ginibre ensemble with coupling $\tau$ is a complex Gaussian  matrix interpolating   between  the Gaussian Unitary Ensemble (GUE) and the Ginibre ensemble. It has been known for some time  that its eigenvalues form a determinantal point process in the complex plane. A recent result of Kanazawa and Kieburg (arXiv:1804.03985) shows that  the   singular values  form   a Pfaffian point process. In this paper we turn to consider an extended  elliptic Ginibre ensemble, which  connects  the  GUE and the spiked Wishart matrix, and prove that the singular values still build a Pfaffian point process with  correlation kernels  expressed by   contour integral representations.  As   $\tau$ tends to 1  at  a certain critical rate,  we prove that the  limiting distribution  of the largest singular value  is described   as a  new Fredholm Pfaffian  series,   which connects  two distributions $F_{\mathrm{GUE}}$ and  $F^{2}_{\mathrm{GUE}}$ where  $F_{\mathrm{GUE}}$ is
the  GUE Tracy-Widom distribution.
For fixed $\tau$,  we prove  the Baik-Ben Arous-P\'ech\'e  transition of   the largest singular value  and the sine kernel in the bulk.
  We also observe  a  crossover phenomenon    at the origin when $\tau$ tends to 1  at  another critical rate. \end{abstract}
\date{\today}

\maketitle

\tableofcontents

\section{Introduction and main results}

\subsection{Introduction}

The study of non-Hermitian random matrices     was first initiated in 1965
by Ginibre \cite{ginibre65},  who introduced Gaussian random matrices (complex, real and quaternion real) as a mathematical
extension of Hermitian random matrix theory.
 Although Ginibre  was  motivated  by the works of Wigner, Dyson and Mehta on   random Hamiltonians,    he  remarked   that `\textit{apart from the intrinsic interest of the problem, one may hope that the methods and
results will provide further insight in the cases of physical interest or suggest as yet
lacking applications}'. It is indeed true that  now non-Hermitian random matrices  have many interesting  and important applications, including   modelling of
fractional quantum-Hall effect,  Coulomb plasma \cite{forrester10}, and stability
of the complex ecological system \cite{May72,AT12b,AT15}; see \cite{KS11} for an overview.

In the case of the complex  Ginibre ensemble, which is an $N\times N$ complex random matrix  $X$ with  density  proportional  to $\exp\{-\Tr(X^* X)\}$, its  eigenvalues  form a determinantal point process in the complex plane \cite{ginibre65} and the corresponding eigenvalue statistics have been extensively studied, see e.g. \cite{forrester10}. Another  well-studied model is the (complex) elliptic Ginibre ensemble which was introduced  as an interpolation between Hermitian and non-Hermitian random matrices; see \cite{girko85} and \cite{FKS97,FKS98}. To be precise,  let $H_1,H_2$ be two independent random  matrices sampled  from  the Gaussian Unitary Ensemble (GUE for short), which is described by the Gaussian measure with density $C_{N}\exp\{-\Tr(H^2)\}$ on the space of $N\times N$ Hermitian matrices, then  $X=\sqrt{1+\tau}H_{1}+i\sqrt{1-\tau}H_{2}$ with a coupling parameter  $\tau\in [0, 1]$ is an interpolating ensemble   between the  Ginibre ensemble  ($\tau=0$) and  the GUE  ($\tau=1$).  The matrix $X=[X_{i,j}]$ has a probability density  function (PDF for short)
			 \begin{equation} \label{ellipticPDF}
			 \widetilde{P}_{N}(X)= \frac{1}{(\pi\sqrt{1-\tau^2})^{N^2}}\exp\!\Big\{-\frac{1}{1-\tau^2}\Tr\left(X^{*}X-\frac{\tau}{2}\big(X^{2}+(X^{*})^{2}\big)\right)\Big\}.
			 \end{equation}
Here it is worth stressing that (i) $\{X_{j,j}:1\leq j\leq N\} \cup \{(X_{i,j}, X_{j,i}):1\leq i<j\leq N\}$ are independent and (ii)  $\tau$  is the correlation coefficient between  $\re\{X_{i,j}\}$ and $\re\{X_{j,i}\}$   while
$-\tau$  is the correlation coefficient between   $\im\{X_{i,j}\}$ and $\im\{X_{j,i}\}$  for any $i<j$.

The model \eqref{ellipticPDF}  has an interesting  variation    in Quantum Chromodynamics (QCD), which relates to  a certain Dirac operator
$$\mathcal{D}= \begin{bmatrix}
       0& iX\\
       iX^{*} & 0
      \end{bmatrix}.$$
  This chiral form     has  at least three major applications in   QCD: 4D QCD at high temperature, 3D QCD at  finite   isospin chemical potential and 3D lattice QCD for staggered fermions; see  \cite{KK1,KK2} and references therein  for detailed discussion. An important characteristic of  $ \mathcal{D}$ is that  its eigenvalues come in pairs  $\{ \lambda_j, -\lambda_j\}$ because of  the chiral symmetry, and this is crucial in the study of  the Hermitian Wilson Dirac operator
  $$\mathcal{D}_5= \begin{bmatrix}
       0& X\\
       X^{*} & 0
      \end{bmatrix}+\mu H,$$
  where $H$ is a  GUE matrix of size $2N$. The latter  was considered by Akemann and Nagao \cite{AN2011}
 as an  interpolating ensemble between the GUE and chiral GUE.

It is known from \cite{FKS97} that the elliptic Ginibre ensemble \eqref{ellipticPDF}  has  eigenvalue  PDF    \begin{align}\label{EEdensity}
      \widetilde{\mathcal{P}}_{N}(z)
      &=\frac{1}{(\pi\sqrt{1-\tau^2})^{N} \prod_{j=1}^{N}j!}  \exp\!\Big\{-\frac{1}{1-\tau^2}\sum_{j=1}^{N}\left(|z_{j}|^2-\frac{\tau}{2}\big(z_{j}^{2}+\overline{z}_{j}^{2}\big)\right)\Big\} \,
       \prod_{j<k}|z_k-z_j|^2, \nonumber
    \end{align}
and it is a simple corollary  that the statistical state is a determinantal point process in the complex plane. For any fixed $\tau\in [0,1)$,  the
empirical spectral measure  of eigenvalues converges   to a uniform law   on  an ellipse
$$\Big\{z=x+iy\in \mathbb{C}:  \frac{x^2}{(1+\tau)^2}+\frac{y^2}{(1-\tau)^2}\leq 1 \Big\}.$$
This  is called an   elliptic law in the literature and in fact  holds for iid type random matrices, see e.g. \cite{girko85, NO15}.
 When the parameter $\tau$ approaches 1, its support   rapidly degenerates from an ellipse   to an interval $[-2,2]$ of the real line. Then certain crossover phenomena   for local eigenvalue statistics may occur when $\tau$ changes  at  a certain  critical rate.
 Actually, with the help of Hermite polynomials,  in the so-called  weak non-Hermiticity situation that $1-\tau=\kappa/N$,
Fyodorov, Khoruzhenko and Sommers observed  a   transition  of eigenvalues in the  bulk    from Wigner-Dyson to
Ginibre   statistics;  see \cite{FKS97,KS11}. In another critical regime   of  $1-\tau=\kappa N^{-1/3}$,   Bender  \cite{bender10} proved an  edge  transition for  real part of complex eigenvalues  from Poisson to Airy point processes.  See \cite{ACV16} for very recent  results related to these settings.

The squared singular values (to be abbreviated  `singular values' throughout the paper  in this context for brevity) of the complex Ginibre  ensemble $X$,   equivalently,
eigenvalues of the well-known  complex Wishart  ensemble $X^{*}X$,  admit a determinantal structure and have been studied extensively; see e.g. \cite{forrester93, Johansson00,Johnstone01, tw1994, tw1996}.  However,   it is only very recently  that    Kanazawa and Kieburg \cite{KK2} have  identified the statistical state formed by the   singular values  of the elliptic Ginibre ensemble.   Instead of a determinantal point process, it is proved that   the singular values  build  a Pfaffian point process.  Further,  explicit formulas for  skew orthogonal polynomials are constructed  via the supersymmetry method  and  a finite sum of  skew-orthogonal polynomials for correlation kernel  is thus derived.
In principle, this has opened up the possibility for  asymptotic analysis of the local statistics,  e.g.,  limiting distributions of both the smallest and largest eigenvalues.

 In the present paper  we study the singular values of a generalization of the elliptic Ginibre ensemble,  interpolating between  the GUE and  the spiked Wishart ensemble.  Specifically,  for positive integers $M\geq N$, introduce an $M\times N$ rectangular matrix   $A = \begin{bmatrix}
      I_{N}&0_{N\times(M-N)}
    \end{bmatrix}$ where $I_N$ is an identity matrix of size $N$ and $0_{N\times(M-N)}$ is a null matrix, and  let $\Sigma$ be an $N \times N$  positive definite matrix with eigenvalues $\sigma_{1}, \ldots, \sigma_{N}>0$.  We define an  $M\times N$ random  matrix $X$  with  PDF  with respect to  the Lebesgue measure on $\mathbb{C}^{M \times N}$
  \begin{equation}\label{densitySCEE}
    P_{N}(X) = \frac{1}{Z_{N}} \exp\!\Big\{-\eta_{+} \Tr\big(X^{*}X\Sigma\big) + \frac{1}{2}\eta_{-}\Tr\big((X A)^{2} + (A^{*}X^{*})^{2}\big)\Big\},
  \end{equation}
  where   $\eta_{+}>0$,   $\sigma_1, \ldots, \sigma_{N}>\eta_{-}/\eta_{+}$ and the normalization constant    \begin{equation*}  \frac{1}{Z_{N}}= \Big(\frac{\eta_{+}}{\pi} \Big)^{M} \big(\!\det{\Sigma}\big)^{M-N}
     \sqrt{\det \Big(\Sigma \otimes \Sigma - \big(\frac{\eta_{-}}{\eta_{+}}\big)^{2}I_{N^2}\Big)},
  \end{equation*}
  and throughout  the paper
 \begin{equation}
    \eta_{+}=\frac{1}{1-\tau^2},  \qquad  \eta_{-}=\frac{\tau}{1-\tau^2}, \qquad \tau\in (0,1). \label{etapm}
  \end{equation}
This may be treated as   a spiked rectangular version of the complex elliptic ensemble,  since the case of  $M=N$ and $\Sigma=I_{N}$  reduces to  the elliptic ensemble. In the special case  of  $\tau=0$,  $X^{*}X$ is  the complex Wishart matrix with a spike.  We remark that spiked Wishart matrices  appear  in a wide range of applications of Random Matrix  Theory (RMT),  for example   in statistics,  signal
 processing and random growth models;   see e.g. \cite{BBP05, peche14}.

The data matrix $X$ defined  in \eqref{densitySCEE}  plays a significant role  in multivariate statistics and data science. When $\tau=0$ in \eqref{densitySCEE},   $X$ consists of $M$  independent  column variables (samples), each  of which has a  covariance matrix $\Sigma^{-1}$. However, for many types of  microarray datasets, e.g., genetic and  financial data,  the case of independent samples is then not applicable; see e.g.
\cite{AT12, efron09} and references therein.    It raises  a serious problem  about   the seemingly natural independence assumption  of samples and    suggests the need for conducting independence test  among samples. The challenge   arising from correlation among samples  has recently been taken on  by Allen and Tibshirani \cite{AT12}, Chen and Liu \cite{CL18}, Pan  et al.  \cite{pan14}.  Correlated sample matrices  also appear in MIMO channels  wireless communications \cite{HG12,MGB07}.   When $\tau\in (0,1)$,  the data matrix $X$  in \eqref{densitySCEE} is in fact transposable, meaning that    both rows and columns are potentially correlated  \cite{AT12}.   Thus the mixing effect from correlation between columns  caused by $\Sigma$ and  correlation  between  entries  $\re\{X_{i,j}\}$ and $\re\{X_{j,i}\}$ ($i\neq j$) by     $\tau$  shows that  both rows and columns are   correlated.
To the best of our  knowledge,  our study   is the first example of transposable data matrices  for which the singular value PDF is   given in closed form and moreover     local    statistics  can be  analyzed  in more details; cf. Theorems
\ref{edgecritthm}, \ref{largestcritthm}, \ref{bulk limit}, \ref{edge limit} and \ref {originlimit} below.

 One  fundamental problem in RMT is to prove universality of local eigenvalue statistics, typically according to the spectral position of limiting eigenvalue density and symmetry class  of matrix entries.    In the Hermitian case the universal phenomenon   is usually described by  Airy kernel at the soft edge, Sine kernel in the bulk and Bessel kernel at the hard edge.     In the literature there are a lot of relevant articles on universality phenomenon, say, \cite{EY2010, EY2012, TV2010,TV2011,TV2012} or the  recent monographs  \cite{AB11, Deift99, EY2017, forrester10} and references therein.
 At the soft edge, the spiked Wishart matrix with $\Sigma$ being a low rank (rank $r$ say) perturbation of the identity  matrix   can  unlock  a separation  of the $r$ largest eigenvalues. This  novel phenomenon is usually called the BBP transition after Baik, Ben Arous and P\'ech\'e   \cite{BBP05}, and has  aroused great concern   in recent years.
See \cite{FL16}  for a hard-edge analog of the BBP transition for the smallest singular values  in a non-central complex Wishart matrix.

Furthermore, RMT affords  a versatile description of  crossover  phenomena  between different symmetry classes. Two classical interpolating random matrix ensembles  were introduced
and solved by Mehta and Pandey  \cite{PM83, MP83}; see \cite{FNH99} for recent and relevant results. These works  describe   transitions between the  GUE, relevant for systems without time-reversal invariance, and the Gaussian orthogonal ensemble as well as the Gaussian symplectic ensemble   for systems with time-reversal symmetry and
even or odd spin, respectively.  Various limits of the symplectic-unitary and orthogonal-unitary transition kernels were  first considered in  \cite{FNH99}.  Similarly,
 as  an interpolating ensemble  between the
 GUE and the Ginibre ensemble,  the eigenvalues of  the complex elliptic  Ginibre ensemble describe a transition from complex to real eigenvalues in the weak non-Hermiticity   regime; see \cite{bender10, FKS97,FKS98} or \cite{ACV16} for recent developments.   It is reasonable to expect  that  the  singular values will exhibit an analogous effect.
 Indeed, the recent work of Kanazawa and Kieburg \cite{KK1} conjectured
a crossover behaviour involving the smallest singular values, and employed a combination of
analytic and numeric methods to predict an expression for the corresponding microscopic level
density at the origin; see \cite[eq.(14)]{KK1}.

\subsection{Main results}
As  noted above,  Kanazawa and Kieburg \cite{KK2} have studied the singular values for the  complex elliptic Ginibre ensemble defined  in \eqref{ellipticPDF} and  proved that these  form  a Pfaffian point process with   correlation kernel in terms of certain   skew orthogonal polynomials.  One of our main  goals  is to  further study  the largest singular value and
investigate its  limiting distribution.
In the case when $\tau$ changes at a critical rate such that  $1-\tau \propto N^{-1/3}$, we will show that the largest singular value exhibits  a new crossover phenomenon.
It   is described   by  a  particular  Fredholm Pfaffian  series,  and  is actually an interpolation between  the  GUE Tracy-Widom distribution  $F_{\mathrm{GUE}}$  and  the distribution of the maximum of two independent and identically distributed random variables with  distribution  $F_{\mathrm{GUE}}$ (cf.\,Theorem\,\ref{SVGUE} in Sect. \ref{lastsect}).

More generally,  for the spiked model \eqref{densitySCEE}  we derive the singular value PDF
 as  a Pfaffian point process and further  investigate scaling limits  in all different spectral regimes.   We will prove   the  BBP transition (cf. \cite{BBP05})  for  the largest  singular value  when $\tau\in (0,1)$ is fixed,   while
at the  same critical value of $\tau$  we obtain  a multi-parameter    variation of the interpolating distribution given in Definition \ref{pflargest} below. In addition,   our method to solve a $2\times 2$ matrix  kernel  for the Pfaffian point process  in Theorem \ref{pfstruct} of Section \ref{sectkernel} should be of independent interest.

  Our first main result is a Pfaffian structure for  the  eigenvalue PDF of $W=X^{*}X$.
  \begin{thm} \label{SVdensity} With  $X$ defined as in  \eqref{densitySCEE},
   for even $M$ the joint probability density function of eigenvalues  $\lambda_{1}, \ldots, \lambda_{N}$ of $W=X^{*}X$ is given by
    \begin{align}\label{densityCEES}
      f_{N}(\lambda)
      &=\frac{1}{Z_{M,N}} \det\!\big[e^{-(\eta_{+}\sigma_{j} + \eta_{-})\lambda_{k}}\big]_{j,k=1}^{N} \mathrm{Pf}\!\begin{bmatrix}
        \begin{array}{c|c}
          \Epsilon(\eta_{-}\lambda_{j}, \eta_{-}\lambda_{k}) & g_{b}(\lambda_{j})\\
          \hline
          -g_{a}(\lambda_{k}) & \alpha_{a,b}
        \end{array}
      \end{bmatrix}_{\substack{j,k = 1, \ldots, N\\a, b = 1, \ldots, M-N}},
    \end{align}
    where  $Z_{M, N}$ is the normalisation  constant and  with  $I_{0}$ denoting the modified Bessel function
    \begin{align}
      \Epsilon(u, v) = \int_{\mathbb{R}^{2}} dx dy\, \frac{x-y}{x+y} e^{-\frac{1}{2}(x^{2}+y^{2})}       \left(I_{0}(2 x \sqrt{u}) I_{0}(2 y \sqrt{v}) - I_{0}(2 y \sqrt{u})I_{0}(2 x \sqrt{v})\right), \label{pdfeq1}
    \end{align}
  \begin{align}
      g_{a}(u) &= \eta_{-}^{a-1} \int_{\mathbb{R}^{2}} dx dy\, \frac{x-y}{x+y} e^{-\frac{1}{2}(x^{2}+y^{2})}  \left(  x^{2(a-1)}  I_{0}(2 y \sqrt{\eta_{-} u}) - y^{2(a-1)} I_{0}(2 x \sqrt{\eta_{-} u}) \right), \label{pdfeq2}\\
      \alpha_{a,b} &= \eta_{-}^{a+b-2} \int_{\mathbb{R}^{2}} dx dy\,  \frac{x-y}{x+y} e^{-\frac{1}{2}(x^{2}+y^{2})} \big (x^{2(a-1)} y^{2(b-1)}  -  y^{2(a-1)}x^{2(b-1)}\big).\label{pdfeq3}
    \end{align}
  \end{thm}

We remark that  $Z_{M, N}$ can be expressed  in terms of the determinant of the Gram type moment matrix  in
Sect.\,\ref{sectkernel}; cf. \eqref{grammatrix} below.
When some of parameters $\sigma_{1}, \ldots, \sigma_{N}$ coincide, \eqref{densityCEES} reduces to a limiting  density by applying L' Hospital's rule. In the case  that $M=N$ and $\Sigma=I_N$,  it  reduces  to the Kanazawa-Kieburg's result \cite[Sect.\,2]{KK2}. When $\tau$ tends to $0$, the singular value PDF indeed converges to that of the spiked complex Wishart matrix.

  Next,  we will prove that  the singular value PDF  of \eqref{densityCEES} forms a Pfaffian point process and  the associated correlation kernel  $K_{N}(u,v)$ can be    given by   contour integrals. Thus, recalling the definition of $k$-point correlation functions (see e.g. \cite{forrester10,mehta2004}) by
  \begin{equation}
   R_{N}^{(k)}(\lambda_1, \ldots, \lambda_k):=\frac{N!}{(N-k)!} \int \cdots \int f_{N}(\lambda_1, \ldots, \lambda_N) d\lambda_{k+1} \cdots d\lambda_{N},
  \end{equation}
 we then  have
   \begin{equation}
   R_{N}^{(k)}(\lambda_1, \ldots, \lambda_k)=  \mathrm{Pf} \left[
        K_{N}(\lambda_{i}, \lambda_{j}) \right]_{i,j = 1}^{k}. \label{kernelexpression}
  \end{equation}

{\bf Notation.}  In the present paper    we let   $\sqrt{z}$   denote  the principal value  such that  it is positive for positive $z$, and let  $\sqrt{z^{2}-1}=\sqrt{z+1}\sqrt{z-1}$ and $\sqrt{1-z^{2}}=\sqrt{1+z}\sqrt{1-z}$ unless otherwise specified. The notation  $\mathcal{C}_{A}$ will be used to denote an anti-clockwise contour encircling all  points in the set $A$ but not any other poles of the integrand.

  \begin{thm} \label{thmkernel}
    With the same notations as in Theorem \ref{SVdensity}, for even $M$  the correlation kernel associated with \eqref{densityCEES}
    is given by a $2\times 2$ matrix
      \begin{align}
      K_{N}(u, v) = \begin{bmatrix}
        DS_{N}(u, v) & S_{N}(u, v)\\
        -S_{N}(v, u) & IS_{N}(u, v)
      \end{bmatrix}, \label{matrixkernel}
    \end{align}
    where
     \begin{align}
      DS_{N}(u, v)       &= \frac{\eta_{-}^2}{4\pi}   \int_{\mathcal{C}_{\{\sigma_{1}/\tau, \ldots, \sigma_{N}/\tau\}}}\frac{dz}{2\pi i} \int_{ \mathcal{C}_{\{\sigma_{1}/\tau, \ldots, \sigma_{N}/\tau\}}} \frac{dw}{2\pi i}  \frac{e^{-\eta_{-} v(z+1)} } {\sqrt{z^2-1}} \frac{e^{-\eta_{-} u(w+1)} } {\sqrt{w^2-1}}
        \nonumber\\
      & \quad\times
            \big(zw\big)^{M-N}\frac{z-w}{1-zw}    \prod_{k = 1}^{N} \frac{\sigma_{k}z-\tau}{ \tau z-\sigma_{k}}
        \frac{\sigma_{k}w-\tau}{ \tau w-\sigma_{k}}, \label{DSneq-2}
    \end{align}
    \begin{align}
      S_{N}(u, v) &= \int_{0}^{\infty} dt\,\Epsilon(\eta_{-}v, \eta_{-}t) DS_{N}(u, t),\label{kernel12}\\
      IS_{N}(u, v) &= -\Epsilon(\eta_{-}u, \eta_{-}v) +\int_{0}^{\infty} \int_{0}^{\infty} dsdt\,  \Epsilon(\eta_{-}v,\eta_{-}t) DS_{N}(s, t)  \Epsilon(\eta_{-}u, \eta_{-}s).\label{kernel22}
    \end{align}
     \end{thm}

  Finally, we turn to investigate  local statistical behaviour of the singular values as the matrix size $N$ goes to infinity.  In addition to the usual  sine kernel in  the bulk and Airy kernel at the soft edge in Sect. \ref{sect:noncritical},   attention will be payed  to crossover phenomena
 as the coupling parameter $\tau$ changes from 0 to 1 at a certain  critical rate.
   At the  soft and hard  edges of the spectrum,  we  observe transitions of   the  singular values from those of Ginibre  to GUE  exactly at  critical values of   $\tau=1-2^{2/3}N^{-1/3}\kappa$ and $\tau=1 -  \kappa/N$ with $\kappa\in (0,\infty)$, respectively.  We just state  the  crossover   results at the soft edge   and detail the hard edge  limits in Sect. \ref{sectlimits}.

 For $\kappa\geq 0$ and an integer  $m\geq 0$,   with   $\pi_1, \ldots, \pi_m>0$
 we use double contour integrals to specify the following three functions
  \begin{align}
 S^{(\mathrm{soft})}(\kappa, \pi;u,v)=& \int_{\mathcal{C}_{>}}\frac{dz}{2\pi i }    \int_{\mathcal{C}_{<}}  \frac{dw}{2\pi i }   e^{\frac{1}{3}(w-\kappa)^3-u(w-\kappa)-\frac{1}{3}(z-\kappa)^3+ v(z-\kappa)}    \nonumber\\
& \times  \frac{1}{\sqrt{4zw}}
\frac{z+w}{z-w}   \prod_{k = 1}^{m} \frac{z-\pi_{k}}{ z+ \pi_{k}}\frac{w+\pi_{k}}{w-\pi_{k}}, \label{Ssoft}
  \end{align}
  \begin{align}
 DS^{(\mathrm{soft})}(\kappa, \pi;u,v)=& \int_{\mathcal{C}_{<}}\frac{dz}{2\pi i }    \int_{\mathcal{C}_{<}}  \frac{dw}{2\pi i }   e^{\frac{1}{3}(w-\kappa)^3-u(w-\kappa)+\frac{1}{3}(z-\kappa)^3- v(z-\kappa)}    \nonumber\\
& \times  \frac{1}{\sqrt{4zw}}
\frac{w-z}{z+w}   \prod_{k = 1}^{m} \frac{z+\pi_{k}}{z - \pi_{k}}\frac{w+\pi_{k}}{w - \pi_{k}}, \label{DSsoft}
  \end{align}
    \begin{align}
 IS^{(\mathrm{soft})}(\kappa, \pi;u,v)=& \int_{\mathcal{C}_{>}}\frac{dz}{2\pi i }    \int_{\mathcal{C}_{>}}  \frac{dw}{2\pi i }   e^{-\frac{1}{3}(w-\kappa)^3+u(w-\kappa)-\frac{1}{3}(z-\kappa)^3+v(z-\kappa)}    \nonumber\\
& \times  \frac{1}{\sqrt{4zw}}
\frac{w-z}{z+w}   \prod_{k = 1}^{m} \frac{z-\pi_{k}}{z +\pi_{k}}\frac{w-\pi_{k}}{w + \pi_{k}}, \label{ISsoft}
  \end{align}
  where  $\mathcal{C}_{>}$  is a contour from   $e^{-2i\pi/3}\infty$ to $\delta/2$ and then to   $e^{2i\pi/3}\infty$, while   $\mathcal{C}_{<}$   is a contour  from
 $e^{i\pi/3}\infty$ to $\delta$ and then  to  $e^{-i\pi/3}\infty$ with some $0<\delta<\min\{\pi_1, \ldots, \pi_m\}$.
Further,  introduce a   $2\times 2$ matrix-valued  skew-symmetric kernel
  \begin{equation}
  K^{(\mathrm{soft})}(\kappa, \pi;u, v) =\begin{bmatrix}
       DS^{(\mathrm{soft})}(\kappa, \pi;u,v) &  S^{(\mathrm{soft})}(\kappa, \pi;u,v)  \\
       - S^{(\mathrm{soft})}(\kappa, \pi;v,u)  &  IS^{(\mathrm{soft})}(\kappa, \pi;u,v)
        \end{bmatrix}. \label{mkernel}
\end{equation}

 \begin{defn} \label{pflargest}   Define a
  Fredholm Pfaffian   by the series expansion
       \begin{align}
    F^{(\mathrm{soft})}&(\kappa,\pi;x):=  \mathrm{Pf}[J- K^{(\mathrm{soft})}(\kappa, \pi)] \nonumber \\
    &=1+ \sum_{k=1} ^{\infty}  \frac{(-1)^k}{k!}\int_{x}^{\infty} \cdots   \int_{x}^{\infty}  \mathrm{Pf}[
      {   K^{(\mathrm{soft})}(\kappa, \pi;  u_i, u_j)}]_{i,j = 1}^{k} du_1 \cdots   du_k, \label{pfaffianseries}
  \end{align}
  where  $J$ and  $K^{(\mathrm{soft})}(\kappa, \pi)$ are   the $2\times 2$ integral operators with  kernel
   \begin{equation*}
 J(u, v) =\begin{bmatrix}
       0 &  1_{u=v}  \\
       - 1_{u=v}   &  0
        \end{bmatrix}
\end{equation*}
  and  $K^{(\mathrm{soft})}(\kappa, \pi;u,v)$ given   in  \eqref{mkernel}, respectively.
    \end{defn}
We will see that  the Pfaffian series expansion  in \eqref{pfaffianseries} converges  for any $x\in \R$   and
hence  $F^{(\mathrm{soft})}(\kappa,\pi;x)$ is  well defined (cf. the proof of Theorem \ref{largestcritthm}  in Sect. \ref{proofcritical}).   Moreover,
    as a limit of nondecreasing functions $F^{(\mathrm{soft})}(\kappa,\pi;x)$   is a continuous  nondecreasing   function,    and  we know  from   \eqref{upperbound} and \eqref{seriessum} below that
$F^{(\mathrm{soft})}(\kappa,\pi;x)\to 1$ as $x\to +\infty$.
Although   the proof  of the fact that  $F^{(\mathrm{soft})}(\kappa,\pi;x)\to 0$ as $x\to -\infty$  is not trivial, we believe that $F^{(\mathrm{soft})}(\kappa,\pi;x)$ is a    probability distribution.

Next, we state    the limit theorems about correlation functions  and  the largest eigenvalue.  They characterize certain  interpolating phenomenon   between the   singular values of the GUE and the spiked complex Ginibre ensemble.
 \begin{thm}
 \label{edgecritthm}
 With the same notations as in Theorem \ref{SVdensity} and with     $M=N$ even,  for $\kappa>0$  let   $\tau=1-2^{\frac{2}{3}}N^{-\frac{1}{3}}\kappa$. Given  two fixed nonnegative integer $0\leq m\leq n$,  assume  that $\sigma_{n+1} = \cdots = \sigma_{N} = 1$ and
    \begin{equation}
     \sigma_k=\frac{1}{2}(1+\tau^2)+  N^{-\frac{1}{3}}2^{-\frac{4}{3}} (1-\tau^2)(\pi_{k}-\kappa),  \quad  k= 1, \ldots, m, \label{criticalsigma}
    \end{equation}
    where all $\pi_k>0$.  When  $\pi_1, \ldots, \pi_m$  are in a compact subset of $(0,\infty)$  and $\sigma_{m+1},\ldots,\sigma_n$       are in a compact subset of $(1, \infty)$,
we have       \begin{align}
   \lim_{N \to \infty} & \big(2^{\frac{4}{3}}N^{\frac{1}{3}}\big)^k R_{N}^{(k)}\Big(4N+2^{\frac{4}{3}}N^{\frac{1}{3}}u_1, \ldots,  4N+2^{\frac{4}{3}}N^{\frac{1}{3}}u_k\Big)=  \mathrm{Pf}\big[
      {   K^{(\mathrm{soft})}(\kappa, \pi;  u_i, u_j)}\big]_{i,j = 1}^{k},
  \end{align}
 uniformly for  $u_1, \ldots, u_k$ in a compact subset of $\mathbb{R}$.
    \end{thm}

   \begin{thm}
 \label{largestcritthm}  With  $X$ being defined  in  \eqref{densitySCEE} and  with the same  assumptions  as in Theorem \ref{edgecritthm},
 let $\lambda_{\mathrm{max}}$ be  the largest  eigenvalue  of $W=X^{*}X$. For any  $x$ in a compact set of $\mathbb{R}$,   we have          \begin{align}
  \lim_{N \to \infty}\mathbb{P}\Big(2^{-\frac{4}{3}}N^{-\frac{1}{3}}\big(\lambda_{\mathrm{max}}-4N\big) \leq   x\Big)=  F^{(\mathrm{soft})}(\kappa,\pi;x).
  \end{align}
    \end{thm}

The rest of this article is organized as follows.
In the next Sect. \ref{sectpdf} we are devoted to the derivation of the  singular value PDF in Theorem \ref{SVdensity}.
In Sect. \ref{sectkernel} we show that  singular values form  a Pfaffian point process and all the  sub-kernels of the   $2 \times 2$  matrix   kernel admit contour integral representations. In  Sect. \ref{sectedge}  we focus  on local singular value statistics  in the bulk and at the soft edge. In particular, we will complete the proof of Theorems  \ref{edgecritthm} and \ref{largestcritthm}.
We leave the study of the hard edge limits  at the origin  in  Sect. \ref{sectlimits} and make some further discussions   in Sect. \ref{lastsect}.

  \section{Singular value PDF} \label{sectpdf}
  In this section we present   the derivation of the singular value PDF of $X$,
  by adapting the approach of \cite{KK2} used to analyze  \eqref{ellipticPDF}.

  \begin{proof}[Proof of Theorem \ref{SVdensity}]
    Write $n=M-N$ for simplicity.
    Use  the singular value decomposition $X = U \Lambda V^{*}$, where two unitary matrices  $U \in \mathcal{U}(M)$, $V \in \mathcal{U}(N)$, and
    \begin{equation}
      \Lambda = \begin{bmatrix}
        \Lambda_{0}\\
         0_{n\times N}
      \end{bmatrix} ,  \quad
        \Lambda_{0}=\mathrm{diag}\Big(\sqrt{\lambda_{1}}, \ldots, \sqrt{\lambda_{N}}\Big). \nonumber
    \end{equation}
    Under this change of variables, the standard  approach  of calculating the Jacobian  as in \cite[Chapter 3]{forrester10} shows
    \begin{align*}
      [dX] =
      C_{N,M} \Delta_{N}^{2}(\lambda_{1}, \ldots, \lambda_{N}) \big(\prod_{k=1}^{N} \lambda_{k}^{n} d\lambda_{k} \big)\,d\mu_{M}(U) d\mu_{N}(V),
    \end{align*}
    where $\Delta_{N}(\lambda_{1}, \ldots, \lambda_{N})$  is the Vandermonde determinant, $[dX]$  and $ \mu_{N}$ denote the  Lebesgue measure on $\mathbb{C}^{M \times N}$ and   the Haar measure on the space of unitary $N\times N$ matrices  $\mathcal{U}(N)$ respectively. Integrating over the  unitary matrices gives that   $W=X^* X$ has   eigenvalue PDF
    \begin{align*}
      & f_{N}(\lambda)
      = \frac{C_{N,M}}{Z_{N}} \Delta_{N}^{2}(\lambda) \prod_{k=1}^{N} \lambda_{k}^{n}  \int_{\mathcal{U}(N)} d\mu_{N}(V) \exp\left\{-\eta_{+} \Tr\big(V \Lambda_{0}^{2} V^{*}\Sigma\big)\right\} \nonumber\\
      &\quad \times \int_{\mathcal{U}(M)} d\mu_{M}(U) \exp\left\{\frac{\eta_{-}}{2}\Tr\big((\Lambda A \diag(V^{*}, I_{n}) U)^{2} + (U^{*} \diag(V, I_{n}) A^{*} \Lambda^{\dag})^{2}\big)\right\},
    \end{align*}
    where we have used the simple fact that $V^{*}A=A\diag(V^*, I_{n})$.

    Absorbing the factor $ \diag(V^{*}, I_{n})$ into the group $ \mathcal{U}(M)$  shows
    \begin{align}
      f_{N}(\lambda)
      = \frac{C_{N,M}}{Z_{N}} \Delta_{N}^{2}(\lambda) \big(\prod_{k=1}^{N} \lambda_{k}^{n} \big)\, I_{1} I_{2},     \label{eqn1}
    \end{align}
    where
    \begin{align}
      I_{1} =\int_{\mathcal{U}(N)} d\mu_{N}(V) \exp\left\{-\eta_{+} \Tr\big(V \Lambda_{0}^{2} V^{*}\Sigma\big)\right\},
    \end{align}
    and
    \begin{align}
      I_{2}=  \int_{\mathcal{U}(M)} d\mu_{M}(U) \exp\left\{\frac{\eta_{-}}{2}\Tr\big((\Lambda A U)^{2} + (U^{*} A^{*} \Lambda^{*})^{2}\big)\right\}.
    \end{align}
    Applying the Harish-Chandra-Itzykson-Zuber (HCIZ) integral formula (cf. \cite{forrester10,mehta2004})  we then have
    \begin{align}
      I_{1}
      = \prod_{k=0}^{N-1} k!\,  \frac{\det[e^{-\eta_{+}\sigma_{i}\lambda_{j}}]_{i,j=1}^N}{\Delta_{N}(\sigma) \Delta_{N}(\lambda)}.  \label{eqn2}
    \end{align}

    For  the integral $I_2$, noting  $\Lambda A=\diag(\Lambda_0, 0_n)$,  let's introduce  an extended $M\times M$ matrix
    $$\tilde{\Lambda}=\diag\big( \sqrt{\lambda_1}, \ldots, \sqrt{\lambda_N}, \sqrt{\lambda_{N+1}}, \ldots,\sqrt{\lambda_M}\big).
    $$
    By a new integral formula over the unitary group due to Kanazawa and Kieburg \cite[Appendix A]{KK2}, we obtain     \begin{align}
      I_{2}      &=\lim_{\lambda_{N+1}, \ldots, \lambda_{M} \to 0} \int_{\mathcal{U}(M)} \exp\left\{\frac{1}{2}\eta_{-}\Tr\big((\tilde{\Lambda} U)^{2} + (U^{*}\tilde{\Lambda})^{2}\big)\right\} d\mu_{M}(U) \nonumber\\
      &= \frac{\prod_{k=0}^{M-1} k!}{(2\pi)^{M/2} M!} \eta_{-}^{{-(M-1)M}/{2}}  \int_{\mathbb{R}^{M}} dx \prod_{k=1}^{M} e^{-\frac{1}{2} x_{k}^{2}} \frac{\Delta_{M}^{2}(x_{1}, \ldots, x_{M})}{\Delta_{M}(x_{1}^{2}, \ldots, x_{M}^{2})} \nonumber\\
      &\quad \times \lim_{\lambda_{N+1}, \ldots, \lambda_{M} \to 0} \frac{\prod_{k=1}^{M} e^{-\eta_{-} \lambda_{k}}}{\Delta_{M}(\lambda_{1}, \ldots, \lambda_{M})} \det\left[I_{0}(2 x_{i} \sqrt{\eta_{-} \lambda_{j}})\right]_{1 \leq i, j \leq M}.
    \end{align}
    Applying  the L'H\^{o}spital's rule to take the limit shows
    \begin{align}
      I_{2} &= \frac{\prod_{k=0}^{M-1} k!}{(2\pi)^{M/2} M! \prod_{k=0}^{n-1} k!} \frac{\prod_{k=1}^{N} \lambda_{k}^{-n} e^{-\eta_{-} \lambda_{k}}}{\Delta_{N}(\lambda) \eta_{-}^{(M(M-1))/2}} \int_{\mathbb{R}^{M}} dx \prod_{k=1}^{M} e^{-\frac{1}{2} x_{k}^{2}} \nonumber\\
      &\quad\times  \frac{\Delta_{M}^{2}(x_{1}, \ldots, x_{M})}{\Delta_{M}(x_{1}^{2}, \ldots, x_{M}^{2})} \det\begin{bmatrix}
        \begin{array}{c|c}
          I_{0}(2 x_{i} \sqrt{\eta_{-} \lambda_{j}}) & \frac{\eta_{-}^{a-1}}{(a-1)!}x_{i}^{2(a-1)}
        \end{array}
      \end{bmatrix}_{\substack{i = 1, \ldots, M;\\j = 1, \ldots, N;~a = 1, \ldots, n}}. \label{2.6}
    \end{align}

    By the Schur Pfaffian identity
    \begin{align}
      \frac{\Delta_{M}^{2}(x_{1}, \ldots, x_{M})}{\Delta_{M}(x_{1}^{2}, \ldots, x_{M}^{2})}
      =
      \begin{cases}
        \mathrm{Pf} \left[
          \frac{x_{i} - x_{j}}{x_{i}+x_{j}}\right]_{i,j = 1}^{M},  &\qquad M~\text{even}; \\
        \mathrm{Pf}  \begin{bmatrix}
          \begin{array}{c|c}
            \frac{x_{i} - x_{j}}{x_{i}+x_{j}} 
              &  {1_{M\times 1}}\\
            \hline
            -1_{1\times M}&0
          \end{array}
         \end{bmatrix}_{i,j = 1}^{M}, &\qquad M~\text{odd},
      \end{cases} \label{schur}
    \end{align}
  where $1_{1\times M}=(1, \ldots, 1)$ with $M$ 1's
 and together with Lemma \ref{Bruijn0} below,    for even $M$ we obtain
    \begin{align}
      I_{2} &= \frac{(4\pi)^{-M/2}  \prod_{k=0}^{M-1} k!}{ \prod_{k=0}^{n-1} (k!)^2} \frac{\prod_{k=1}^{N} \lambda_{k}^{-n} e^{-\eta_{-} \lambda_{k}}}{\Delta_{N}(\lambda) \eta_{-}^{M(M-1)/2}}    \mathrm{Pf} \begin{bmatrix}
        \begin{array}{c|c}
          \Epsilon(\eta_{-}\lambda_{j}, \eta_{-}\lambda_{k}) & g_{b}(\lambda_{j})\\
          \hline
          -g_{a}(\lambda_{k}) & \alpha_{a,b}
        \end{array}
      \end{bmatrix}_{\substack{j,k = 1, \ldots, N\\a, b = 1, \ldots, n}}, \label{eqn3}
   \end{align}
    with the same notations as in Theorem \ref{SVdensity}.

    Combining \eqref{eqn1},  \eqref{eqn2} and \eqref{eqn3}, we thus have completed the proof of Theorem \ref{SVdensity}.
  \end{proof}

  \begin{lem}\label{Bruijn0}
    Let  $m\geq 2$ be even. Given a measure space $(\Omega, \nu)$, $\phi_{1}, \ldots, \phi_{m}: \Omega \longrightarrow \mathbb{C}$, and an antisymmetric function  $\epsilon: \Omega \times \Omega \longrightarrow \mathbb{C}$, if
    \begin{equation}
      \epsilon_{i, j} := \int_{\Omega^{2}} \epsilon(x, y) \det\!\begin{bmatrix}
        \phi_{i}(x) & \phi_{i}(y)\\
        \phi_{j}(x) & \phi_{j}(y)
      \end{bmatrix} d\nu(x) d\nu(y) \label{Ematrix}
    \end{equation}
    are well defined for any $1 \leq i, j \leq m$, then
    \begin{align}
      \int_{\Omega^{m}} \mathrm{Pf}\begin{bmatrix}
        \epsilon(x_{i}, x_{j})
      \end{bmatrix}_{i, j=1}^{m} \det\left[\phi_{i}(x_{j})\right]_{i, j=1}^{m} \prod_{i=1}^{m} d\nu(x_{i})
      = \frac{m!}{2^{\frac{m}{2} }} \mathrm{Pf}\begin{bmatrix}
        \epsilon_{i,j}
      \end{bmatrix}_{i, j=1}^{m}. \label{generaleq}
    \end{align}
  \end{lem}
  \begin{proof} We proceed by induction.
   Obviously, the desired result is true for $m=2$.

   Suppose that it remains  true  for $m-2$.  For $m$, starting from the recursive definition of the Pfaffian, we have
    \begin{align}\label{recPf}
      \mathrm{Pf} \begin{bmatrix}
        \epsilon(x_{i}, x_{j})
      \end{bmatrix}_{1 \leq i, j \leq m}
      &= \sum_{k=2}^{m} (-1)^{k} \epsilon(x_{1}, x_{k}) \mathrm{Pf}\begin{bmatrix}
        \epsilon(x_{i}, x_{j})
      \end{bmatrix}_{i, j \notin \{1, k\}}.
    \end{align}
    Here $\mathrm{Pf}\begin{bmatrix} \epsilon(x_{i}, x_{j}) \end{bmatrix}_{i, j \notin \{1, k\}}$ is the Pfaffian of the matrix $[\epsilon(x_{r}, x_{s})]_{1 \leq r, s \leq m}$ with both $1$-th and $k$-th rows and columns removed.
    Correspondingly, take the Laplace expansion of the determinant $\det\left[\phi_{i}(x_{j})\right]$ as
    \begin{align}\label{LapExp}
      \det\left[\phi_{i}(x_{j})\right]
      &= \sum_{1 \leq i < j \leq m} (-1)^{i+j+k+1} \det\begin{bmatrix}
        \phi_{i}(x_{1}) & \phi_{i}(x_{k}) \\
        \phi_{j}(x_{1}) & \phi_{j}(x_{k})
      \end{bmatrix} \det\begin{bmatrix}
        \phi_{r}(x_{s})
      \end{bmatrix}_{r \notin \{i, j\}, s \notin\{1, k\}}.
    \end{align}
    Taking \eqref{recPf} and \eqref{LapExp} into consideration, we obtain
    \begin{align*}
      \mathrm{LHS \ of\  \eqref{generaleq}}&= \sum_{k=2}^{m} \sum_{1 \leq i < j \leq m} (-1)^{i+j+k+1} (-1)^{k} \int_{\Omega^{m}} \epsilon(x_{1}, x_{k}) \mathrm{Pf}\begin{bmatrix}
        \epsilon(x_{i}, x_{j})
      \end{bmatrix}_{i, j \notin \{1, k\}}\\
      &\quad \times \det\begin{bmatrix}
        \phi_{i}(x_{1}) & \phi_{i}(x_{k}) \\
        \phi_{j}(x_{1}) & \phi_{j}(x_{k})
      \end{bmatrix} \det\begin{bmatrix}
        \phi_{r}(x_{s})
      \end{bmatrix}_{r \notin \{i, j\}, s \notin\{1, k\}} \prod_{i=1}^{m} d\nu(x_{i})\\
      &= \sum_{k=2}^{m} \sum_{1 \leq i < j \leq m} (-1)^{i+j+1}  \int_{\Omega^{2}} \epsilon(x, y) \det\begin{bmatrix}
        \phi_{i}(x) & \phi_{i}(y) \\
        \phi_{j}(x) & \phi_{j}(y)
      \end{bmatrix} d\nu(x) d\nu(y)\\
      &\quad \times \int_{\Omega^{m-2}} \mathrm{Pf}\begin{bmatrix}
        \epsilon(x_{i}, x_{j})
      \end{bmatrix}_{1 \leq i, j \leq m-2} \det\begin{bmatrix}
        \phi_{r}(x_{s})
      \end{bmatrix}_{r \notin \{i, j\}, 1 \leq s \leq m-2} \prod_{i=1}^{m-2} d\nu(x_{i}).    \end{align*}

    By the induction hypothesis, we further get
     \begin{align*}
      \mathrm{LHS \ of\  \eqref{generaleq}}
      &= \frac{(m-1)!}{2^{\frac{m-2}{2}}} \sum_{1 \leq i < j \leq m} (-1)^{i+j+1} \epsilon_{i,j} \mathrm{Pf}\begin{bmatrix}
        \epsilon_{r,s}
      \end{bmatrix}_{r, s \notin \{i, j\}}\\
      &= \frac{(m-1)!}{2^{\frac{m}{2}}} \sum_{1 \leq i \neq j \leq m} (-1)^{i+j+1+\Theta(i-j)} \epsilon_{i,j} \mathrm{Pf}\begin{bmatrix}
        \epsilon_{r,s}
      \end{bmatrix}_{r, s \notin \{i, j\}}\\
      &= \frac{m!}{2^{\frac{m}{2}}} \mathrm{Pf}\begin{bmatrix}
        \epsilon_{i,j}
      \end{bmatrix}_{1 \leq i, j \leq m},
    \end{align*}
    where $\Theta$ denotes  the  Heaviside step function.  We
   thus complete the proof.
  \end{proof}

 Lemma \ref{Bruijn0}  is nothing but  a slight variation  of  the de Bruijn integration  formula \cite{de}.   Indeed, if   \begin{equation*}
     \tilde{\epsilon}_{i, j} := \int_{\Omega^{2}} \epsilon(x, y) 
        \phi_{i}(x)   \phi_{j}(y)
    d\nu(x) d\nu(y)<\infty
    \end{equation*}
    are well defined  for any $1 \leq i, j \leq m$, 
  then  \eqref{generaleq}  is  de Bruijn integration  formula since   $\epsilon$ is antisymmetric and thus  $ \epsilon_{i, j} =2\tilde{\epsilon}_{i, j}$; see   e.g.   
  \cite[Proposition 6.3.5]{forrester10}.  However, in our case (cf.  \eqref{2.6} and \eqref{schur})   the above  integrability  assumption is  not  met,  so  it is necessary  for us to have recourse  to  Lemma \ref{Bruijn0}. 
   Indeed,  the formula \eqref{generaleq}  has been   used   in a particular situation by Kanazawa and Kieburg \cite{KK2} but without a detailed proof.

We conclude this section with a remark about the limit of $\tau \to 0$,   following the method of \cite[Sect. 6.2]{KK2}. A different limit of $\tau \to 1$,  relevant  to singular values of GUE,  is much more technical; see   \cite{Edelman,KK2}.
  Since  each entry from   the Pfaffian in \eqref{eqn3} tends to $0$ as  $\tau \to 0$, we need to expand it  into a series in $\eta_{-}$.
 Recalling   the  series expansion of the modified Bessel function $I_0$, we have
  \begin{align*}
    \Epsilon(\eta_{-}\lambda_{j}, \eta_{-}\lambda_{k}) &= \sum_{r, s = 0}^{\infty} \beta_{r,s}(\eta_{-}\lambda_{j})^{r} (\eta_{-}\lambda_{k})^{s}, \\
    g_{a}(\lambda_{j}) &= \eta_{-}^{a-1} \sum_{r=0}^{\infty} \beta_{r,s}(\eta_{-}\lambda_{j})^{r},
  \end{align*}
  where
  \begin{align*}
    \beta_{r,s} = \int_{\mathbb{R}^{2}} dx dy \frac{x-y}{x+y} e^{-\frac{1}{2}(x^{2}+y^{2})} \frac{1}{(r!s!)^{2}} (x^{2r}y^{2s} - x^{2s}y^{2r}).
  \end{align*}
  Introduce  an infinite matrix
  \begin{align*}
    B =\big[\beta_{r,s}\big]_{r,s = 0}^{\infty}=:\begin{bmatrix}
      B_{11} & B_{12} \\
      B_{21} & B_{22}
    \end{bmatrix}
  \end{align*}
  with  $B_{11}$ denoting  its left-upper sub-matrix of order $N+n$, and  a  Vandermonde matrix
  \begin{align*}
    V =\big[(\eta_{-}\lambda_{r})^{s-1} \big]_{r =1,\ldots,N; s= 1, \ldots, \infty}=: [V_{1}, V_{2}]   \end{align*}
  where $V_{1}$ is the  subblock consisting of first $N+n$ columns. Let $$D = \mathrm{diag} \Big(1, \eta_{-}, \ldots, ((n-1)!)^{2}\eta_{-}^{n-1}\Big),$$
  we  then obtain
  \begin{align*}
    &\begin{bmatrix}
        \begin{array}{c|c}
          \Epsilon(\eta_{-}\lambda_{j}, \eta_{-}\lambda_{k}) & g_{b}(\lambda_{j})\\
          \hline
          -g_{a}(\lambda_{k}) & \alpha_{a,b}
        \end{array}
      \end{bmatrix}_{\substack{j,k = 1, \ldots, N\\a, b = 1, \ldots, n}}= \begin{bmatrix}
        V_{1} & V_{2} \\
        D & 0
      \end{bmatrix} \begin{bmatrix}
        B_{11} & B_{12} \\
        B_{21} & B_{22}
      \end{bmatrix} \begin{bmatrix}
        V_{1}^{t} & D \\
        V_{2}^{t} & 0
      \end{bmatrix}\\
      &=
        \begin{bmatrix}
          V_{1}\\
          D
        \end{bmatrix} B_{11} \begin{bmatrix}
          V_{1}^{t} & D
        \end{bmatrix}  + \begin{bmatrix}
          V_{2}\\
          0
        \end{bmatrix} B_{21} \begin{bmatrix}
          V_{1}^{t} & D
        \end{bmatrix}
        +\begin{bmatrix}
          V_{1}\\
          D
        \end{bmatrix} B_{12} \begin{bmatrix}
          V_{2}^{t} & 0
        \end{bmatrix} + \begin{bmatrix}
          V_{2}\\
          0
        \end{bmatrix} B_{22} \begin{bmatrix}
          V_{2}^{t} & 0
        \end{bmatrix}.
          \end{align*}

 Recalling the definition of $\eta_{-}$  in \eqref{etapm}, we arrive at
  \begin{align*}
    &\lim_{\tau \to 0} \frac{1}{\eta_{-}^{M(M-1)/2}} \mathrm{Pf} \begin{bmatrix}
        \begin{array}{c|c}
          \Epsilon(\eta_{-}\lambda_{j}, \eta_{-}\lambda_{k}) & g_{b}(\lambda_{j})\\
          \hline
          -g_{a}(\lambda_{k}) & \alpha_{a,b}
        \end{array}
      \end{bmatrix}_{j,k = 1, \ldots, N;a, b = 1, \ldots, n}\\
      &= \mathrm{Pf} \big(\widetilde{V} B_{11} \widetilde{V}^{t}\big)     = \mathrm{det} \big(\widetilde{V}\big) \mathrm{Pf}(B_{11}) \propto  \Big(\prod_{k=1}^{N} \lambda_{k}^{n}\Big) \, \Delta_N(\lambda),
  \end{align*}
  where
  \begin{align*}
   \widetilde{V} =   \begin{bmatrix}
     \begin{array}{cc}
       (\lambda_{i}^{j-1} )_{i=1, \ldots,N;j=1, \ldots, N+n} & \\
       \mathrm{diag}(1, 1!, \cdots, (n-1)!) & 0_{n\times N}
      \end{array}
\end{bmatrix}.
\end{align*}

  Together with   \eqref{eqn1}, \eqref{eqn2} and \eqref{eqn3}, we have shown that  when  $\tau \to 0$ \eqref{densityCEES} in Theorem  \ref{SVdensity}  reduces to  the eigenvalue PDF for the  complex spiked Wishart  matrix.

\section{Integrals for correlation kernels} \label{sectkernel}
  In order to derive  a compact  expression of  the correlation kernel associated with  the Pfaffian point process \eqref{densityCEES},  we need to generalize  the result of \cite[Theorem 1.1]{R00} which is inspired by the results of  \cite[Sect. 7]{tw1998}.  See  \cite[Proposition 3.1]{ACLS} for a determinantal analog of the following proposition.
  \begin{prop}\label{rains}
    Let $N>0$ and  $n$ be nonnegative integers such that $M = N + n$ is even.   For  a measure space  $(\Omega, \nu)$, let $\phi_{1}, \ldots, \phi_{N}$, $h_{1}, \ldots, h_{n}$ be functions from $\Omega$ to $\mathbb{R}$, and let $\epsilon$ be an antisymmetric function from $\Omega \times \Omega$ to $\mathbb{R}$.  For any probability distribution on $\Omega^N$ with  density
    \begin{align}
      f(x_{1}, \ldots, x_{N}) \propto \det\begin{bmatrix}
        \phi_{j}(x_{k})
      \end{bmatrix}_{1 \leq j, k \leq N} \mathrm{Pf}\begin{bmatrix}
        \begin{array}{c|c}
          \epsilon(x_{j}, x_{k}) & h_{b}(x_{j})\\
          \hline
          -h_{a}(x_{k}) & \alpha_{a,b}
        \end{array}
      \end{bmatrix}_{\substack{1 \leq j, k \leq N\\1 \leq a, b \leq n}}, \label{densityDPf}
    \end{align}
    where $[\alpha_{a,b}]_{1 \leq a, b \leq n}$ is an antisymmetric matrix, introduce
    \begin{align}\label{intoper00}
      \epsilon\phi_{k}(x) = \begin{cases}
        \int_{\Omega} \epsilon(x, z) \phi_{k}(z) d\nu(z), & k = 1, \ldots, N,\\
        h_{k-N}(x), & k > N,
      \end{cases}
    \end{align}
    and the Gram matrix $G = [g_{jk}]_{1 \leq j, k \leq M}$ with
    \begin{align}
      g_{jk} = \begin{cases}
        \int_{\Omega} \phi_{j}(x)\epsilon\phi_{k}(x) d\nu(x), & 1 \leq j \leq N, 1 \leq k \leq M,\\
        -\int_{\Omega} \epsilon\phi_{j}(x)\phi_{k}(x) d\nu(x), & N < j \leq M, 1 \leq k \leq N,\\
        \alpha_{jk}, & N < j, k \leq M.
      \end{cases} \label{mmatrix}
    \end{align}
    Assume that $G$ is invertible and  let  $C = G^{-1}$, then  \eqref{densityDPf}
 forms a Pfaffian point process with  correlation kernel
    \begin{align}
      K_{N}(x, y) = \begin{bmatrix}
        DS_{N}(x, y) & S_{N}(x, y) \\
        -S_{N}(y, x) & IS_{N}(x, y)
      \end{bmatrix}, \label{kernelDPf}
    \end{align}
    where
    \begin{align}
      DS_{N}(x, y) &= \sum_{j, k = 1}^{N} \phi_{j}(x) c_{kj} \phi_{k}(y), \label{defkernel1}\\
      S_{N}(x, y) &= \sum_{j=1}^{N} \phi_{j}(x) \sum_{k=1}^{N+n} c_{kj} \epsilon\phi_{k}(y),\label{defkernel2}\\
      IS_{N}(x, y) &= -\epsilon(x, y) + \sum_{j,k=1}^{N+n} \epsilon\phi_{j}(x) c_{kj} \epsilon\phi_{k}(y).\label{defkernel3}
    \end{align}
  \end{prop}
  \begin{proof}
   In the     matrix $[K_{N}(x_{j}, x_{k})]_{1 \leq j, k \leq N}$,  interchange the $(2k-1)$-th and $(k-1)$-th  columns for any $k$, and also interchange the corresponding rows. The  determinant  is unchanged.   With $C=[c_{i,j}]_{i,j=1, \ldots, N+n}$,  and in $2\times 2$ block form $C=[C_{i,j}]_{i,j=1,2}$ with $C_{1,1}$ of size $N\times N$,   we therefore have   \begin{align*}
      &\det\begin{bmatrix}
       K_{N}(x_{j}, x_{k})
     \end{bmatrix}_{\substack{1 \leq j, k \leq N}}\\
     &=\det\begin{bmatrix}
       \begin{array}{c|c}
          \sum_{r, s = 1}^{N} \phi_{r}(x_{j}) c_{sr} \phi_{s}(x_{k}) & \sum_{r=1}^{N} \sum_{s=1}^{N+n} \phi_{r}(x_{j}) c_{sr} \epsilon\phi_{s}(x_{k})\\
        \hline
         \sum_{r=1}^{N+n} \sum_{s=1}^{N} c_{sr} \epsilon\phi_{r}(x_{j}) \phi_{s}(x_{k}) & - \epsilon(x_{j},x_{k}) + \sum_{r,s=1}^{N+n} \epsilon\phi_{r}(x_{j}) c_{sr} \epsilon\phi_{s}(x_{k})
       \end{array}
      \end{bmatrix}\\
     &=\det\begin{bmatrix}
       \begin{array}{c|c}
       \Phi C_{1,1}^{t} \Phi^{t} & \Phi C_{1,1}^{t} \epsilon\Phi^{t} + \Phi C_{2,1}^{t} H^{t}\\
         \hline
           \epsilon\Phi C_{1,1}^{t} \Phi^{t} + H C_{1,2}^{t} \Phi^{t} & - E + \epsilon\Phi C_{1,1}^{t} \epsilon\Phi^{t} + \epsilon\Phi C_{1,2}^{t} H^{t} + H C_{2,1}^{t} \epsilon\Phi^{t} + H C_{2,2}^{t} H^{t}
        \end{array}
     \end{bmatrix}
    \end{align*}
    where $$\Phi = [\phi_{k}(x_{j})]_{1\leq j, k \leq N}, \quad \epsilon\Phi = [\epsilon\phi_{k}(x_{j})], \quad H = [h_{k}(x_{j})]_{1 \leq j \leq N, 1 \leq k \leq n},$$ and $E= [\epsilon(x_{j}, x_{k})]$. Notice that, for any given $x_{1}, \ldots, x_{N}$ such that $\Phi$ is invertible, $\epsilon\Phi$ is a linear transformation of $\Phi$. Thus, we obtain
    \begin{align}
      \det\begin{bmatrix}
        K_{N}(x_{j}, x_{k})
      \end{bmatrix}
      &=\det\begin{bmatrix}
        \begin{array}{c|c}
          \Phi C_{1,1}^{t} \Phi^{t} & \Phi C_{2,1}^{t} H^{t}\\
          \hline
           H C_{1,2}^{t} \Phi^{t} & -E + H C_{2,2}^{t} H^{t}
        \end{array}
      \end{bmatrix} \nonumber\\
      &=\det\left\{\begin{bmatrix}
        \Phi & \\
         & I_{N}
      \end{bmatrix}\begin{bmatrix}
        \begin{array}{c|c}
          {C_{1,1}} & C_{1,2} H^{t}\\
          \hline
           H C_{2,1} & E + H C_{2,2} H^{t}
        \end{array}
      \end{bmatrix}\begin{bmatrix}
        \Phi^{t} & 0\\
        0 & I_{N}
      \end{bmatrix}\right\}\nonumber\\
      &=(\det\Phi)^{2} \lim_{\delta \to 0}\det\begin{bmatrix}
        \delta I_{N} + C_{1,1} & C_{1,2} H^{t}\\
        H C_{2,1} & E + H (\delta I_{n} + C_{2,2}) H^{t}
      \end{bmatrix}. \nonumber
    \end{align}
        Here we can always let  $\delta$  change in such a way  that  all involved inverses below  exist.

To continue these calculations, we will make use of  some  well-known  facts about the   Schur complement of a $2\times 2$ block matrix. 
To be precisely, let $A_{1} \in \mathbb{R}^{N \times N}$, $A_{2}, A_{3}^{t} \in \mathbb{R}^{N \times n}$ and  $A_{4} \in \mathbb{R}^{n\times n}$, and assume that  the inverses exist. Then
      \begin{align}
        \det\begin{bmatrix}
          A_{1} & A_{2} \\
          A_{3} & A_{4}
        \end{bmatrix}
        = \det(A_{1})\det(A_{4} - A_{3} A_{1}^{-1} A_{2}) = \det(A_{4})\det(A_{1} - A_{2} A_{4}^{-1} A_{3}) , \label{Schur1}
      \end{align}
      and
      \begin{align}
        \begin{bmatrix}
          A_{1} & A_{2} \\
          A_{3} & A_{4}
        \end{bmatrix}^{-1}
        =\begin{bmatrix}
        A_{1}^{-1} + A_{1}^{-1} A_{2} (A_{4} - A_{3}A_{1}^{-1}A_{2})^{-1} A_{3} A_{1}^{-1} & -A_{1}^{-1}A_{2}(A_{4} - A_{3}A_{1}^{-1}A_{2})^{-1}\\
          -(A_{4} - A_{3}A_{1}^{-1}A_{2})^{-1}A_{3}A_{1}^{-1}& (A_{4} - A_{3}A_{1}^{-1}A_{2})^{-1}
        \end{bmatrix}. \label{Schur2}
      \end{align}
By use of the formula  \eqref{Schur1} we get
    \begin{align}
      &\det\begin{bmatrix}
        K_{N}(x_{j}, x_{k})
      \end{bmatrix}\nonumber\\
      &= (\det\Phi)^{2}  \lim_{\delta \to 0} \det[\delta I_{N} + C_{1,1}) \det\big[E + H ((\delta I_{n} + C_{2,2}) - C_{2,1} (\delta I_{N} + C_{1,1})^{-1} C_{1,2}) H^{t}\big] \nonumber\\
      &= (\det\Phi)^{2} \lim_{\delta \to 0} \frac{\det(\delta I_{N} + C_{1,1})}{\det((\delta I_{n} + C_{2,2}) - C_{2,1} (\delta I_{N} + C_{1,1})^{-1} C_{1,2})} \nonumber \\
      &\quad \times \det\begin{bmatrix}
        E & H\\
        -H^{t} & \big((\delta I_{n} + C_{2,2}) - C_{2,1} (\delta I_{N} + C_{1,1})^{-1} C_{1,2}\big)^{-1}
      \end{bmatrix}. \label{keq1}
    \end{align}
  On the other hand,  we   easily see  from  the formula  \eqref{Schur2} that
    \begin{align*}
      G &=\lim_{\delta \to 0} \begin{bmatrix}
        {\delta I_{N} + C_{1,1}} & {C_{1,2}}\\
        {C_{2,1}} & \delta I_{n} + {C_{2,2}}
      \end{bmatrix}^{-1}\\
      &= \begin{bmatrix}
        G_{11} & G_{12}\\
        G_{21} & \lim_{\delta \to 0}\big((\delta I_{n} + {C_{2,2}}) - {C_{2,1}}(\delta I_{N} + C_{1,1})^{-1}{C_{1,2}}\big)^{-1}
      \end{bmatrix}.
    \end{align*}
   Compare the right-lower block of both sides and   we  arrive at
    \begin{equation}\lim_{\delta \to 0}\Big((\delta I_{n} + {C_{2,2}}) - {C_{2,1}}(\delta I_{N} + C_{1,1})^{-1}{C_{1,2}}\Big)^{-1}=\alpha. \label{keq2} \end{equation}

   Combination of \eqref{keq1} and  \eqref{keq2} gives us
    \begin{align*}
      \mathrm{Pf}\begin{bmatrix}
        K_{N}(x_{j}, x_{k})
      \end{bmatrix} = \sqrt{\det\begin{bmatrix}
        K_{N}(x_{j}, x_{k})
      \end{bmatrix}} \propto \det(\Phi) \, \mathrm{Pf}\!\begin{bmatrix}
        E & H\\
        -H^{t} & \alpha
      \end{bmatrix}.
    \end{align*}
    This means that the Pfaffian $\mathrm{Pf}[K_{N}(x_{j}, x_{k})]_{1 \leq j, k \leq N}$ exactly  defines the same probability distribution as in \eqref{densityDPf}.

    Next, we need to verify that the kernel $K_{N}(x, y)$ indeed defines a Pfaffian point process, meaning that the $l$-th order Pfaffian with elements $K_N(x_j,x_k)$ corresponds
to the $l$-th correlation function. It is sufficient to establish the integral identity
    \begin{align*}
      \int_{\Omega} \mathrm{Pf}[K_{N}(x_{j}, x_{k})]_{ j, k=1}^{l} d\nu(x_{l})
      &= (N - l +1) \mathrm{Pf}[K_{N}(x_{j}, x_{k})]_{ j, k=1}^{l-1}.
    \end{align*}
Actually,  this can be checked  by expanding $\mathrm{Pf}[K_{N}(x_{j}, x_{k})]_{1 \leq j, k \leq l}$ along the bottom two rows,  with the help of the following  integrals:
    \begin{align}
      \int_{\Omega} - S_{N}(x_{l}, x_{l}) d\nu(x_{l})
      &= -N,\\
      \int_{\Omega} - S_{N}(x_{k}, x_{l}) DS_{N}(x_{l}, x_{j}) d\nu(x_{l})
      &= - DS_{N}(x_{k}, x_{j}),\\
      \int_{\Omega} - S_{N}(x_{k}, x_{l}) S_{N}(x_{l}, x_{j}) d\nu(x_{l})
      &= - S_{N}(x_{k}, x_{j}),\\
      \int_{\Omega} IS_{N}(x_{l}, x_{k}) DS_{N}(x_{l}, x_{j}) d\nu(x_{l})
      &= 0,\\
      \int_{\Omega} IS_{N}(x_{l}, x_{k}) S_{N}(x_{l}, x_{j}) d\nu(x_{l})
      &=0.
    \end{align}
    These integrals can be easily verified according to the definitions  \eqref{defkernel1}-\eqref{defkernel3}.
  \end{proof}

The double summations \eqref{defkernel1}--\eqref{defkernel3} can in fact be written as double contour integrals, and
 it is the latter which are ideally suited to asymptotic analysis.
   Our subsequent goal is to derive  double contour integrals for correlation kernel associated with  the singular value  PDF \eqref{densityCEES}.  For this, we find that it is convenient   to  introduce  an extended ensemble. Let
    \begin{align}
      \varrho_{k} = \frac{\tau}{\sigma_{k}  + \tau}, \qquad  k=1, 2, \ldots, N, \label{rhodef}
    \end{align}
    and introduce $M-N$ extended parameters $\varrho_{N+1}, \ldots, \varrho_{M}\in (-\infty, 1/2)$.  It is worth stressing that  that $\varrho_{k}<1/2$  for all $1\leq k\leq M$.
  \begin{thm}\label{pfstruct}
    With the notations in Proposition \ref{rains}, let $\Omega=(0,\infty)$ and $\nu$  be the Lebesgue measure. Let  $\epsilon(u, v) = \Epsilon(\eta_{-}u, \eta_{-}v)$,
    $\phi_{j}(u) = e^{-(\eta_{+}\sigma_{j} + \eta_{-}) u}$ for $ j=1, \ldots, N$ and for  $1\leq a,b\leq n$
     \begin{align}
      h_{a}(u) &= \eta_{-}^{a-1} \int_{\mathbb{R}^{2}} dx dy\, e^{-\frac{1}{2}(x^{2}+y^{2})} \frac{x-y}{x+y} \left (I_{0}(2 y \sqrt{\eta_{-} u}) e^{\varrho_{a+N}x^{2}} - I_{0}(2 x \sqrt{\eta_{-} u}) e^{ \varrho_{a+N} y^{2}}\right), \nonumber\\
      \alpha_{a, b} &= \eta_{-}^{a+b-2} \int_{\mathbb{R}^{2}} dx dy\, e^{-\frac{1}{2}(x^{2}+y^{2})} \frac{x-y}{x+y} \left(e^{\varrho_{b+N}y^{2}} e^{\varrho_{a+N}x^{2}} - e^{\varrho_{b+N}x^{2}} e^{\varrho_{a+N}y^{2}}\right). \nonumber
    \end{align}
    If $\varrho_1, \ldots, \varrho_{N+n} \in (0,1/2) $ are pairwise distinct,  then
           \begin{align}
      S_{N}(u, v) &= \int_{0}^{\infty} dt\, \Epsilon(\eta_{-}v, \eta_{-}t) \widetilde{DS}_{N}(u, t),\\
      IS_{N}(u, v) &= -\Epsilon(\eta_{-}u, \eta_{-}v) +\int_{0}^{\infty} \int_{0}^{\infty}ds dt \,  \Epsilon(\eta_{-}v,\eta_{-}t) \widetilde{\widetilde{DS}}_{N}(s, t)  \Epsilon(\eta_{-}u, \eta_{-}s),
    \end{align}
and
      \begin{align}\label{extendedkernel11}
        DS_{N}(u, v) &= \frac{ \eta_{-}^2}{4\pi} \int_{\mathcal{C}_{\{\varrho_1, \ldots, \varrho_{N}\}}} \frac{dw}{2\pi i}  \int_{\mathcal{C}_{\{1-\varrho_1, \ldots, 1-\varrho_{N}\}}} \frac{dz}{2\pi i}  e^{-\frac{\eta_{-} u}{w}- \frac{\eta_{-} v}{1-z}}  \frac{1}{\sqrt{(2z-1)(1-2w)}} \nonumber \\
        & \quad \times
        \frac{1-z-w}{w-z} \frac{1}{(1-z)w}
        \prod_{k = 1}^{M} \frac{z-\varrho_{k}}{w - \varrho_{k}} \frac{1 - w - \varrho_{k}}{1 - z - \varrho_{k}},
    \end{align}
  while     $\widetilde{DS}_{N}=DS_{N}$ but with the contour    $\mathcal{C}_{\{1-\varrho_{1}, \ldots, 1-\varrho_{N}\}} \mapsto \mathcal{C}_{\{1-\varrho_{1}, \ldots, 1-\varrho_{N+n}\}}$ and      $\widetilde{\widetilde{DS}}_{N}=\widetilde{DS}_{N}$ but with the contour   $\mathcal{C}_{\{\varrho_{1}, \ldots, \varrho_{N}\}} \mapsto \mathcal{C}_{\{\varrho_{1}, \ldots, \varrho_{N+n}\}}$.
  \end{thm}
  \begin{proof} First, we  need to calculate  the Gram matrix $G = [g_{jk}]_{1 \leq j, k \leq M}$ defined  in \eqref{mmatrix}.

   When $1 \leq j,k \leq N$,
     using  the simple fact    (cf. \cite[10.43.23]{OOL})  \begin{align}
      \int_{0}^{\infty}  I_{0}(\beta \sqrt{t}) e^{-\alpha t} dt = \frac{1}{ \alpha} e^{\frac{\beta^{2}}{4\alpha}}, \quad \Re\{\alpha\}>0, \label{intBessel01}
\end{align}
    we see from Fubini's theorem that
    \begin{align}
      g_{jk} &= \int_{\mathbb{R}_{+}^{2}} du dv\,  \phi_{j}(u) \phi_{k}(v) \Epsilon(\eta_{-}u, \eta_{-}v) \nonumber \\
      &= \int_{\mathbb{R}^{2}} dx dy \,e^{-\frac{1}{2} x^{2} - \frac{1}{2} y^{2}} \frac{x-y}{x+y} \int_{\mathbb{R}_{+}^{2}} du dv\, e^{-(\eta_{+}\sigma_{j}+\eta_{-})u} e^{-(\eta_{+}\sigma_{k}+\eta_{-})v} \nonumber\\
      &\quad\times \Big(I_{0}(2 x \sqrt{\eta_{-} u}) I_{0}(2 y \sqrt{\eta_{-} v}) - I_{0}(2 y \sqrt{\eta_{-} u})I_{0}(2 x \sqrt{\eta_{-} v})\Big)  \nonumber\\
      &= \frac{1}{(\eta_{+}\sigma_{j}+\eta_{-})(\eta_{+}\sigma_{k}+\eta_{-})} \int_{\mathbb{R}^{2}} dx dy\, e^{-\frac{1}{2} x^{2} - \frac{1}{2} y^{2}} \frac{x-y}{x+y} \big(e^{\varrho_{j} x^{2} + \varrho_{k} y^{2}} - e^{\varrho_{j} y^{2} + \varrho_{k} x^{2}}\big), \nonumber
    \end{align}
    where $\varrho_{k}$'s are   given by \eqref{rhodef}.     Applying Lemma \ref{intgauss} to the above integral, we obtain
    \begin{align}
      g_{jk} = \frac{4\pi }{\eta_{-}^{2}} \frac{ \varrho_{j}\varrho_{k}(\varrho_{j} - \varrho_{k})}{(1 - \varrho_{j} - \varrho_{k}) \sqrt{1 - 2\varrho_{j}} \sqrt{1 - 2\varrho_{k}}}, \qquad 1 \leq j, k \leq N.
    \end{align}

    Similarly, when $1 \leq j \leq N < k \leq N+n$, we obtain
    \begin{align}
      g_{jk}:&= \int_{\mathbb{R}_{+}} du\,  \phi_{j}(u) h_{k-N}(u)  \nonumber\\
      &= \eta_{-}^{k-N-1} \int_{\mathbb{R}^{2}} dx dy\, e^{-\frac{1}{2}(x^{2}+y^{2})} \frac{x-y}{x+y}  \nonumber\\
      &\quad\times \int_{\mathbb{R}_{+}} du e^{-(\eta_{+}\sigma_{j} + \eta_{-}) u} \big(I_{0}(2 x \sqrt{\eta_{-} u}) e^{\varrho_{k}y^{2}} - I_{0}(2 y \sqrt{\eta_{-} u}) e^{\varrho_{k}x^{2}}\big)\nonumber\\
      &= 4\pi \eta_{-}^{k-N-2}\frac{ \varrho_{j}(\varrho_{j} - \varrho_{k})}{(1 - \varrho_{j} - \varrho_{k}) \sqrt{1 - 2\varrho_{j}} \sqrt{1 - 2\varrho_{k}}}.
    \end{align}

  Also, when $ N< j, k \leq N+n$,  it's easy to see from Lemma \ref{intgauss} that
  \begin{align}
      g_{jk}:&={  \alpha_{j-N, k-N} =}4\pi  \eta_{-}^{j+k-2N-2} \frac{ \varrho_{j} - \varrho_{k}}{(1 - \varrho_{j} - \varrho_{k}) \sqrt{1 - 2\varrho_{j}} \sqrt{1 - 2\varrho_{k}}}.
 \end{align}
 The we can
   rewrite the Gram matrix as
    \begin{align}
      G&= 4 \pi D_{1} D_{2} \widetilde{G} D_{2} D_{1}, \label{grammatrix}
    \end{align}
    where
    \begin{align}
      D_{1} &= \mathrm{diag}\big(\varrho_1 / \eta_{-}, \ldots, \varrho_N / \eta_{-}, 1, \eta_{-}, \ldots, \eta_{-}^{n-1}\big), \label{matrixd1}\\
      D_{2} &= \mathrm{diag}\big((1 - 2\varrho_{1})^{-1/2}, \ldots, (1 - 2\varrho_{N+n})^{-1/2}\big),\label{matrixd2}
    \end{align}
and  \begin{align}
      \widetilde{G} &= [\tilde{g}_{jk}]:= \left[\frac{\varrho_{j} - \varrho_{k}}{1 - \varrho_{j} - \varrho_{k}}\right]_{1 \leq j, k \leq N+n}. \label{matrixG}
    \end{align}

    Let $\widetilde{C} = [\tilde{c}_{jk}]$ be the inverse of matrix $\widetilde{G}$.  We claim that there exists a system of identical relations among  entries of $\widetilde{C} $. For this, introduce some rational  functions
    \begin{align}
      f_{j}(z) = \sum_{l=1}^{N+n} \frac{z - \varrho_{l}}{1 - z - \varrho_{l}} \tilde{c}_{lj}, \qquad j = 1, \ldots, N+n. \label{frelation1}
    \end{align}
   For fixed $j$, it immediately  follows from  the orthogonal relations
    \begin{align}
      f_{j}(\varrho_{k})
      = \sum_{l=1}^{N+n} \tilde{g}_{kl} \tilde{c}_{lj} = \delta_{kj}
    \end{align}
    that $f_{j}(z)$ has $N+n-1$ zeros  at $\varrho_{1}, \ldots, \varrho_{j-1},  \varrho_{j+1}, \ldots, \varrho_{N+n}$. On the other hand, since    $\widetilde{C}$ is anti-symmetric and thus  $\tilde{c}_{jj}=0$,  it has  $N+n-1$ poles at $1-\varrho_{1}, \ldots, 1-\varrho_{j-1},  1-\varrho_{j+1}, \ldots, 1-\varrho_{N+n}$. Since $\varrho_{k}$'s are pairwise distinct, recalling that a rational  function is  uniquely  determined by its  zeros, poles, and the value of the function at one extra point,   we get
    \begin{align}
      f_{j}(z)
      &= \prod_{k \neq j, k = 1}^{N+n} \frac{z-\varrho_{k}}{1 - z - \varrho_{k}} \frac{1 - \varrho_{j} - \varrho_{k}}{\varrho_{j} - \varrho_{k}}, \qquad j = 1, \ldots, N+n.  \label{frelation2}
    \end{align}

    In order to obtain explicit form of correlation kernels as in Proposition \ref{rains},
    recalling the definition \eqref{intoper00} and $C = [c_{jk}]$ being the inverse of  $G$, and rewriting \eqref{matrixd1} and \eqref{matrixd2} as $D_l=\diag(d_{l1}, \ldots,d_{lM})$ where $l=1,2$, for each $j$
  we must  evaluate          \begin{align*}
      \sum_{k=1}^{N+n} c_{kj} \epsilon\phi_{k}(v)
      = \int_{0}^{\infty} dt \,\Epsilon(\eta_{-}v, \eta_{-}t)
       \frac{1}{4\pi d_{1j} d_{2j}}   \sum_{k=1}^{N+n} \tilde{c}_{kj} \sqrt{1-2\varrho_k} e^{-\frac{\eta_{-}}{\varrho_{k}}t} \frac{\eta_{-}}{\varrho_{k}}
    \end{align*}
    where    use  has been made of the fact that
        \begin{equation*}
    h_{a}(v)= \eta_{-}^{a}\int_{0}^{\infty} dt \,\Epsilon(\eta_{-}v, \eta_{-}t)
   e^{-\frac{\eta_{-}}{\varrho_{a+N}}t} \frac{1}{\varrho_{a+N}}.
    \end{equation*}

     By Cauchy's residue theorem, it is easy to see that for  $1\leq k\leq N+n$
     \begin{align}
      \sqrt{1 - 2\varrho_{k}} \frac{1}{\varrho_{k}} e^{-\frac{\eta_{-}}{\varrho_{k}} v}
      &= \frac{-1}{2\pi i } \int_{\mathcal{C}_{1 - \varrho_{k}}} dz \frac{z - \varrho_{k}}{1 - z - \varrho_{k}} \frac{1}{\sqrt{2z - 1}} \frac{1}{1 - z} e^{- \frac{\eta_{-}}{1 - z} v},  \label{intrepvip}
    \end{align}
    where $\mathcal{C}_{1 - \varrho_{k}}$ is an anticlockwise contour encircling $1 - \varrho_{k}$.
    Noting $$ \frac{1}{d_{1j}d_{2j}}=\frac{\eta_{-}}{\varrho_{j}}\sqrt{1-2\rho_j},\qquad j=1, \ldots, N,$$
    when $1\leq j\leq N$  we obtain
   \begin{align*}
      \sum_{k=1}^{N+n} c_{kj} \epsilon\phi_{k}(v)
      &=   \int_{0}^{\infty} dt \,\Epsilon(\eta_{-}v, \eta_{-}t)   \frac{-\eta^{2}_{-}}{4\pi }
    \frac{\sqrt{1-2\rho_j}}{\varrho_{j}}\nonumber \\
      & \quad
      \times \int_{\mathcal{C}_{\{1-\varrho_{1}, \ldots, 1-\varrho_{N+n}\}}} \frac{dz}{2\pi i}\,  e^{- \frac{\eta_{-}}{1 - z} t} \frac{1}{\sqrt{2z - 1}} \frac{1}{1 - z}  \sum_{k=1}^{N+n} \tilde{c}_{kj}  \frac{z - \varrho_{k}}{1 - z - \varrho_{k}}.
    \end{align*}

    Substituting  the evaluation determined by
    \eqref{frelation1} and    \eqref{frelation2}  into the above shows
     \begin{align}
      \sum_{k=1}^{N+n} c_{kj} \epsilon\phi_{k}(v)
      &= \int_{0}^{\infty} dt \,\Epsilon(\eta_{-}v, \eta_{-}t)   \frac{-\eta^{2}_{-}}{4\pi }
    \frac{\sqrt{1-2\rho_j}}{\varrho_{j}} \int_{\mathcal{C}_{\{1-\varrho_{1}, \ldots, 1-\varrho_{N+n}\}}} \frac{dz}{2\pi i}\,  e^{- \frac{\eta_{-}}{1 - z} t}\nonumber \\
      & \quad
      \times \frac{1}{\sqrt{2z - 1}} \frac{1}{1 - z}
      \prod_{k \neq j, k = 1}^{N+n} \frac{z-\varrho_{k}}{1 - z - \varrho_{k}} \frac{1 - \varrho_{j} - \varrho_{k}}{\varrho_{j} - \varrho_{k}}. \label{Skernelsum}
    \end{align}
Recalling  \eqref{defkernel2}, and making use of a contour integral (integration variable $w$)  to perform the symmetry
 of $j$,  we arrive at
 \begin{equation*}S_{N}(u, v) = \int_{0}^{\infty} dt\, \Epsilon(\eta_{-}v, \eta_{-}t) \widetilde{DS}_{N}(u, t), \end{equation*}
    with
      \begin{align}\label{extendedkernel11'}
        \widetilde{DS}_{N}(u, v) &= \frac{ \eta_{-}^2}{4\pi} \int_{\mathcal{C}_{\{\varrho_{1}, \ldots, \varrho_{N}\}}} \frac{dw}{2\pi i}\int_{\mathcal{C}_{\{1-\varrho_{1}, \ldots, 1-\varrho_{N+n}\}}} \frac{dz}{2\pi i}  e^{-\frac{\eta_{-} u}{w}- \frac{\eta_{-} v}{1-z}}  \frac{1}{\sqrt{(2z-1)(1-2w)}} \nonumber \\
        &\quad\times \frac{1-z-w}{w-z} \frac{1}{(1-z)w}
        \prod_{k = 1}^{N+n} \frac{z-\varrho_{k}}{w - \varrho_{k}} \frac{1 - w - \varrho_{k}}{1 - z - \varrho_{k}}.
    \end{align}

    To find an integral representation of  $DS_{N}(u, v)$,  note
     \begin{align}
     0
      &= \frac{-1}{2\pi i } \int_{ \mathcal{C}_{\{1-\varrho_{1}, \ldots, 1-\varrho_{N}\}}} dz \frac{z - \varrho_{k}}{1 - z - \varrho_{k}} \frac{1}{\sqrt{2z - 1}} \frac{1}{1 - z} e^{- \frac{\eta_{-}}{1 - z} v},  \quad k>N. \label{intrepvip2}
    \end{align}
    Combining \eqref{intrepvip} and using  similar argument as in the derivation of $S_{N}(u, v)$, we will see that $DS_{N}(u, v)$ is just  equal to $\widetilde{DS}_{N}(u, v)$ except that  the $z$-contour $\mathcal{C}_{\{1-\varrho_{1}, \ldots, 1-\varrho_{N+n}\}}$ is changed to $\mathcal{C}_{\{1-\varrho_{1}, \ldots, 1-\varrho_{N}\}}$. In fact,
    \begin{align}
      DS_{N}(u, v)
      &=  \frac{ \eta_{-}^2}{4\pi} \sum_{j=1}^{N} \sqrt{1 - 2\varrho_{j}} \frac{1}{\varrho_{j}} e^{-\frac{\eta_{-}}{\varrho_{j}} u} \sum_{k=1}^{N} \tilde{c}_{kj} \sqrt{1 - 2\varrho_{k}} \frac{1}{\varrho_{k}} e^{-\frac{\eta_{-}}{\varrho_{k}} v} \nonumber\\
      &=  \frac{ \eta_{-}^2}{4\pi}  \sum_{j=1}^{N} \sqrt{1 - 2\varrho_{j}} \frac{1}{\varrho_{j}} e^{-\frac{\eta_{-}}{\varrho_{j}} u} \nonumber\\
      &\quad\times \int_{\mathcal{C}_{\{1-\varrho_{1}, \ldots, 1-\varrho_{N}\}}} \frac{-dz}{2 \pi i} e^{- \frac{\eta_{-}}{1 - z} v} \frac{1}{\sqrt{2z - 1}} \frac{1}{1 - z}  \sum_{k=1}^{N+n} \tilde{c}_{kj}  \frac{z - \varrho_{k}}{1 - z - \varrho_{k}} \nonumber
    \end{align}
from which the integral representation of $DS_{N}(u, v)$ immediately  follows via Cauchy's residue theorem.

   Finally, we  turn to the representation of  $IS_{N}(u, v)$. Similar to the computation of $S_{N}(u, v)$, we obtain
   \begin{align*}
     \sum_{j=1}^{N+n} \epsilon\phi_{j}(u) \sum_{k=1}^{N+n} c_{kj} \epsilon\phi_{k}(v)
     &=  \frac{ \eta_{-}^2}{4\pi} \int_{0}^{\infty} ds \int_{0}^{\infty} dt \Epsilon(\eta_{-}u, \eta_{-}s) \Epsilon(\eta_{-}v, \eta_{-}t) \nonumber\\
     &\quad\times\sum_{j}^{N+n} \frac{1}{\varrho_{j}} \sqrt{1 - 2\varrho_{j}} e^{-\frac{\eta_{-}}{\varrho_{j}} s} \sum_{j}^{N+n} \tilde{c}_{kj} \frac{1}{\varrho_{k}} \sqrt{1 - \varrho_{k}} e^{-\frac{\eta_{-}}{\varrho_{k}} t}\nonumber\\
     &= \int_{0}^{\infty} \int_{0}^{\infty} \Epsilon(\eta_{-}u, \eta_{-}s) \widetilde{ \widetilde{DS}}_{N}(s, t) \Epsilon(\eta_{-}v, \eta_{-}t) ds dt
   \end{align*}
  with a defined contour  $\widetilde{\widetilde{DS}}_{N}(u, v)$.
  \end{proof}

   \begin{rem}\label{vipremk}
  When all $\varrho_{N+1}, \ldots, \varrho_{N+n} \rightarrow 0$   in  the integral representation  of  $DS_{N}(u,v)$  \eqref{extendedkernel11},  we can choose contours of $DS_{N}(u,v)$ such that   $\mathcal{C}_{\{\varrho_{1}, \ldots, \varrho_{N}\}}$   encircles $\varrho_{1},  \ldots, \varrho_{N}$ but not $0$ while
$\mathcal{C}_{\{1-\varrho_{1}, \ldots, 1-\varrho_{N}\}}$  encircles  $1-\varrho_{1},  \ldots, 1-\varrho_{N}$  but not $1$. That is,
     \begin{align*}
        DS_{N}(u, v) &=   \frac{ \eta_{-}^2}{4\pi}   \int_{ \mathcal{C}_{\{\varrho_{1}, \ldots, \varrho_{N}\}}} \frac{dw}{2\pi i} \int_{ \mathcal{C}_{\{1-\varrho_{1}, \ldots, 1-\varrho_{N}\}}} \frac{dz}{2\pi i} e^{-\frac{\eta_{-} u}{w}- \frac{\eta_{-} v}{1-z}}  \frac{1}{\sqrt{(2z-1)(1-2w)}} \nonumber \\
        &\times
        \frac{1-z-w}{w-z} \frac{1}{(1-z)w}    \Big(\frac{z}{w} \frac{1 - w}{1 - z}\Big)^{n}
        \prod_{k = 1}^{N} \frac{z-\varrho_{k}}{w - \varrho_{k}} \frac{1 - w - \varrho_{k}}{1 - z - \varrho_{k}}.
    \end{align*}
  The reason  is that  the LHS of \eqref{intrepvip} goes to zero and we can choose the $z$-contour on the RHS not  encircling 1. Similarly, we can choose the $w$-contour that does not contain 0  by noting   the factor on the RHS of \eqref{Skernelsum}.

  \end{rem}

  \begin{lem}\label{intgauss}
    For two  complex  numbers $\alpha, \beta$  such that $\Re{\alpha}<1/2$ and $\Re{\beta} < 1/2$, we have
    \begin{align}
      &\int_{\mathbb{R}^{2}} dx dy\, e^{-\frac{1}{2} x^{2} - \frac{1}{2} y^{2}} \frac{x-y}{x+y} \Big(e^{\alpha x^{2} + \beta y^{2}} - e^{\alpha y^{2} + \beta x^{2}}\Big)
      = \frac{4\pi(\alpha - \beta)}{(1 - \alpha - \beta) \sqrt{1 - 2\alpha} \sqrt{1 - 2\beta} 
      }.\label{intgauss0}
    \end{align}
  \end{lem}

  \begin{proof}
    After change of variables $x=u+v, y=u-v$, the LHS of \eqref{intgauss0} becomes
    \begin{align*}
      2\int_{-\infty}^{\infty} du\, \frac{1}{u} e^{-(1-\alpha-\beta)u^{2}} \Big(\int_{-\infty}^{\infty} dv\, ve^{-(1-\alpha-\beta)v^{2}}\big(e^{2(\alpha-\beta)vu}  -  e^{- 2(\alpha-\beta)vu}\big)\Big).
    \end{align*}
   Integrate out the variable $v$   with the help of  the expectation of a Gaussian random  variable. Then the LHS of \eqref{intgauss0} further reduces to      \begin{align*}
    \frac{4\sqrt{\pi}(\alpha-\beta)}{(\sqrt{1-\alpha-\beta})^3} \int_{-\infty}^{\infty} du\, e^{-  \frac{(1-2\alpha)(1-2\beta)}{1-\alpha-\beta} u^{2}},
    \end{align*}
    from which the sought integral formula immediately follows.
  \end{proof}

With the above  preparation,    we are ready to complete the proof of Theorem \ref{thmkernel}.
\begin{proof}[Proof of Theorem \ref{thmkernel}]
   With   Remark  \ref{vipremk} in mind,  let $\varrho_{N+1}, \ldots, \varrho_{N+n} \to 0$ in
 Theorem \ref{pfstruct}.  Change  variables   $z\to z/(z+1)$ and $w\to 1/(w+1)$,  and notice the relation  \eqref{rhodef}.  The  expression \eqref{DSneq-2} immediately results  from \eqref{extendedkernel11}, and  the proof is  complete.   \end{proof}

  \begin{rem}
  Although we cannot obtain an explicit form for  the inverse matrix  $\widetilde{C}$   of   $\widetilde{G}$ defined in \eqref{matrixG}, we have verified that entries of $\widetilde{C}$ satisfy a family of  identical equations   \eqref{frelation2} whenever  each of $\{\varrho_{k}\}$ differs from the other. The latter may well find use in RMT  beyond the present model.
    \end{rem}

\section{Bulk and soft edge limits} \label{sectedge}
This section is devoted to asymptotics of    correlation functions in the bulk and  at the soft edge of the spectrum, and limiting distribution of the largest eigenvalue. In particular, we will complete the proofs of Theorems  \ref{edgecritthm} and \ref{largestcritthm}.

\subsection{General procedure}

In order to   investigate local statistical properties of eigenvalues, we need to rewrite double integral representations for correlation kernels.

\begin{thm} \label{newrep}
      Let
      \begin{align}
      g(z,v)&=  - z M(\frac{1}{2},1,2v)+
       (1-z^2)  \int_{0}^{v} ds\, e^{(v-s)(z+1)} M(\frac{1}{2},1,2s),\label{gfunction}
    \end{align}
 where $M(1/2,1,z)$ denotes a confluent hypergeometric function.  With the same notations as in Theorems \ref{SVdensity}  and \ref{thmkernel},  for    $M$ even we  can rewrite  $S_N$ and $IS_N$  from  \eqref{matrixkernel} as
 \begin{align}
      S_{N}(u, v)       &=\eta_{-} \int_{\mathcal{C}_{\{\tau/\sigma_{l}\}}}
       \frac{dz}{2\pi i}  \int_{\mathcal{C}_{\{\sigma_{l}/\tau\}}} \frac{dw}{2\pi i}  \frac{  g(z,\eta_{-}v)}{1-z^2}  \frac{e^{-\eta_{-} u(w+1)} } {\sqrt{w^2-1}}   \nonumber\\
      & \quad\times   \Big(\frac{w}{z}  \Big)^{M-N}
        \frac{1-zw}{w-z}         \prod_{k = 1}^{N} \frac{\tau z-\sigma_{k}}{\sigma_{k}z-\tau  }
     \frac{\sigma_{k}w-\tau}{ \tau w-\sigma_{k}}, \label{SneqII}
    \end{align}
and  \begin{align}
      IS_{N}(u, v)       &= -\Epsilon(\eta_{-}u, \eta_{-}v)+4\pi  \int_{\mathcal{C}_{\{\tau/\sigma_{l}\}}} \frac{dz}{2\pi i}
      \int_{\mathcal{C}_{\{\tau/\sigma_{l}\}}} \frac{dw}{2\pi i} \frac{  g(z,\eta_{-}v)}{1-z^2}  \frac{  g(w,\eta_{-}u)}{1-w^2}  \nonumber\\
      & \quad\times \Big(\frac{1}{zw}  \Big)^{M-N}
        \frac{z-w}{1-zw}        \prod_{k = 1}^{N} \frac{\tau z-\sigma_{k}}{\sigma_{k}z-\tau}
         \frac{\tau w-\sigma_{k}}{\sigma_{k}w-\tau},
        \label{ISneqII}
    \end{align}
    where  the contours are chosen  such that $|z|<1$ and $\Re{z}>0$ for all  $z\in  \mathcal{C}_{\{ \tau/\sigma_{l}\}}$ and
    $|z|>1$ for all
    $z\in \mathcal{C}_{\{\sigma_{l}/\tau\}}$.
In particular for   $M=N$ even,   we have
  \begin{align}
      S_{N}(u, v)       &=\eta_{-} \int_{\mathcal{C}_{\{ \pm1,\tau/\sigma_{l}\}}} \frac{dz}{2\pi i}  \int_{\mathcal{C}_{\{\sigma_{l}/\tau\}}} \frac{dw}{2\pi i}  \frac{  g(z,\eta_{-}v)}{1-z^2}  \frac{e^{-\eta_{-} u(w+1)} } {\sqrt{w^2-1}}   \nonumber\\
      & \quad\times
        \frac{1-zw}{w-z}         \prod_{k = 1}^{N} \frac{\tau z-\sigma_{k}}{\sigma_{k}z-\tau  }
     \frac{\sigma_{k}w-\tau}{ \tau w-\sigma_{k}}, \label{Sneq-2}
    \end{align}
and  \begin{align}
      IS_{N}(u, v)       &= 4\pi  \int_{\mathcal{C}_{\{\pm1, \tau/\sigma_{l}\}}} \frac{dz}{2\pi i}
      \int_{\mathcal{C}_{\{\pm 1,\tau/\sigma_{l}\}}} \frac{dw}{2\pi i} \frac{  g(z,\eta_{-}v)}{1-z^2}  \frac{  g(w,\eta_{-}u)}{1-w^2}  \nonumber\\
      & \quad\times
        \frac{z-w}{1-zw}        \prod_{k = 1}^{N} \frac{\tau z-\sigma_{k}}{\sigma_{k}z-\tau}
         \frac{\tau w-\sigma_{k}}{\sigma_{k}w-\tau},
        \label{ISneq-2}
    \end{align}
where the  two disjoint  contours are   chosen  such that $|z|>1$ for all  $z\in  \mathcal{C}_{\{\pm1, \tau/\sigma_{l}\}}$ and $z\in \mathcal{C}_{\{\sigma_{l}/\tau\}}$.
\end{thm}

 \begin{proof}
 For $S_N$,  let's introduce  a function
     \begin{align}
      h(z,w) &=\frac{\eta_{-}^{2}}{4\pi}  \frac{e^{-\eta_{-}u(w+1)} }{\sqrt{w^{2}-1}\sqrt{z^{2}-1} }
        \frac{z-w}{1-zw}         \big(zw\big)^{M-N}  \prod_{k = 1}^{N} \frac{\sigma_{k}z-\tau}{ \tau z-\sigma_{k}}
        \frac{\sigma_{k}w-\tau}{ \tau w-\sigma_{k}}.
    \end{align}
     Inserting \eqref{pdfeq1}
    and \eqref{DSneq-2}
     into \eqref{kernel12}, we see from   \eqref{intBessel01} that
    \begin{align*}
      S_{N}(u, v)
      &= \int_{\mathcal{C}_{\{\sigma_{l}/\tau\}}} \frac{dw}{2\pi i}  \int_{\mathcal{C}_{\{ \sigma_{l}/\tau\}}} \frac{dz}{2\pi i} h(w, z) \int_{\mathbb{R}^{2}} dx dy \,\frac{x-y}{x+y} e^{-\frac{1}{2}(x^{2} + y^{2})} \nonumber\\
      &\times \int_{0}^{\infty} dt\, e^{-\eta_{-}t(z+1)} \Big(I_{0}(2x\sqrt{\eta_{-}v}) I_{0}(2y\sqrt{\eta_{-}t}) - I_{0}(2y\sqrt{\eta_{-}v}) I_{0}(2x\sqrt{\eta_{-}t})\Big)\nonumber \\
      &=\int_{\mathcal{C}_{\{\sigma_{l}/\tau\}}} \frac{dw}{2\pi i}  \int_{\mathcal{C}_{\{ \sigma_{l}/\tau\}}} \frac{dz}{2\pi i} h(z, w) \int_{\mathbb{R}^{2}} dx dy \,\frac{x-y}{x+y} e^{-\frac{1}{2}(x^{2} + y^{2})} \nonumber\\
      &\quad\times  \frac{1}{\eta_{-}(z+1)}\Big(I_{0}(2x\sqrt{\eta_{-}v}) e^{\frac{y^{2}}{z+1}} - I_{0}(2y\sqrt{\eta_{-}v}))  e^{\frac{x^{2}}{z+1}}\Big)
    \end{align*}
    where we assume that $\Re\{z+1\}>2$ in order to ensure the existence of  $t$-integral.
           Using  the integral representation of the modified Bessel function $I_{0}(z)$
    \begin{align}
      I_{0}(2\sqrt{z}) = \int_{\mathcal{C}_{\{0\}}} \frac{ds}{2 \pi is} e^{sz + \frac{1}{s}}, \label{I0rep}
    \end{align}
    and Lemma \ref{intgauss}, we obtain
    \begin{align}
      S_{N}(u, v) &=  \int_{\mathcal{C}_{\{\sigma_{l}/\tau\}}} \frac{dw}{2\pi i}  \int_{\mathcal{C}_{\{ \sigma_{l}/\tau\}}} \frac{dz}{2\pi i} h(z, w) \frac{1}{\eta_{-}(z+1)} \int_{\mathcal{C}_{\{0\}}} \frac{ds}{2 \pi is}  e^{\frac{\eta_{-} v}{s}} \nonumber\\
      &\quad\times \int_{\mathbb{R}^{2}} dx dy \frac{x-y}{x+y} e^{-\frac{1}{2}(x^{2} + y^{2})} \big(e^{s  x^{2}} e^{\frac{y^{2}}{z+1}} - e^{s y^{2}} e^{\frac{x^{2}}{z+1}}\big)\nonumber\\
      &=\int_{\mathcal{C}_{\{\sigma_{l}/\tau\}}} \frac{dw}{2\pi i}  \int_{\mathcal{C}_{\{ \sigma_{l}/\tau\}}} \frac{dz}{2\pi i} h(z, w) \frac{4\pi}{\eta_{-}}  \frac{1}{\sqrt{z^{2}-1}}\,   g(\frac{1}{z},\eta_{-}v). \label{Sneq-3}
    \end{align}
    Here we assume   that $\Re(s) < 1/2$ and $g(z,v)$ is defined by       \begin{align}
      g(z,v)&=   \int_{\mathcal{C}_{\{0\}}} \frac{ds}{2 \pi is}  e^{\frac{v}{s}}
       \frac{(z+1)s-z}{1-(z+1)s } \frac{1}{\sqrt{1 - 2  s}}. \label{gdef}
    \end{align}
   Change variables $z\to 1/z$  and we immediately  obtain   \eqref{SneqII} from  \eqref{Sneq-3}, if  we can  derive another expression of $ g(z,v)$ as required.

   In fact,   we can rewrite $g\big(z,v\big)=  -I_1 +(1-z)I_2$ where
\begin{align*}
       I_1 = \int_{\mathcal{C}_{\{0\}}} \frac{ds}{2 \pi is}  e^{\frac{v}{s}}
       \frac{1}{\sqrt{1 - 2  s}},    \quad    I_2 =\int_{\mathcal{C}_{\{0\}}} \frac{ds}{2 \pi is}  e^{\frac{v}{s}}
       \frac{1}{1-(z+1)s } \frac{1}{\sqrt{1 - 2  s}}.
    \end{align*}
    Take the Taylor series of $1/\sqrt{1 - 2  s}$, integrate term by term  and we see that $I_1$ is
      a confluent hypergeometric function $M(\frac{1}{2},1,2v)$; see \cite[13.2.2]{OOL}.  Similarly,  we arrive at
    \begin{align*}
       I_2 &=\sum_{k=0}^{\infty} \sum_{l=0}^{\infty} \frac{(\frac{1}{2})_{l}}{l!} \int_{\mathcal{C}_{\{0\}}} \frac{ds}{2 \pi is}  e^{\frac{v}{s}}
      \big(s(z+1)\big)^k (2s)^l\nonumber \\
      &= \sum_{l=0}^{\infty} \big(\frac{2}{z+1}\big)^l \frac{(\frac{1}{2})_{l}}{l!}   \sum_{k=l}^{\infty} \frac{1}{k!}     \big(v(z+1)\big)^k,
    \end{align*}
    where   the rising factorial $(a)_l=a (a-1)\cdots (a-l+1)$.
    Thus, by  the simple fact
      \begin{align*}
       \sum_{k=l+1}^{\infty} \frac{1}{k!}    \big(v(z+1)\big)^k=(z+1)^{l+1}e^{v(z+1)}\frac{1}{l!}\int_{0}^{v} ds\, s^le^{(v-s)(z+1)},
    \end{align*}
    from which we further have
     \begin{align*}
       I_2 &=M(\frac{1}{2},1,2v)+
       (z+1) \int_{0}^{v} ds\, e^{(v-s)(z+1)} M(\frac{1}{2},1,2s).
    \end{align*}

    Combination of    $I_1$ and $I_2$  leads to \eqref{gfunction}.  Moreover,  note that the residues of the  $z$-variable integrand  in \eqref{SneqII}  at $z=1$ and $z=-1$  cancel out,   we  thus deform the contour such that it contains $\pm 1$ and 0 since  0 is not a pole when $M=N$. This    completes  the proof of \eqref{Sneq-2}.

    In order to complete     \eqref{ISneqII},
    we just need to perform similar operations upon \eqref{Sneq-3} and change $ w\to 1/w$.  To prove  \eqref{ISneq-2}, let $\mathcal{C}_{r}$ denote a centered circle of radius $r$, then with $  \max\{\tau/\sigma_{l}: 1\leq l\leq N\} <r_1<1$ we have
     \begin{align*}
      IS_{N}(u, v)       = -\Epsilon(\eta_{-}u, \eta_{-}v) +&4\pi  \int_{\mathcal{C}_{r_1}} \frac{dz}{2\pi i}
      \int_{ \mathcal{C}_{r_1}} \frac{dw}{2\pi i}  \frac{  g(w,\eta_{-}u)}{1-w^2}  \frac{  g(z,\eta_{-}v)}{1-z^2} \nonumber\\
      & \quad\times
        \frac{z-w}{1-zw}        \prod_{k = 1}^{N} \frac{\tau z-\sigma_{k}}{\sigma_{k}z-\tau}
         \frac{\tau w-\sigma_{k}}{\sigma_{k}w-\tau}.     \end{align*}
    We can first deform  $\mathcal{C}_{r_1}$ into $\mathcal{C}_{r_2}$ where $1<r_2<1/r_{1}$  since
     the residues of  the  $w$-variable integrand at $\pm 1$   cancel out. Next,  for the  $z$-variable function we deform $\mathcal{C}_{r_1}$ into   $\mathcal{C}_{r_3}$  such that   $  r_{3}>\max\{\sigma_{l} /\tau: 1\leq l\leq N\}$  and $  r_{3}>1/r_{2} $.  So by considering  the residues at $\pm 1$,  which cancel out,   and at $z=1/w$, we get
   \begin{align*}
      IS_{N}(u, v)  &   = -\Epsilon(\eta_{-}u, \eta_{-}v)   +4\pi
      \int_{ \mathcal{C}_{r_2}} \frac{dw}{2\pi i}  \frac{ 1}{w^{2}-1}   g(w,\eta_{-}u)   g(\frac{1}{w},\eta_{-}v) \nonumber\\
      & +
       4\pi  \int_{\mathcal{C}_{r_3}} \frac{dz}{2\pi i}
      \int_{ \mathcal{C}_{r_2}} \frac{dw}{2\pi i}  \frac{  g(w,\eta_{-}u)}{1-w^2}  \frac{  g(z,\eta_{-}v)}{1-z^2} \frac{z-w}{1-zw}        \prod_{k = 1}^{N} \frac{\tau z-\sigma_{k}}{\sigma_{k}z-\tau}
         \frac{\tau w-\sigma_{k}}{\sigma_{k}w-\tau}.     \end{align*}
  The requested integral representation  immediately follows from  the  fact that   the first two terms on the right-hand side cancel out.

  As a matter of fact,  using the integral representation   \eqref{I0rep} and Lemma \ref{intgauss} we rewrite
   \begin{align}
      \Epsilon(u, v) =4\pi
      \int_{ \mathcal{C}_{r_1}} \frac{ds}{2\pi i s}   \int_{ \mathcal{C}_{r_1}} \frac{dt}{2\pi i t} e^{\frac{u}{s}+\frac{v}{t}}  \frac{ s-t}{1-s-t } \frac{1}{\sqrt{1 - 2  s}} \frac{1}{\sqrt{1 - 2  t}}, \label{weightE}
    \end{align}
    where $r_1<1/2$.   Note that  $w=t/(1-t)$  is the unique pole of the following $w$- function  inside  the  circle $ \mathcal{C}_{r_4}$ with  $2r_1<r_4<1$, we have
      \begin{align*}
      \frac{ s-t}{1-s-t } =
      \int_{ \mathcal{C}_{r_4}} \frac{dw}{2\pi i }   \frac{ 1}{w^2 -1 } \frac{(w+1)s-w}{1 -(w+1)   s}  \frac{(\frac{1}{w}+1)s-\frac{1}{w}}{1 -(\frac{1}{w}+1)   s}.    \end{align*}
    Substituting it into  \eqref{weightE} and  recalling  the definition of $g(z,v)$, we have
       \begin{align*}
    \Epsilon(\eta_{-}u, \eta_{-}v)   =4\pi
      \int_{ \mathcal{C}_{r_4}} \frac{dw}{2\pi i}  \frac{ 1}{w^{2}-1}   g(w,\eta_{-}u)   g(\frac{1}{w},\eta_{-}v)
      \end{align*}
      from which the required fact follows since  $ \mathcal{C}_{r_4}$ can be deformed to  $\mathcal{C}_{r_2}$.
  \end{proof}

We also need asymptotic approximations  for the function  $g(z,v)$.
\begin{prop} \label{g-asym}
      Let  $g(z,v)$ be defined by \eqref{gfunction}, as $v\to \infty$ we have  asymptotic approximations
      \begin{align}  g(z,v)
      =\begin{cases}-e^{(z+1)v} \sqrt{z^2-1} +\frac{e^{2v}}{\sqrt{2\pi v}}\big(1+\bo(\frac{1}{v})\big) , \quad & \Re\{z\}>1;\\
    \frac{e^{2v}}{\sqrt{2\pi v}}\big(1+\bo(\frac{1}{v})\big),\quad &  \Re\{z\}< 1.   \end{cases}
    \end{align}
\end{prop}

 \begin{proof}   Noticing  the approximation of  the confluent hypergeometric function \cite[13.7.2]{OOL}
  \begin{align}   M(\frac{1}{2},1,2v) =\frac{e^{2v}}{\sqrt{2\pi v}}\big(1+\bo(\frac{1}{v})\big), \label{confasym}
    \end{align}
     it remains to  approximate the integral
     \begin{equation}I= \int_{0}^{v} ds\, e^{-(z+1)s} M(\frac{1}{2},1,2s).
     \end{equation}

     When $\Re\{z\}>1$,  apply  the integral formula \cite[13.10.4]{OOL} and we  get
       \begin{equation*}I= \frac{1}{\sqrt{(z+1)(z-1)}}-\int_{v}^{\infty} ds\, e^{-(z+1)s} M(\frac{1}{2},1,2s).
     \end{equation*}
       Changing variables $s$ to  $v(s+1)$ and substituting  the approximation \eqref{confasym}   give us
      \begin{align*}I&= \frac{1}{\sqrt{z^2-1}}-v\int_{1}^{\infty} ds\, e^{(1-z)v(s+1)}   \frac{1}{\sqrt{2\pi v(s+1)}}\Big(1+\bo\big(\frac{1}{v(s+1)}\big)\Big)\\
      & =\frac{1}{\sqrt{z^2-1}}-ve^{(1-z)v} \big(1+\bo(\frac{1}{v})\big) \int_{1}^{\infty}ds\, e^{(1-z)vs}   \frac{1}{\sqrt{2\pi v}}\\
      &=\frac{1}{\sqrt{z^2-1}}+e^{(1-z)v}  \frac{1}{\sqrt{2\pi v}} \frac{1}{1-z} \big(1+\bo(\frac{1}{v})\big).
     \end{align*}
Together with \eqref{confasym},  we obtain the approximation of $g(z,v)$ when $\Re\{z\}>1$.

Next, we turn to the case of $\Re\{z\}< 1$.   By \eqref{confasym}, we have for large $s$
  \begin{equation*}-(z+1)s+\log M(\frac{1}{2},1,2s)=(1-z)s +\log\sqrt{2\pi s}+\bo\big(\frac{1}{s}\big),
     \end{equation*}
from which there exists a large $v_0\in (0,v)$ such that  $\Re\{-(z+1)s+\log M(\frac{1}{2},1,2s)\}$ is
a strictly monotone increasing function over $(v_0,v)$ because $\Re\{z\}<1$. Therefore, we  use   the Laplace method  and take the leading  contribution near $v$ to arrive at
 \begin{align*}  \int_{v_0}^{v} ds\, e^{-(z+1)s} M(\frac{1}{2},1,2s)&=
  \int_{v_0}^{v} ds\, e^{(1-z)s}    \frac{1}{\sqrt{2\pi s}}\big(1+\bo(\frac{1}{s})\big) \\
  &= e^{(1-z)v}    \frac{1}{(1-z)\sqrt{2\pi v}} \big(1+\bo(\frac{1}{v})\big),
     \end{align*}
from  which we further  obtain
\begin{align*}
I= e^{(1-z)v}    \frac{1}{(1-z)\sqrt{2\pi v}}  \big(1+\bo(\frac{1}{v})\big).
 \end{align*}
Combine  with \eqref{confasym} and we complete the proof   when $ \Re\{z\}< 1$.
\end{proof}

In  order to study  local eigenvalue statistics as $N \rightarrow \infty$,   we need to discuss some properties   of the $z$-function
  \begin{align}
    f(x;z) = \eta_{-}x (z+1)+ \log(\tau z - 1) -\log(z-\tau), \label{phase}
  \end{align}
  where  the  spectral variable $x\in  (0,\infty)$.   Solving the saddle point equation
  \begin{align*}
    \frac{d}{dz} f(x;z)= \eta_{-} x + \frac{\tau}{\tau z - 1} - \frac{1}{z - \tau}=0,
  \end{align*}
  and noting  $\eta_{-}=\tau /(1-\tau^2)$, we see that it has two solutions
  \begin{align}
    z_{\pm}  = \frac{1}{2}\Big(\frac{1}{\tau}+\tau \pm i \big( \frac{1}{\tau}-\tau\big)\sqrt{\frac{4}{x}-1}\Big).
  \end{align}
 Therefore, we can  divide the spectral variable $x\in (0,\infty)$ into three different regimes    by the nature of the saddle points:  \begin{enumerate}
\item [(I)] In the bulk regime of  $x\in (0,4)$, there are two complex conjugate roots;
 \item [(II)] at the soft edge  regime of $x=4$ the two complex saddle points  coalesce into  one
  \begin{equation}
    z_0=\frac{1}{2}(\frac{1}{\tau}+\tau); \label{z0saddle}
  \end{equation}
\item [(III)]in the outlier  regime of  $x\in (4,\infty)$, there are two real roots.
 \end{enumerate}

Given  a  fixed nonnegative integer $ n$,  let $\sigma_{n+1} = \cdots = \sigma_{N} = 1$. Introduce an undetermined non-zero parameter $\varphi(x)$ which depends on  $x\in(0, \infty)$. Using  a   gauge  transformation, which does not change the Pfaffian point process,  allows us  to rewrite  the  sub-kernels $DS_N, S_N$ and $IS_N$, say $S_N$,  as
  \begin{align}
   e^{\frac{\eta_{-}}{\varphi(x)}(\Re{z_{-}}+1)( u-v)}   \frac{1}{\varphi(x)} S_{N}&\Big(N\big(x+\frac{u}{N\varphi(x)}\big), N\big(x+\frac{v}{N\varphi(x)}\big)\Big) =: I_{N}+J_{N}  \label{Sleading}
  \end{align}
  where
  \begin{align}
    I_{N} &= \frac{ \eta_{-}}{\varphi(x)} \int_{\{z \in \mathcal{C}:\,  \Re{z} > 1\}} \frac{dz}{2\pi i}  \int_{\Sigma} \frac{dw}{2\pi i}   e^{N(f(x;z) - f(x;w))}  e^{\frac{\eta_{-}}{\varphi(x)}( v(z-\Re{z_{-}})-u(w-\Re{z_{-}})) } H(z, w) \label{SleadingI}
  \end{align}
  and by Proposition \ref{g-asym}
  \begin{align}
    J_{N} &=   \frac{1}{\sqrt{2\pi \eta_{-}(Nx+\frac{v}{\varphi(x)})}} \frac{ \eta_{-}}{\varphi(x)} \int_{\mathcal{C} } \frac{dz}{2\pi i}  \int_{\Sigma} \frac{dw}{2\pi i}   e^{N(f_1(x;z) - f(x;w))}  \nonumber\\
    &\quad\times
    e^{\frac{\eta_{-}}{\varphi(x)}( v(1-\Re{z_{-}})-u(w-\Re{z_{-}})) }  H(z, w)
    \frac{\sqrt{z^{2}-1}}{1-z^{2}}  \Big(1 + \bo\big(\frac{1}{  \eta_{-}(Nx+\frac{v}{\varphi(x)})}\big)\Big),\label{SleadingJ}
  \end{align}
where \begin{equation} f_1(x;z)=2\eta_{-}x+ \log(\tau z - 1) -\log(z-\tau) \label{phase1}, \end{equation} and
  \begin{align}
    H(z, w) = \frac{ 1}{\sqrt{w^{2}-1}\sqrt{z^{2}-1}}   \frac{1-zw}{w-z}    \Big( \frac{z - \tau}{\tau z - 1}   \frac{\tau w - 1}{ w - \tau}\Big)^{n} \prod_{k=1}^{n} \frac{\sigma_{k}w - \tau}{\tau w - \sigma_{k}}  \frac{\tau z - \sigma_{k} }{ \sigma_{k}z - \tau}.
  \end{align}
   Below, we will specifically  choose proper contours $ \mathcal{C}$ and $\Sigma$  as required.

The following lemma will play a central role in choosing  integration contours.
 \begin{lem}\label{extr-f}
    Given  $\tau \in (0, 1)$ and $\theta \in [0, \frac{\pi}{2})$, let
    \begin{align}
      f_{\tau ,\theta}(z) = \frac{4 \tau \cos^{2}\theta }{1 - \tau^{2}} (z + 1) + \log\frac{\tau z - 1}{z - \tau}, \quad z \in  \mathbb{C}, \label{phase-var}
    \end{align}
    then  the following hold true.
    \begin{enumerate}
      \item\label{S1} Introduce two circles
            \begin{align}
              \mathcal{C}_{\tau, \theta} = \left\{\tau + \frac{1}{2 \cos\theta} \Big(\frac{1}{\tau} - \tau\Big) e^{i\phi}: \phi \in (-\pi, \pi]\right\}, \label{Ctau}
            \end{align}
            and  \begin{align}
              \mathcal{C}_{ 1/\tau, \theta} = \left\{\frac{1}{\tau} + \frac{1}{2 \cos\theta} \Big(\frac{1}{\tau} - \tau\Big) e^{i\phi}: \phi \in (-\pi, \pi]\right\},  \label{C/tau}
            \end{align}
           then   $\Re{f_{\tau, \theta}(z)}$ attains its maximum value  along  the curve $\mathcal{C}_{\tau, \theta}$ only  at  the two conjugate points
            \begin{align}
              z_{\pm} = \tau + \frac{1}{2 \cos\theta} \Big(\frac{1}{\tau} - \tau\Big) e^{\pm i \theta}, \label{saddle-var}
            \end{align}
while  $\Re{f_{\tau, \theta}(z)}$ obtains its minimum value along   $\mathcal{C}_{1/\tau, \theta}$ at  the same two points $ z_{\pm}$.
      \item\label{S2} For any given $\phi \in [-\frac{1}{4} \pi, \frac{1}{4} \pi]$, let   $$w_{\phi} = \frac{1}{\tau} + \frac{1}{2\cos\theta} ( \frac{1}{\tau}-\tau ) e^{i \phi},$$ $\Re{f_{\tau ,\theta}(r + i\Im{w_{\phi}})}$   is strictly increasing when $ r$  increases in the interval $ [\Re w_{\phi} , \infty)$. While,  as $ r$  increases in the interval $ [1, \frac{1}{2}(\tau + \frac{1}{\tau})- \frac{1}{2}( \frac{1}{\tau} -\tau)\delta_0]$ where $\delta_0>0$ is a sufficiently small  number,  there exists $y_{0} > 0$ such that,  for any given $y \in [-y_{0}, y_{0}]$,$\Re{f_{\tau ,\theta}(r + iy)}$ is strictly decreasing.
      \item\label{S3} For given $r \in (0, \frac{1}{2}(\tau + \frac{1}{\tau}))$, $\Re{f_{\tau ,\theta}(r+iy)}$ is strictly decreasing (increasing) as  $y$ changes from 0 to $\infty$ ($-\infty$); while  for given $r \in ( \frac{1}{2}(\tau + \frac{1}{\tau}), \infty)$, $\Re{f_{\tau ,\theta}(r+iy)}$ is strictly increasing (decreasing) as  $y$ changes from 0 to $\infty$ ($-\infty$).   When $r = \frac{1}{2}(\tau + \frac{1}{\tau})$, $\Re{f_{\tau ,\theta}(r+iy)}$ is constant.
        \item \label{S4}   When $x>4$,  statement \eqref{S3} holds for  the function $\Re{f(x;r+iy)}$ defined in \eqref{phase}, too.
             \end{enumerate}
  \end{lem}
  \begin{proof}
    When $z \in   \mathcal{C}_{\tau, \theta}$,   statement \ref{S1}  immediately follows from the fact  that
    \begin{align*}
      \frac{\partial}{\partial \phi} \Re{f_{\tau, \theta}\Big(\tau + \frac{1}{2 \cos\theta} \Big(\frac{1}{\tau} - \tau\Big) e^{i\phi}\Big)}
      &= \frac{8 \cos^{2}\theta \sin\phi (\cos\phi - \cos\theta)}{4\cos^{2}\theta - 4\cos\theta\cos\phi + 1}.
    \end{align*}
 The case of  $w \in   \mathcal{C}_{1/\tau, \theta}$ can be similarly proven.

    To prove  statement \ref{S2}, simple calculation yields
    \begin{align}
      \frac{\partial}{\partial r} \Re{f_{\tau, \theta}(r + iy)}
      &=\frac{4\tau }{1 - \tau^{2}} \cos^{2}\theta + \frac{1 - \tau^{2}}{\tau}  \frac{r^{2} - r(\tau + \frac{1}{\tau}) + 1 - y^{2}}{\big((r - \frac{1}{\tau})^{2} + y^{2}\big) \big((r - \tau)^{2} + y^{2}\big)}.  \label{temp1}
    \end{align}
    When $w = r + iy \in \mathcal{C}_{1/\tau, \theta}$, say  $w_{\phi}$ above, we have
    \begin{align*}
      r^{2} - r(\tau + \frac{1}{\tau}) + 1 - y^{2}
      = \frac{1}{4} \big(\frac{1}{\tau} - \tau\big)^{2} \frac{ 2\cos\theta \cos\phi +  \cos 2\phi}{\cos^2\theta },
    \end{align*}
  which is positive whenever   $\phi \in[-\frac{1}{4} \pi,  \frac{1}{4}\pi]$. Hence    $\frac{\partial}{\partial r} \Re{f_{\tau, \theta}(r + i\Im{w_{\phi}})} > 0$ for all $r > \Re w_{\phi} $ and the required result follows.

  For  sufficiently small  $y$, take Taylor's expansion  and the RHS of  \eqref{temp1} becomes
    \begin{align*}
      \frac{\partial}{\partial r} \Re{f_{\tau, \theta}(r + iy)}
      &= \frac{4 \tau \cos^{2}\theta}{1 - \tau^{2}} - \frac{ \frac{1}{\tau}-\tau}{(\frac{1}{\tau}-r)(r - \tau)} + \bo(y^{2}).
    \end{align*}
    Note that   when $r\in [1, \frac{1}{2}(\tau + \frac{1}{\tau})- \frac{1}{2}( \frac{1}{\tau} -\tau)\delta_0]$  one obtains
     \begin{align*}
    & \frac{4 \tau \cos^{2}\theta}{1 - \tau^{2}} - \frac{ \frac{1}{\tau}-\tau}{(\frac{1}{\tau}-r)(r - \tau)}\leq  \frac{4 \tau }{1 - \tau^{2}} (\cos^{2}\theta -\frac{1}{1-\delta_{0}^2})<0,
    \end{align*}
 which implies that there exists $y_{0} > 0$ such that,  for any given $y \in [-y_{0}, y_{0}]$ the derivative $\frac{\partial}{\partial r} \Re{f_{\tau, \theta}(r + iy)}$ is negative. The desired monotonicity  thus follows.

    For statement \ref{S3},  since
    \begin{align*}
      \frac{\partial}{\partial y} \Re{f_{\tau, \theta}(r + iy)}
      &= \frac{2y \big(r - \frac{1}{2}(\tau + \frac{1}{\tau})\big) \big(\frac{1}{\tau} - \tau\big)}{\big((r - \frac{1}{\tau})^{2} + y^{2}\big) \big((r - \tau)^{2} + y^{2}\big)},
    \end{align*}
 when $y > 0$ it is easily to see that  the derivative   is negative for $r < \frac{1}{2}(\tau + \frac{1}{\tau})$,   positive for $r > \frac{1}{2}(\tau + \frac{1}{\tau})$ and is equal to  zero for $r = \frac{1}{2}(\tau + \frac{1}{\tau})$. When $y < 0$ it is  opposite in sign. Obviously, the required results follow.
  \end{proof}

\subsection{Critical case} \label{proofcritical}

Both the proofs of Theorems  \ref{edgecritthm} and  \ref{largestcritthm} reply on asymptotic analysis of the correlation kernels,   but for  the latter  we  need more delicate estimates  to obtain an upper bound. So  it is sufficient for us to  focus on the proof of Theorem\ref{largestcritthm}.

  \begin{proof}[Proof of Theorems  \ref{edgecritthm} and \ref{largestcritthm}]
   We need to  obtain asymptotic behaviour  of  the three sub-kernels $S_N, DS_N$ and $IS_N$, which are given by  \eqref{Sneq-2}, \eqref{DSneq-2} and \eqref{ISneq-2} respectively. The integrand in  the integral representation of $S_N$ has a typical form, whose  factors are relevant with   those in  both $DS_N$ and $IS_N$. Since all   can be handled similarly, we will just  focus on the most typical one $S_N$.

\textbf{Step 1:  Choice of contours.}

At the soft edge $x=4$, recalling  the assumptions on $\tau$ and $\sigma_{k}$  it is easy to obtain that  the saddle point
    \begin{align*}
      z_{0}
      = 1 +  2^{\frac{1}{3}}N^{-\frac{2}{3}}\kappa^{2} + \bo\big(N^{-1}\big),
    \end{align*}
    and for $k=1, \ldots, m$
    \begin{align*}
      \frac{\sigma_{k}}{\tau}
= 1 + 2^{\frac{1}{3}}N^{-\frac{2}{3}}\kappa\pi_{k} + \bo\big(N^{-1}\big).
    \end{align*}
   Therefore,   for sufficiently large $N$ we can choose  two points $z_1$ and $w_1$  such that $1 < z_{1} < w_{1} < z_{0}$ where
   \begin{equation*} z_{1} = 1 + \frac{1}{3}2^{\frac{1}{3}}N^{-\frac{2}{3}} \kappa \pi_{*}, \quad w_{1} = 1 + \frac{2}{3} 2^{\frac{1}{3}} N^{-\frac{2}{3}} \kappa \pi_{*}, \quad \pi_{*} =  \min\{ \kappa,  \pi_{1}, \ldots,  \pi_{m} \}. \end{equation*}
 Introduce a circle with centre on the real axis,   which   passes through one point    $z_{2} = 1 + i 2^{\frac{2}{3}} \sqrt{3} N^{-\frac{1}{3}} \kappa$,
     \begin{align}
       \mathcal{Q}_{\tau,z_2} &=\left\{z\in \mathbb{C}:   \Big|\frac{\tau z-1}{z-\tau} \Big|= \Big|\frac{\tau z_{2}-1}{z_{2}-\tau} \Big|\right\}. \label{circletau}
    \end{align}
    It is easy to verify that  both $-1$ and $1$ lie inside the circle  $\mathcal{Q}_{\tau,z_2}$, but   $z_{1}$  may be   inside  or  to the right of  the circle.

     When  $z_{1}$ is to the right of   $\mathcal{Q}_{\tau,z_2}$,  we choose  $ \mathcal{Q}_{\tau,z_2}$  as   the $z$-contour  $ \mathcal{C}$ in \eqref{SleadingI}  and \eqref{SleadingJ}.
    While $z_{1}$ is inside  $\mathcal{Q}_{\tau,z_2}$,
    we deform  $\mathcal{C}$ into
     \begin{equation} \mathcal{\tilde{C}}= \left\{z\in \mathcal{Q}_{\tau,z_2} :  \mathrm{Arg}(z_{3}) \leq  |\mathrm{Arg}(z) | \leq \pi  \right\}   \cup  \left\{z=z_{1}+iy:  - \Im{z_{3}} \leq y\leq  \Im{z_{3}}   \right\}. \label{dcontour}\end{equation}
Here  $z_{3}$ is   the intersection of the vertical line $x=z_{1}$ and $\mathcal{Q}_{\tau,z_2}$ in the upper plane,  and  $\bar{z}_{3}$ be its conjugate point.

           Furthermore, we deform  $ \{z\in \mathcal{Q}_{\tau,z_2}:  \Re{z} > 1\}$ for  $I_N$   into  $\mathcal{C}_{-}^{1} \cup  \mathcal{C}^{\mathrm{local}}\cup   \mathcal{C}_{+}^{1}$, where
    \begin{align*}
      \mathcal{C}_{\pm}^{\mathrm{local}} &=\left\{z_{1} + r e^{\pm i\frac{2}{3}\pi}: r \in [0, 2^{\frac{5}{3}} N^{-\frac{1}{3}} \kappa]\right\},\\
      \mathcal{C}_{\pm}^{1} &=\left\{t \pm i 2^{\frac{2}{3}} \sqrt{3} N^{-\frac{1}{3}} \kappa: t \in [z_{1} - 2^{\frac{2}{3}} N^{-\frac{1}{3}} \kappa, 1]\right\},
    \end{align*}
and $\mathcal{C}^{\mathrm{local}} = \mathcal{C}_{-}^{\mathrm{local}} \cup \mathcal{C}_{+}^{\mathrm{local}}$.

   On the other hand,   let $R = 2 +  \max\{\sigma_{k}/\tau:m+1 \leq k \leq n\}$,  set
    \begin{align*}
      \Sigma_{\pm}^{\mathrm{local}} &= \left\{w_{1} + r e^{\pm i\frac{1}{3}\pi}: r \in [0, 2 \kappa N^{-\frac{1}{3}})\right\},\\
      \Sigma_{\pm}^{1} &= \left\{w_{1} + \kappa N^{-\frac{1}{3}} \pm iy: y \in [\sqrt{3} \kappa N^{-\frac{1}{3}}, \Im{w_{2}}]\right\},\\
      \Sigma_{\pm}^{2} &= \left\{Re^{\pm i\phi}: \phi \in [0, \mathrm{Arg}(w_{2}))\right\},
    \end{align*}
   where $w_{2}$  denotes  the intersection of   the vertical line  $x= w_{1} + \kappa N^{-\frac{1}{3}}$ and   $\Sigma_{+}^{2}$.   Write $\Sigma^{\mathrm{local}} = \Sigma_{+}^{\mathrm{local}} \cup \Sigma_{-}^{\mathrm{local}}$,  $\Sigma^{\mathrm{global}} = \Sigma_{-}^{1} \cup \Sigma_{-}^{2}\cup \Sigma_{+}^{2} \cup \Sigma_{+}^{1}$,  and  take $\Sigma^{\mathrm{global}}  \cup \Sigma^{\mathrm{lobal}}$ as the contour $\Sigma$ for $I_N$ and $J_N$,  defined in \eqref{SleadingI}  and \eqref{SleadingJ}.
See Figure \ref{softcritical} for an illustration.
    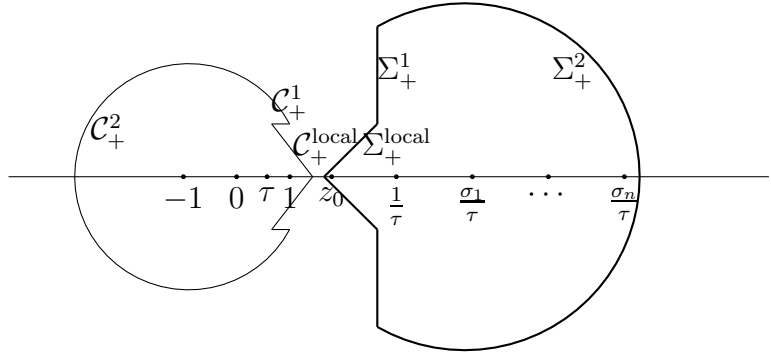
\begin{figure}[h]
      \centering
      \begin{tikzpicture}
        \draw (-3,0) -- (7,0);
        \draw (1, 0) -- ({-0.6 + 1.5*cos(pi/4 r)}, 0.7);
        \draw (0.6,0.4) node[right] {$\mathcal{C}_{+}^{\mathrm{local}}$};
        \draw ({-0.6 + 1.5*cos(pi/4 r)}, 0.7) -- (0.7, 0.7);
        \draw (0.7,0.6) node[above] {$\mathcal{C}_{+}^{1}$};
        \draw (1, 0) -- ({-0.6 + 1.5*cos(pi/4 r)}, -0.7);
        \draw ({-0.6 + 1.5*cos(pi/4 r)}, -0.7) -- (0.7, -0.7);
        \draw (-1.3,0.6) node[left] {$\mathcal{C}_{+}^{2}$};
        \draw[domain=0.155*pi:pi] plot({-0.63 + 1.5*cos(\x r)} ,{1.5*sin(\x r)});
        \draw[domain=-pi:-0.155*pi] plot({-0.63 + 1.5*cos(\x r)} ,{1.5*sin(\x r)});
        \draw[thick] (1.15, 0) -- ({3 - 2.3*cos(pi/3 r)}, 0.7);
        \draw (1.5,0.4) node[right] {$\Sigma_{+}^{\mathrm{local}}$};
        \draw[thick] ({3 - 2.3*cos(pi/3 r)} ,{2.3*sin(pi/3 r)}) -- ({3 - 2.3*cos(pi/3 r)}, 0.7);
        \draw (1.7,1.4) node[right] {$\Sigma_{+}^{1}$};
        \draw[thick] (1.15, 0) -- ({3 - 2.3*cos(pi/3 r)}, -0.7);
        \draw[thick] ({3 - 2.3*cos(pi/3 r)}, -0.7) -- ({3 - 2.3*cos(pi/3 r)} ,{-2.3*sin(pi/3 r)});
        \draw[domain=pi/3:pi,thick] plot({3 - 2.3*cos(\x r)} ,{2.3*sin(\x r)});
        \draw (4,1.4) node[right] {$\Sigma_{+}^{2}$};
        \draw[domain=-pi:-pi/3,thick] plot({3 - 2.3*cos(\x r)} ,{2.3*sin(\x r)});
        \filldraw
          (-0.7,0) circle (0.7pt) node[below] {$-1$}
          (0,0) circle (0.7pt) node[below] {$0$}
          (0.4,0) circle (0.7pt) node[below] {$\tau$}
          (0.7,0) circle (0.7pt) node[below] {$1$}
          (5/4,0) circle (0.7pt) node[below] {$z_{0}$}
          (2.1,0) circle (0.7pt) node[below] {$\frac{1}{\tau}$}
          (3.1,0) circle (0.7pt) node[below] {$\frac{\sigma_{1}}{\tau}$}
          (4.1,0) circle (0.7pt) node[below] {$\cdots$}
          (5.1,0) circle (0.7pt) node[below] {$\frac{\sigma_{n}}{\tau}$};
      \end{tikzpicture}
      \caption{Contours  of double integrals for $S_N$: critical} \label{softcritical}
    \end{figure}

\textbf{Step 2:  Estimate of $I_{N}$.}

    To estimate $I_{N}$, we split it into three parts
     \begin{align*}
      I_{N} &=   \Big(\int_{\mathcal{C}^{\mathrm{local}}} \frac{dz}{2\pi i}  \int_{\Sigma^{\mathrm{local}}} \frac{dw}{2\pi i}  + \int_{\mathcal{C}_{-}^{1} \cup \mathcal{C}_{+}^{1}} \frac{dz}{2\pi i} \int_{\Sigma} \frac{dw}{2\pi i}+\int_{\mathcal{C}^{\mathrm{local}}} \frac{dz}{2\pi i}  \int_{\Sigma^{\mathrm{global}}} \frac{dw}{2\pi i}   \Big)\nonumber\\
      & \qquad  \frac{\eta_{-}}{\varphi(4)}  e^{N(f(4;z) - f(4;w))}  e^{\frac{\eta_{-}}{\varphi(4)}( v(z-z_0)-u(w-z_0)) } H(z, w) \nonumber\\
      &=:I_{N}^{1}+I_{N}^{2}+I_{N}^{3}.
    \end{align*}
  Let's  choose  $\varphi(4)=(16N)^{-\frac{1}{3}}$ in $I_N$ and $J_N$,
   and   tune   the   parameters $\sigma_1, \ldots,\sigma_n$ near  $z_0$  as assumed in Theorems  \ref{edgecritthm} \& \ref{largestcritthm}.   Noting that  for $k=1,
  \ldots, m$
\begin{equation*}
    (16N)^{\frac{1}{3}}\eta_{-}\big( z_{0}-\frac{\tau}{\sigma_k}\big)=  \pi_{k}+\kappa +\bo(N^{-\frac{1}{3}}), \       z_{0}-1=\frac{1}{4}\eta_{-}(1-\tau)  +\bo\big(\frac{1}{N}\big),  \end{equation*}
    take the Taylor expansion at $z_0$ and we obtain
    \begin{align*}
     H(z,w)&= \frac{ 1}{2\sqrt{w-z_{0}+ \frac{1}{4\eta_{-}}(1-\tau)}\sqrt{z-z_{0}+\frac{1}{4\eta_{-}}(1-\tau)}}        \frac{\frac{1}{2\eta_{-}}(1-\tau)+(z+w-2 z_{0})}{(z-z_{0})-(w-z_{0})}
  \nonumber\\
      &\quad\times  \Big(1+\bo(N^{-\frac{1}{3}})\Big)  \prod_{k = 1}^{m}
       \frac{(16N)^{\frac{1}{3}}\eta_{-}(z-z_{0})-\pi_{k}+\kappa}{(16N)^{\frac{1}{3}}\eta_{-}(w-z_{0})-\pi_{k}+\kappa }  \frac{(16N)^{\frac{1}{3}}\eta_{-}(w-z_{0})+\pi_{k}+\kappa}{(16N)^{\frac{1}{3}}\eta_{-}(z-z_{0})+\pi_{k}+\kappa }.
        \end{align*}

On the other hand, $f(4;z)=f(4;z_0)+\frac{1}{6}f'''(z_0) (z-z_{0})^3+ \bo(\eta_{-}^{4}(z-z_{0})^4)$
   where
     \begin{align*}
   f(4;z_0)=\log(-\tau)+2(1+\tau)/(1-\tau),    \qquad   f'''(4;z_0) =  -32 \eta_{-}^{3}.
 \end{align*}
   Thus,  by using   change of variables   $z
   \to z_{0}+(16N)^{-\frac{1}{3}} \frac{1}{\eta_{-}}z$,   $w
   \to z_{0}+(16N)^{-\frac{1}{3}} \frac{1}{\eta_{-}}w$  we arrive at      \begin{align}
        I_{N}^{1} &= \int_{ (16N)^{\frac{1}{3}}\eta_{-}\big( \mathcal{C}^{\mathrm{local}}-z_{0}\big)  } \frac{dz}{2\pi i }  \int_{(16N)^{\frac{1}{3}}\eta_{-}\big(\Sigma^{\mathrm{local}} -z_{0}\big)} \frac{dw}{2\pi i}  e^{\frac{1}{3}(w^3-z^3)}  e^{ vz- uw}  \nonumber\\
      &\quad\times  \Big(1+\bo(N^{-\frac{1}{3}})\Big)
       \frac{ 1}{2\sqrt{w+ \kappa}\sqrt{z+\kappa}}
      \frac{z+w+2\kappa}{z-w}  \prod_{k = 1}^{m} \frac{ z-\pi_{k}+\kappa}{z+\pi_{k}+\kappa} \frac{ w+\pi_{k}+\kappa}{w-\pi_{k}+\kappa}. \label{I1eqn-1}
    \end{align}
  As $N\to \infty$, change  the rescaled contours  to  $\mathcal{C}_{>}-\kappa$ and $\mathcal{C}_{<}-\kappa$ and make a shift by $\kappa$,  we thus  obtain  \begin{align}
     I_{N}^{1} &= \Big(1+\bo(N^{-\frac{1}{3}})\Big)  S^{(\mathrm{soft})}(\kappa, \pi;u,v), \label{crit-i11}
    \end{align}
     which holds true uniformly for  $u, v$ in a compact subset of $\mathbb{R}$.
   Moreover,  note that  for $z\in    \mathcal{C}^{\mathrm{local}}$ and $w\in    \Sigma^{\mathrm{local}}$,
  \begin{align}  (16N)^{\frac{1}{3}}\eta_{-}\big(\Re z-z_{0}\big) &\leq  (16N)^{\frac{1}{3}}\eta_{-}\big(z_{1}-z_{0}\big) =\frac{1}{3}\pi_{*} -\kappa +\bo(N^{-\frac{1}{3}}) \leq  -a,
   \end{align}
   and
  \begin{align}  (16N)^{\frac{1}{3}}\eta_{-}\big(z_{0}-\Re w\big) &\leq  (16N)^{\frac{1}{3}}\eta_{-}\big(z_{0}-w_{1}\big)  =\kappa -\frac{2}{3}\pi_{*} +\bo(N^{-\frac{1}{3}}) \leq  b,  \label{b-bound}
   \end{align}
   where  \begin{equation} a= \kappa-\frac{4}{9}\pi_{*}, \quad b= \kappa-\frac{5}{9}\pi_{*}, \label{ab}\end{equation}   we   see  from \eqref{I1eqn-1}  that there exists a constant $C_1>0$  such that for large $N$
    \begin{align}
    |I_{N}^{1}|
      \leq   C_{1} e^{-av + bu},  \quad u, v \geq 0. \label{crit-i11'}
    \end{align}

For $I^{2}_{N}$,   we  claim that    $\Re{f(4;z)}$  increases  as $z$ moves from left to right along the horizontal line  $\mathcal{C}_{+}^{1}$  or  $\mathcal{C}_{-}^{1}$ when  $N$ is  large.  This can be verified from the fact
 \begin{align}
    &  \frac{\partial}{\partial x} \Re{f(4;x + i 2^{\frac{2}{3}} \sqrt{3} N^{-\frac{1}{3}} \kappa)} \nonumber\\
     &= 4\eta_{-} + \frac{x - \frac{1}{\tau}}{(x - \frac{1}{\tau})^{2} + 2^{\frac{4}{3}} 3N^{-\frac{2}{3}} \kappa^{2}} - \frac{x - \tau}{(x - \tau)^{2} + 2^{\frac{4}{3}} 3N^{-\frac{2}{3}} \kappa^{2}}\nonumber\\
      &\geq 4\eta_{-} - \frac{\frac{1}{\tau} - \tau}{  2^{\frac{4}{3}} 3N^{-\frac{2}{3}} \kappa^{2}} = \frac{1}{3\kappa} 2^{\frac{4}{3}}   N^{\frac{1}{3}} \big(1 + \bo(N^{-\frac{1}{3}})\big)> 0 \label{crit-z2}
    \end{align}
    for all $x \in [z_{1} - 2^{\frac{2}{3}} N^{-\frac{1}{3}} \kappa, 1]$, which implies that
 \begin{align*}
      \max_{z \in \mathcal{C}_{+}^{1} \cup \mathcal{C}_{-}^{1}} \Re{f(4; z)} \leq \Re{f(4; z_{2})}.
    \end{align*}
Besides,  simple manipulation leads to
    \begin{align}
      \Re\{f(4;z_{2})- f(4;z_{0})\}
      &=\frac{4\tau }{1 - \tau^{2}}  \big(\Re{z_{2}}-z_{0}\big) + \log\Big|\frac{z_{2} - \frac{1}{\tau}}{z_{2}-\tau}\Big| \nonumber\\
       &=- (2^{-\frac{1}{3}} - 2^{-\frac{4}{3}}) \kappa N^{-\frac{1}{3}} + \bo(N^{-\frac{2}{3}})\nonumber \\
       &\leq -\frac{1}{4}\kappa N^{-\frac{1}{3}}. \label{crit-z1}
         \end{align}
Hence, we get
    \begin{align*}
      \max_{z \in \mathcal{C}_{-}^{1} \cup \mathcal{C}_{+}^{1}} \Re\{f(4;z)- f(4;z_{0})\} \leq -\frac{1}{4}\kappa N^{-\frac{1}{3}}. 
    \end{align*}
   Together with    \begin{align*} \frac{\eta_{-}}{\varphi(4)} (z-z_{0})=\bo(N^{\frac{1}{3}}),  \quad \forall z \in \mathcal{C}_{-}^{1} \cup \mathcal{C}_{+}^{1}, \end{align*}
   we proceed  for the $w$-integral  as in the case of $I_{N}^1$   and   give an estimate  for $N$ large sufficiently
   \begin{align}
      \left|  I_{N}^{2}\right|
      &\leq e^{-\frac{1}{8}\kappa N^{\frac{2}{3}}}, \label{crit-i12}
    \end{align}
     which holds  uniformly for  $u, v$ in a compact subset of $\mathbb{R}$.
    Moreover, there exists a constant $C_2>0$  such that for large $N$ such that
    \begin{align}
      \left| I_{N}^{2}\right|
      &\leq e^{\frac{\eta_{-}}{\varphi(4)}( v(\Re{z_{1}}-z_{0})-u(\Re{w_{1}}-z_{0})) } \frac{\eta_{-}}{\varphi(4)}
      \int_{\mathcal{C}_{-}^{1} \cup \mathcal{C}_{+}^{1}} \frac{|dz|}{2\pi }  \int_{\Sigma} \frac{|dw|}{2\pi }  \left| e^{N(f(4;z) - f(4;w))}   H(z, w)\right| \nonumber\\
      &\leq  C_2 e^{-av + bu} e^{-\frac{1}{8}\kappa N^{\frac{2}{3}}} \label{crit-i12'}
    \end{align}
   for all $u, v \geq  0$; cf. \eqref{crit-i11'}.

At last, for $I^{3}_{N}$, let $w_{2} = w_{1} + 2 \kappa N^{-\frac{1}{3}} e^{i\frac{1}{3}\pi}$, simple manipulation  shows
    \begin{align}
      \Re\{f(4;w_{2})- f(4;z_{0})\}
      &=4\frac{\tau }{1 - \tau^{2}} (\Re{w_{2}}-z_{0}) + \log\left|\frac{w_{2} - \frac{1}{\tau}}{w_{2}-\tau}\right| \nonumber\\
    &= 2^{\frac{1}{3}}  - \frac{1}{2} \log \frac{1+2^{\frac{2}{3}} -2^{\frac{1}{3}}}{
   1+2^{\frac{2}{3}} +2^{\frac{1}{3}}  } + \bo(N^{-\frac{1}{3}}) \nonumber \\
   &\geq   \frac{1}{4} \label{crit-w1}
    \end{align}
    for $N$  large.  Applying  statement \ref{S3} of Lemma \ref{extr-f}  gives us
    \begin{align}
      \min_{w \in \Sigma_{+}^{1} \cup \Sigma_{-}^{1}} \Re\{f(4;w)- f(4;z_{0})\} \geq  \Re{f(4;w_{2})} - \Re{f(4;z_{0})}\geq \frac{1}{4}. \label{crit-wg1}
    \end{align}
However, when $z=R e^{i\phi} \in \Sigma_{+}^{2}\cup \Sigma_{-}^{2}$  with $\phi \in [-\mathrm{Arg}(w_{2}), \mathrm{Arg}(w_{2})]$,
    \begin{align}
      \frac{\partial}{\partial \phi} \Re{f(4;R e^{i\phi})}
      &= - \Big(\frac{4\tau }{1 - \tau^{2}}+ \frac{\tau}{1 - 2\tau R \cos\phi + \tau^2 R^{2}} - \frac{\tau}{\tau^{2} - 2\tau R \cos\phi + R^{2}}\Big) R \sin\phi \nonumber\\
      &
\begin{cases}
<0,\quad  \phi > 0; \\
>0, \quad  \phi < 0.
\end{cases}
  \label{crit-w2}
    \end{align}
    This implies
        \begin{align}
      \min_{w \in \Sigma_{+}^{2} \cup \Sigma_{-}^{2}}\Re\{f(4;w)- f(4;z_{0})\} \geq \Re{f(4;w_{2})} - \Re{f(4;z_{0})} \geq \frac{1}{4}. \label{crit-wg2}
    \end{align}
    Combining \eqref{crit-wg1} and \eqref{crit-wg2}, when  $N$ is large   we obtain
    \begin{align}
      \left|I_{N}^{3}\right|
      &\leq  e^{-\frac{1}{8} N}, \label{crit-i13}
    \end{align}
    uniformly for  $u, v$ in a compact subset of $\mathbb{R}$
    and
    \begin{align}
      |  I_{N}^{3}|&\leq
     e^{\frac{\eta_{-}}{\varphi(4)} ( v(\Re{z_{1}}-z_{0})-u(\Re{w_{1}}-z_{0})) } \frac{\eta_{-}}{\varphi(4)} \int_{\mathcal{C}^{\mathrm{local}}} \frac{|dz|}{2\pi } \int_{\Sigma^{\mathrm{global}}} \frac{|dw|}{2\pi }   \left|e^{N(f(4;z) - f(4;w))}  H(z, w)\right| \nonumber\\
      &\leq  C_{3} e^{-av + bu} e^{-\frac{1}{8} N } \label{crit-i13'}
    \end{align}
    for all $u, v \geq 0$.

    Combining \eqref{crit-i11}, \eqref{crit-i12} and \eqref{crit-i13},  we arrive at
    \begin{align}
     I_{N} =  \big(1+\bo(N^{-\frac{1}{3}})\big) S^{(\mathrm{soft})}(\kappa, \pi;u,v), \label{crit-i1}
    \end{align}
     uniformly for  $u, v$ in a compact subset of $\mathbb{R}$,
   while    combining \eqref{crit-i11'}, \eqref{crit-i12'} and \eqref{crit-i13'} we know  that there exists a constant  $C_4>0$ such that
    \begin{align}
    |I_{N}| \leq C_{4} e^{-au + bv}, \quad  \forall\, u, v\geq 0, \label{crit-i1'}
    \end{align}
    for sufficiently large $N$.

\textbf{Step 3: Estimate  of $J_{N}$}.

    Below we obtain upper bounds     according to  the choice of the $z$-contours  $\mathcal{Q}_{\tau,z_2}$  or $\mathcal{\tilde{C}}$, defined  by \eqref{circletau} or \eqref{dcontour}.

    In the case of the $z$-contour  $\mathcal{Q}_{\tau,z_2}$,     noting the definition of  $\mathcal{Q}_{\tau,z_2}$ and the fact that $ \Re\{f_{1}(4;z_2)- f(4;z_{2})\}=0$,  by \eqref{crit-z1} we have
     \begin{align}
   \Re\{f_{1}(4;z)- f(4;z_{0})\} &= \Re\{f(4;z_2)- f(4;z_{0})\}  \leq -\frac{1}{4}\kappa N^{-\frac{1}{3}}.   \label{crit-j-1}   \end{align}
     Recalling   $J_N$ defined in \eqref{SleadingJ},   since  \begin{align}H(z,w)=\bo(N^{\frac{4}{3}}),  \quad \forall z \in \mathcal{Q}_{\tau,z_2},  \  w\in \Sigma,  \label{H-2}\end{align}
      for $N$ large we can deal with the $w$-integral   in a similar way as for $I_{N}$ and   see from \eqref{crit-j-1}   that
     \begin{align}
    |J_{N}| &=   \bo\big(N^{\frac{2}{3}}\big)   e^{ -\frac{1}{4}\kappa N^{\frac{2}{3}}} \leq e^{ -\frac{1}{8}\kappa N^{\frac{2}{3}}}, \label{critjb1}
  \end{align}
   uniformly for  $u, v$ in a compact subset of $\mathbb{R}$.
   Moreover,  note that  for $z\in  \mathcal{Q}_{\tau,z_2}$ and $w\in    \Sigma$ (cf. eq \eqref{b-bound} and eq \eqref{ab}),
  \begin{align*} \frac{\eta_{-}}{\varphi(4)}  \big(1-z_{0}\big) &=-\kappa \leq  -a,   \quad  \frac{\eta_{-}}{\varphi(4)} \big(z_{0}-\Re w\big)   \leq  b,
   \end{align*}
   we get
  \begin{align}
    |J_{N}|    \leq e^{ -\frac{1}{8}\kappa N^{\frac{2}{3}}}
   e^{-av + bu}, \quad u, v\geq 0, \label{crit-j11'}
    \end{align}
    and  further
     \begin{align}
    |J_{N}| \leq C_5
   e^{-av + bu}, \quad u, v\geq 0 \label{critjb2}
    \end{align}
    for  some $C_5 > 0$.

    On the other hand,  in the case of the $z$-contour  $\mathcal{\tilde{C}}$,   combing the definition of  $\mathcal{Q}_{\tau,z_2}$ and statement \ref{S3} of Lemma \ref{extr-f},  we obtain
     \begin{align}
      \max_{z  \in \mathcal{\tilde{C}}} \Re\{f_{1}(4;z)- f(4;z_{0})\} &= \max_{  -\Im{z_{3}}\leq y\leq  \Im{z_{3}} } \Re\{f_{1}(4;z_{1}+iy)- f(4;z_{0})\} \nonumber\\
      &={f_{1}(4;z_{1})- f(4;z_{0})}\nonumber\\
      &=-\frac{1}{3}2^{\frac{2}{3}}\pi_{*} N^{-\frac{1}{3}}+\bo(N^{-\frac{2}{3}}). \label{crit-gz1}
    \end{align}
     Similarly, for $N$ large we see from \eqref{crit-gz1}   that
     \begin{align}
    |J_{N}| &=   \bo\big(N^{\frac{2}{3}}\big) \exp\Big(  -\frac{1}{3}2^{\frac{2}{3}}\pi_{*} N^{\frac{2}{3}}+\bo(N^{\frac{1}{3}})   \Big) \leq e^{ -\frac{1}{3}\pi_{*} N^{\frac{2}{3}}}, \label{critjb3}
  \end{align}
   uniformly for  $u, v$ in a compact subset of $\mathbb{R}$ and    further
     \begin{align}
    |J_{N}| \leq C_6
   e^{-av + bu}, \quad u, v\geq 0 \label{critjb4}
    \end{align}
    for  some $C_6 > 0$.

    \textbf{Step 4:  Completion}.

    We can proceed as in  the estimates of  $I_N$ and $J_N$ for  $S_N$ to  obtain similar  approximations for  $DS_N$ and $IS_N$; see  \eqref{crit-i1}, \eqref{crit-i1'}, \eqref{critjb1}, \eqref{critjb2}, \eqref{critjb3}, and \eqref{critjb4}. Thus, we summarize them as follows.  Put
      \begin{align*}\varphi(4)=(16N)^{-\frac{1}{3}}, \quad h(u)= e^{\frac{\eta_{-}}{\varphi(4)}(z_{0}+1)u}, \quad T_{n}=e^{-2Nf(z_{0})}\prod_{k=1}^{n}\big(\frac{\sigma_k}{\tau}\big)^{2}.  \end{align*}
    Then   \begin{align*}
   \frac{1}{\varphi(4)}   \frac{
   h(u) }{h(v)}    S_{N}\Big(4N+\frac{u}{\varphi(4)}, 4N+\frac{v}{\varphi(4)}\Big)
    & =  \big(1+\bo(N^{-\frac{1}{3}})\big) S^{(\mathrm{soft})}(\kappa, \pi;u,v),
    \\ T_{n}   h(u)h(v)  \frac{1}{\varphi(4)}  DS_{N}\Big(4N+\frac{u}{\varphi(4)}, 4N+\frac{v}{\varphi(4)}\Big) &=  \big(1+\bo(N^{-\frac{1}{3}})\big) DS^{(\mathrm{soft})}(\kappa, \pi;u,v),
    \\
   \frac{1}{T_{n}   h(u)h(v)}  \frac{1}{\varphi(4)}  IS_{N}\Big(4N+\frac{u}{\varphi(4)}, 4N+\frac{v}{\varphi(4)}\Big) &=  \big(1+\bo(N^{-\frac{1}{3}})\big) IS^{(\mathrm{soft})}(\kappa, \pi;u,v),
    \end{align*}
   hold  uniformly for  $u, v$ in a compact subset of $\mathbb{R}$.  Note that when taking the Pfaffian, the factors   $T_n, h(u)$ and $h(v)$
 cancel out each other,  combination of the above three equations  completes  the proof of Theorem \ref{edgecritthm}.

   Moreover,
  given  any   real $u_0$ there exists some constant $C>0$ such that  for $  \forall u, v\geq u_0$
     \begin{align*}
   \frac{1}{\varphi(4)}   \frac{
   h(u) }{h(v)}
    \abs{S_{N}\Big(4N+\frac{u}{\varphi(4)}, 4N+\frac{v}{\varphi(4)}\Big)}
    &   \leq C e^{-au + bv},
    \\   h(u)h(v) T_{n}   \frac{1}{\varphi(4)}  \abs{DS_{N}\Big(4N+\frac{u}{\varphi(4)}, 4N+\frac{v}{\varphi(4)}\Big)} &\leq C e^{-au -av},
    \\
   \frac{1}{   h(u)h(v) T_{n} }  \frac{1}{\varphi(4)}  \abs{IS_{N}\Big(4N+\frac{u}{\varphi(4)}, 4N+\frac{v}{\varphi(4)}\Big)} &\leq C e^{bu + bv},
       \end{align*}
    when $N$ is  sufficiently large.  By  \cite[Lemma 2.6]{BBCS},  we obtain  for  $k\geq 1$
    \begin{equation}
  \left| \frac{1}{(\varphi(4))^k} \mathrm{Pf} \Big[
        K_{N}\Big(4N+\frac{u_i}{\varphi(4)}, 4N+\frac{u_j}{\varphi(4)}\Big) \Big]_{i,j = 1}^{k} \right| \leq (2k)^{\frac{k}{2}} C^{k} \prod_{j=1}^{k} e^{-(a-b)u_{j}}. \label{upperbound}
  \end{equation}
  Applying  dominated convergence theorem to   the finite $N$  Pfaffian  series expansion
      \begin{align}
 &
 \mathrm{P} \Big(2^{-\frac{4}{3}}N^{-\frac{1}{3}}\big(\lambda_{\mathrm{max}}-4N\big) \leq   x\Big)=1+ \nonumber \\
  &\sum_{k=1} ^{N}  \frac{(-1)^k}{k!} \frac{1}{(\varphi(4))^k}  \int_{x}^{\infty} \cdots   \int_{x}^{\infty}  \mathrm{Pf}\Big[ K_{N}\Big(4N+\frac{u_i}{\varphi(4)}, 4N+\frac{u_j}{\varphi(4)}\Big)
    \Big]_{i,j = 1}^{k} du_1 \cdots   du_k, \label{seriessum}\end{align}
       we thus   conclude  the proof of     Theorem  \ref{largestcritthm}.
  \end{proof}

\subsection{Noncritical case} \label{sect:noncritical}
We state   and prove  limit  results about correlation functions   in the  bulk  and the BBP  transition of the largest eigenvalue when $\tau\in (0,1)$ is a  fixed number.
\begin{thm} \label{bulk limit}
    With the same notations as in Theorem \ref{SVdensity} and with     $M=N$ even,  for fixed  $\tau\in (0,1)$ and nonnegative integer $n$,  assume  that
    \begin{equation}  \sigma_{n+1} = \cdots = \sigma_{N} = 1  \  \mathrm{and} \
    \sigma_i \in (\tau,\infty),  \quad  i= 1, \ldots, n.
    \end{equation}
     Let  \begin{align}\varphi(x)=\frac{1}{2\pi}\sqrt{\frac{4-x}{x}}, \quad 0<x<4,\end{align}
    then  for any fixed $x\in (0,4)$      
    \begin{equation}
   \lim_{N \to \infty}  \Big(\frac{1}{\varphi(x)}\Big)^k R_{N}^{(k)}\Big(Nx+\frac{u_1}{\varphi(x)}, \ldots, Nx+\frac{u_k}{\varphi(x)}\Big)= \det \Big[\frac{\sin \pi(u_{i}-u_{j})}{\pi(u_{i}-u_{j})}\Big]_{i,j = 1}^{k}
  \end{equation}
 holds  uniformly for  $u_1, \ldots, u_k$ in a compact subset of $\mathbb{R}$.
  Particularly when  $k=1$,  it reduces to  the  Marchenko-Pastur law
      \begin{equation}
   \lim_{N \to \infty}  R_{N}^{(1)}\big(Nx, Nx\big)dx=  \varphi(x)dx, \quad 0<x< 4.
  \end{equation}
  \end{thm}

We remark that  Adhikari and  Bose have recently proved the  Marchenko-Pastur law for  singular values of   real elliptic random matrices under the assumption  that matrix entries have  all higher moments; see   \cite[Theorem 3]{AB}. Different from our asymptotic analysis for the one-point correlation function, they  use  the moment method to  obtain more information about joint moments of several independent elliptic random matrices.

We need  the following Airy-type kernel with deformed real parameters $\pi_1, \ldots, \pi_m$ defined in \cite{BBP05}
  \begin{equation}
 K_{\mathrm{Airy}}(\pi;u,v)= \int_{\mathcal{C}_{L}}\frac{dz}{2i\pi }    \int_{\mathcal{C}_{R}}   \frac{dw}{2i\pi } \frac{1}{z-w}e^{\frac{1}{3}w^3 -\frac{1}{3}z^3}   e^{vz-uw} \prod_{k = 1}^{m} \frac{z-\pi_{k}}{w - \pi_{k}} \label{Airykernle}
  \end{equation}
  where  $\mathcal{C}_{L}$ and $\mathcal{C}_{R}$ are  two  non-intersecting  contours   from   $e^{-2i\pi/3}\infty$ to  $e^{2i\pi/3}\infty$
 and from $e^{i\pi/3}\infty$  to $e^{-i\pi/3}\infty$ respectively, such that $\pi_1, \ldots, \pi_m$ are on   the right-hand side of $\mathcal{C}_{R}$.  The finite GUE kernel with external source is defined as
 \begin{equation}
 K_{\mathrm{GUE}}(\pi;u,v)= \int_{-i\infty}^{i\infty}\frac{dz}{2i\pi }    \int_{\mathcal{C}_{\{\pi_{1},\ldots,\pi_m\}}}   \frac{dw}{2i\pi } \frac{1}{z-w}e^{\frac{1}{2}z^2 -\frac{1}{2}w^2}   e^{vz-uw} \prod_{k = 1}^{m} \frac{z-\pi_{k}}{w - \pi_{k}}
  \end{equation}
where all  $\pi_{1},\ldots,\pi_m$ are on  the right-hand side of the $z$-contour. Correspondingly, define two families of probability distributions by the Fredholm determinant   series expansion (see \cite{BBP05} for more details) for $m=0, 1, 2, \ldots,$
       \begin{align}
    F_m&(\pi;x)=1+ \sum_{k=1} ^{\infty}  \frac{(-1)^k}{k!}\int_{x}^{\infty} \cdots   \int_{x}^{\infty}  \det[
      {    K_{\mathrm{Airy}}(\pi;  u_i, u_j)}]_{i,j = 1}^{k} du_1 \cdots   du_k, \label{detAiryseries}
  \end{align}
and    $m=1, 2,  \ldots,$
 \begin{align}
    G_m&(\pi;x)=1+ \sum_{k=1} ^{\infty}  \frac{(-1)^k}{k!}\int_{x}^{\infty} \cdots   \int_{x}^{\infty}  \det[
      {    K_{\mathrm{GUE}}(\pi;  u_i, u_j)}]_{i,j = 1}^{k} du_1 \cdots   du_k. \label{detGUEseries}
  \end{align}
Indeed,  when $m=0$ the distribution $F_0(\pi;x)$ does not depend on parameters $\{\pi_l\}$ and is identical  to  the well-known GUE Tracy-Widom distribution. At this time it  is usually  denoted by
the notation  $F_{\mathrm{GUE}}(x)$ (cf. \cite{tw1994}). For $m\geq 1$,
$G_m(\pi;x)$ is equal to the distribution of the largest eigenvalue for an $m\times m$ GUE matrix with mean matrix  $\text{diag}(\pi_1, \ldots, \pi_m)$.

The famous BBP transition phenomenon  found in \cite{BBP05} occurs  in our model and can be  formulated as follows.
 \begin{thm}  \label{edge limit}
 With the same notations as in Theorem \ref{SVdensity} and with     $M=N$ even,  for fixed  $\tau\in (0,1)$ and nonnegative integer $n$,  assume  that $\sigma_{n+1} = \cdots = \sigma_{N} = 1$.  Let $\lambda_{\mathrm{max}}$ be  the largest  eigenvalue  of $W=X^{*}X$, then     the  following     hold true    uniformly for  $u_1, \ldots, u_k$ in a compact subset of $\mathbb{R}$ or  for any  $x$ in a compact set of $\mathbb{R}$.
\begin{enumerate}
      \item[(i)] 
       For some $0\leq m\leq n$,  set
    \begin{equation}
     \sigma_k=\frac{1}{2}(1+\tau^2)+  N^{-\frac{1}{3}}2^{-\frac{4}{3}} (1-\tau^2)\pi_{k},  \quad  k= 1, \ldots, m.
    \end{equation} When  $\pi_1, \ldots, \pi_m$  are in a compact subset of $\mathbb{R}$  and $\sigma_{m+1},\ldots,\sigma_n$       are in a compact subset of $(\frac{1}{2}(1+\tau^2), \infty)$,

    \begin{equation}
   \lim_{N \to \infty}  (2^{\frac{4}{3}}N^{\frac{1}{3}})^k R_{N}^{(k)}\Big(4N+2^{\frac{4}{3}}N^{\frac{1}{3}}u_1, \ldots,  4N+2^{\frac{4}{3}}N^{\frac{1}{3}}u_k\Big)= \det \big[K_{\mathrm{Airy}}(\pi;u_{i},u_{j})\big]_{i,j = 1}^{k}
  \end{equation}
  and
   \begin{align}
  \lim_{N \to \infty}\mathbb{P}\Big(2^{-\frac{4}{3}}N^{-\frac{1}{3}}\big(\lambda_{\mathrm{max}}-4N\big) \leq   x\Big)=  F_m(\pi;x). \label{Fm}
  \end{align}

  \item[ (ii)] 

  For some $1\leq m\leq n$,  with $1<\theta <\frac{1}{2}(\tau +\frac{1}{\tau})$ set  \begin{equation}
  h(\theta)   =   \frac{  (\theta -\tau) ( \frac{1}{\tau}-\theta) }{\sqrt{(\frac{1}{\tau}-\tau)(\tau +\frac{1}{\tau}-2\theta)}}, \quad x(\theta)= \frac{(\frac{1}{\tau}-\tau)^2 }{ (\theta -\tau) ( \frac{1}{\tau}-\theta) } \label{xtheta},
    \end{equation}
    and \begin{equation}
     \sigma_k= \tau \theta +         \tau   h(\theta)
 \frac{\pi_{k}}{\sqrt{N}},  \quad  k= 1, \ldots,  m.
    \end{equation}
    When  $\pi_1, \ldots, \pi_m$  are in a compact subset of $\mathbb{R}$  and $\sigma_{m+1},\ldots,\sigma_n$       are in a compact subset of $(\tau \theta, \infty)$,
    \begin{multline}
   \lim_{N \to \infty}  \Big(\frac{1-\tau^2}{\tau h(\theta)} N^{\frac{1}{2}}\Big)^k R_{N}^{(k)}\Big( Nx(\theta) + \frac{1-\tau^2}{\tau h(\theta)}  N^{\frac{1}{2}} u_1, \ldots,  Nx(\theta) + \frac{1-\tau^2}{\tau h(\theta)}  N^{\frac{1}{2}} u_k \Big)\\
   = \det \big[K_{\mathrm{GUE}}(\pi;u_{i},u_{j})\big]_{i,j = 1}^{k}
  \end{multline}
  and
   \begin{align}
  \lim_{N \to \infty}\mathbb{P}\Big(
  \frac{\tau h(\theta)  }{1-\tau^2}  N^{-\frac{1}{2}}
   \big(\lambda_{\mathrm{max}}- Nx(\theta)\big) \leq   x\Big)=  G_m(\pi;x). \label{Gm}
  \end{align}
  \end{enumerate}
  \end{thm}

By Slutsky's theorem,   it immediately  follows  from \eqref{Fm} and \eqref{Gm} that
 \begin{cor} \label{largestlimit}
 Under the same assumptions and notations  as in  Theorem \ref{edge limit},   $\lambda_{\mathrm{max}}/N$  converges to $4$  and $x(\theta)$  in probability,  respectively,   in Case (i)  and  Case (ii).
  \end{cor}

Next, we prove Theorem \ref{bulk limit} and   Theorem \ref{edge limit} by the method of steepest descent  for  approximating integrals.

  \begin{proof}[Proof of Theorem \ref{bulk limit}]
    In order to   prove the sine kernel in the bulk we need to  obtain asymptotic behavior for the three sub-kernels $S_N, DS_N$ and $IS_N$. Comparing the integrands associated with  those  kernels, without loss of generality,   we will  focus on  $S_N$ and give a detailed  analysis  upon the representation    \eqref{Sleading}.

    \textbf{Step 1:  Choice of contours.}

For  convenience, let's  introduce a parametrization representation for the spectral variable  in the bulk regime  by $x = 4\cos^{2}\theta, \theta \in (0, \pi /2)$. In this case,    the phase function $f(x;z)$ defined   in  \eqref{phase} is equal to  $f_{\tau, \theta}(z)$  in \eqref{phase-var}, and correspondingly  the    saddle points  of $f_{\tau, \theta}(z)$   are expressed as \eqref{saddle-var}.  We can choose integration contours of $I_N$ and $J_N$ defined in  \eqref{SleadingI}  and \eqref{SleadingJ} as follows.

    For the  $w$-contour,     put   \begin{align*}
       l_{\min}:= \min\Big\{\frac{1}{2}\big(\frac{1}{\tau} +\tau\big), \frac{\sigma_{1}}{\tau}, \ldots, \frac{\sigma_{n}}{\tau}\Big\},
          \end{align*}
let $\mathcal{C}_{1/\tau, \theta}$ be  the circle  defined in \eqref{Ctau}, we will decide the $w$-contour on whether  $l_{\min}$  lies inside   $\mathcal{C}_{1/\tau, \theta}$ or not.

\textbf{Case 1}:  $l_{\min}$ is   outside  or on  $\mathcal{C}_{1/\tau, \theta}$.   We choose $w_1 \in (1,l_{\min})$ and two points $w_2, w_3 \in \mathcal{C}_{1/\tau, \theta}$ such that $\Im w_{2}$ is  a sufficiently small  positive number and $w_3$ is  given by
     \begin{equation*}  w_3=\frac{1}{\tau} + \frac{1}{2\cos\theta} \big(\frac{1}{\tau}-\tau \big) e^{i \frac{\pi}{4}}.\end{equation*}
    At this time  let's define three  curves
   \begin{align*}
      \Sigma_{+}^{1} &= \left\{w:  \Re w= w_{1},  0 \leq \Im w \leq \Im w_2\right\} \cup \left\{w: \Im w= \Im w_2,  w_1 \leq \Re w \leq \Re w_2\right\},\\
      \Sigma_{+}^{2} &= \left\{w \in \mathcal{C}_{\frac{1}{\tau}, \theta}: \Re w_2 \leq \Re{w} \leq \Re w_3, \Im w>0\right\},\\
       \Sigma_{+}^{3} &=  \left\{w: \Im w= \Im w_3,  \Re w_3 \leq \Re w \leq w_4\right\} \cup \left\{w:  \Re w= w_{4},  0 \leq \Im w \leq \Im w_3\right\},
    \end{align*}  where
     \begin{align*}
       w_{4} &=1+ \max\Big\{ \frac{1}{\tau} +\frac{1}{2\cos\theta}\big(\frac{1}{\tau} - \tau\big), \frac{\sigma_{1}}{\tau}, \ldots, \frac{\sigma_{n}}{\tau}\Big\}.
          \end{align*}

          \textbf{Case 2}:  $l_{\min}$ is   inside   $\mathcal{C}_{1/\tau, \theta}$.  We can always choose the  point $w_1$ inside the circle such that  $w_1 \in (1,l_{\min})$.  Further, we denote    the intersection of the vertical line with real part  $w_1$ and  $\mathcal{C}_{1/\tau, \theta}$ by  the same notation   $w_{2}$, for simplicity,    and change the curve    $\Sigma_{+}^{1}$ to  the  line segment joining $w_1$ and $w_2$
 \begin{align*}
      \Sigma_{+}^{1} = \left\{w:  \Re w= w_{1},  0 \leq \Im w \leq \Im w_2\right\}.
    \end{align*}

In both cases we choose the contour  $\Sigma$ for integrals $I_N$ and $J_N$,  defined in \eqref{SleadingI}  and \eqref{SleadingJ}, as $\Sigma=\Sigma_{+}^{1} \cup \Sigma_{+}^{2}\cup \Sigma_{+}^{3} \cup \overline{\Sigma_{+}^{1}} \cup \overline{\Sigma_{+}^{2}} \cup \overline{\Sigma_{+}^{3}} $
where the bar denotes the complex conjugate  operation.  Given  a local curve  $\Sigma^{\mathrm{local}}$   around $z_{+}$ and $z_{-}$  as described in \eqref{nocri-bulk-sig}, set
$\Sigma^{\mathrm{global}} = \Sigma \backslash \Sigma^{\mathrm{local}}$,  we can also divide  it into two parts $\Sigma =\Sigma^{\mathrm{global}}  \cup \Sigma^{\mathrm{lobal}}$.

 For the  $z$-contour,   choose $z_1 \in (1, w_1)$  and let $z_2$  be the intersection of the vertical line $x=1$   and  $\mathcal{C}_{\tau, \theta}$  in the upper half  plane. Defining  three curves
    \begin{align*}
      \mathcal{C}^{1} &= \left\{z=tz_1 +(1-t) z_{2} \,\mathrm{or}\, z= tz_1 +(1-t) \bar{z_{2}}:  0 \leq t<1\right\}, \\
      \mathcal{C}^{2} &= \left\{z \in \mathcal{C}_{\tau, \theta} :  \Re{z}>1\right\},\\
      \mathcal{C}^{3} &= \left\{z \in \mathbb{C}: \left|\frac{\tau z - 1}{z - \tau}\right| = \left|\frac{\tau z_2- 1}{z_2 - \tau}\right|, \Re{z} \leq 1\right\},
    \end{align*}
   we first choose $\mathcal{C}^{3} \cup \mathcal{C}^{1}$ as the $z$-contour $ \mathcal{C}$ in \eqref{SleadingI} and \eqref{SleadingJ}. Then by introducing a virtual curve   $\mathcal{C}^{2}$ and using it two times but with different orientation, we deform  $\mathcal{C}$ into the difference of two contours   $\widetilde{\mathcal{C}} = \mathcal{C}^{2}\cup \mathcal{C}^{3}$ and
   $\mathcal{C}^{1}\cup \mathcal{C}^{2}$.
Given  a local curve  $\mathcal{C}^{\mathrm{local}}$     around $z_{+}$ and $z_{-}$  as described  in  \eqref{nocri-bulk-c}, set
$\mathcal{C}^{\mathrm{global}} = \widetilde{\mathcal{C}}  \backslash \mathcal{C}^{\mathrm{local}}$,  then
$\widetilde{\mathcal{C}}=\mathcal{C}^{\mathrm{global}}  \cup  \mathcal{C}^{\mathrm{local}}$.     See Figure \ref{bulkfig} for an illustration.

    \begin{figure}[h]
      \centering
      \begin{tikzpicture}
        \draw (-3.5,0) -- (7,0);
        \draw (1.5,0.6) node[right] {$\scriptstyle \mathcal{C}^{2}$};
        \draw[domain=-0.42*pi:0.42*pi, dashed] plot({0.4 + 1.25*cos(\x r)} ,{1.25*sin(\x r)});
        \draw[domain=-pi:-0.19*pi] plot({-1 + 2.1*cos(\x r)} ,{2.1*sin(\x r)});
        \draw[domain=0.19*pi:pi] plot({-1 + 2.1*cos(\x r)}, {2.1*sin(\x r)});
        \draw (-0.8,1.6) node[left] {$\scriptstyle \mathcal{C}^{3}$};

        \draw ({0.4 + 1.25*cos(0.42*pi r)} ,{1.25*sin(0.42*pi r)}) -- (0.85, 0);
        \draw ({0.4 + 1.25*cos(0.42*pi r)} ,{-1.25*sin(0.42*pi r)}) -- (0.85, 0);
        \draw (0.8,0.6) node[left] {$\scriptstyle \mathcal{C}^{1}$};

        \draw[thick] ({2.1 - 1.25*cos(0.15*pi r)} ,{1.25*sin(0.15*pi r)}) -- ({2.1 - 1.25*cos(0.15*pi r)} ,{-1.25*sin(0.15*pi r)});
        \draw (0.85,0.3) node[right] {$\scriptstyle \Sigma_{+}^{1}$};
        \draw[domain=0.15*pi:0.8*pi, thick] plot({2.1 - 1.25*cos(\x r)} ,{1.25*sin(\x r)});
        \draw[domain=-0.8*pi:-0.15*pi,thick] plot({2.1 - 1.25*cos(\x r)} ,{1.25*sin(\x r)});
        \draw (2,1.2) node[above] {$\scriptstyle \Sigma_{+}^{2}$};
        \draw[thick] ({2 + 1.25*cos(0.15*pi r)},{ 1.25*sin(-0.8*pi r)}) -- (6,{1.25*sin(-0.8*pi r)});
        \draw (5,0.7) node[above] {$\scriptstyle \Sigma_{+}^{3}$};
        \draw[thick] ({2 + 1.25*cos(0.15*pi r)},{1.25*sin(0.8*pi r)}) -- (6,{1.25*sin(0.8*pi r)});
        \draw (6,0.4) node[right] {$\scriptstyle \Sigma_{+}^{4}$};
        \draw[thick] (6,{1.25*sin(0.8*pi r)}) -- (6,{-1.25*sin(0.8*pi r)});

        \filldraw
          (-0.7,0) circle (0.7pt) node[below] {$\scriptstyle -1$}
          (0,0) circle (0.7pt) node[below] {$\scriptstyle 0$}
          (0.4,0) circle (0.7pt) node[below] {$\scriptstyle \tau$}
          (0.7,0) circle (0.7pt) node[below] {$\scriptstyle 1$}
          (5/4,0) circle (0.7pt) node[below] {$\scriptstyle z_{0}$}
          (2.1,0) circle (0.7pt) node[below] {$\scriptstyle \frac{1}{\tau}$}
          (3.1,0) circle (0.7pt) node[below] {$\scriptstyle \frac{\sigma_{1}}{\tau}$}
          (4.1,0) circle (0.7pt) node[below] {$\cdots$}
          (5.1,0) circle (0.7pt) node[below] {$\scriptstyle \frac{\sigma_{n}}{\tau}$};
      \end{tikzpicture}
      \caption{Contours of double integrals for $S_N$: bulk}\label{bulkfig}
    \end{figure}
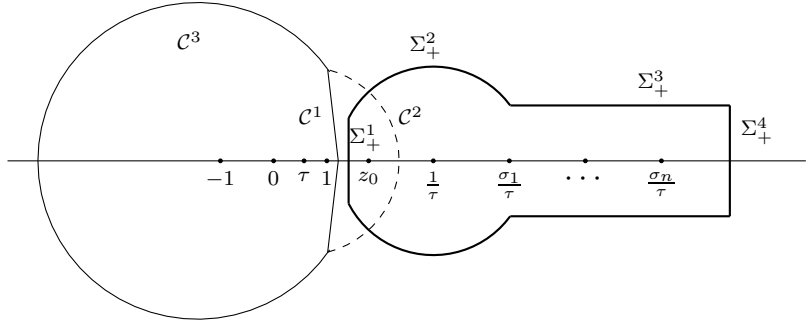

    \textbf{Step 2:  Estimates of $I_{N}$ and $J_{N}$.}

To estimate the sub-kernel  $S_N$ given  in \eqref{Sleading}, we need to split the integrals $I_N$ and $J_N$, which are defined in  \eqref{SleadingI} and \eqref{SleadingJ} respectively,  into four parts,  and prove  which parts are dominant. For this, let
    \begin{align*}
      I_{N}&=\Big(-\mathrm{P.V.} \int_{\mathcal{C}^1 \cup    \mathcal{C}^{2} } \frac{dz}{2\pi i}  \int_{\Sigma} \frac{dw}{2\pi i} + \mathrm{P.V.} \int_{\mathcal{C}^{\mathrm{local}}} \frac{dz}{2\pi i}  \int_{\Sigma^{\mathrm{local}}} \frac{dw}{2\pi i} + \int_{\mathcal{C}^{\mathrm{local}}} \frac{dz}{2\pi i}  \int_{\Sigma^{\mathrm{global}}} \frac{dw}{2\pi i} \nonumber\\
      &\quad + \int_{\mathcal{C}^{2}  \backslash \mathcal{C}^{\mathrm{local}}} \frac{dz}{2\pi i}  \int_{\Sigma} \frac{dw}{2\pi i}\Big)  \frac{\eta_{-}}{\varphi(x)}  e^{N(f(x;z) - f(x;w))}  e^{\frac{\eta_{-}}{\varphi(x)}( v(z - \Re{z_{-}})-u(w - \Re{z_{-}}))} H(z, w) \nonumber\\
      &=: -I_{N}^{1} + I_{N}^{2} + I_{N}^{3} + I_{N}^{4},
    \end{align*}
   and  \begin{align*}
      J_{N} &= \frac{1}{\sqrt{2\pi \eta_{-}(xN+\frac{v}{\varphi(x)})}}  \Big(- \mathrm{P.V.} \int_{\mathcal{C}^1 \cup    \mathcal{C}^{2} } \frac{dz}{2\pi i}  \int_{\Sigma} \frac{dw}{2\pi i}+ \mathrm{P.V.} \int_{\mathcal{C}^{\mathrm{local}}} \frac{dz}{2\pi i}  \int_{\Sigma^{\mathrm{local}}} \frac{dw}{2\pi i}  \nonumber\\
      &\quad + \int_{\mathcal{C}^{\mathrm{local}}} \frac{dz}{2\pi i}  \int_{\Sigma^{\mathrm{global}}} \frac{dw}{2\pi i} + \int_{\mathcal{C}^{\mathrm{global}}} \frac{dz}{2\pi i}  \int_{\Sigma} \frac{dw}{2\pi i}\Big) \frac{\eta_{-}}{\varphi(x)} e^{N(f(x;z) - f(x;w))}  \nonumber\\
      &\quad\times  e^{x(1-z)\eta_{-}N+\frac{\eta_{-}}{\varphi(x)}( v(1 - \Re{z_{-}})-u(w - \Re{z_{-}})) } \frac{\sqrt{z^{2}-1}}{1-z^{2}} H(z, w) \Big(1 + \bo\big(\frac{1}{N}\big)\Big) \nonumber\\
      &=: -J_{N}^{1} + J_{N}^{2} + J_{N}^{3} + J_{N}^{4},
    \end{align*}
where all the contours are positively oriented.

   For  $I_{N}^{1}$,  since  the  $z$-integrand  has  no other poles  except for $z=w$,    applying  the residue theorem  yields
    \begin{align}
      I_{N}^{1} &= - \int_{z_{-}}^{z_{+}}  \frac{dw }{2\pi i }  \frac{\eta_{-}}{\varphi(x)}   e^{\frac{\eta_{-}}{\varphi(x)}(v-u)(w - \Re{z_{-}})}  = -  \frac{ \sin \pi (u-v)}{\pi (u-v)}. \label{nocri-bulk-i11}
    \end{align}
   Similarly,    apply  the residue theorem  and then by Proposition \ref{g-asym} we obtain for  large  $N$
    \begin{align*}
      \abs{J_{N}^{1}} &= \frac{1}{\sqrt{2\pi \eta_{-}(4N+\frac{v}{\varphi(x)})}} \frac{\eta_{-}}{\varphi(x)} \nonumber \\
      &\quad \times \abs{\int_{z_{-}}^{z_{+}} \frac{dw}{2\pi i} e^{x(1-w)\eta_{-}N+\frac{\eta_{-}}{\varphi(x)}(v(1 - \Re{z_{-}})-u(w - \Re{z_{-}})) } \frac{1}{\sqrt{w^{2}-1}} \Big(1 + \bo\big(\frac{1}{N}\big)\Big)}  \nonumber\\
      &\leq  e^{-\frac{(1-\tau)^2}{4\tau}x\eta_{-}N}
    \end{align*}
    uniformly for  $u, v$ in a compact subset of $\mathbb{R}$.  Here in the last  inequality we have changed $x$ to $x/2$ in order to get an upper bound.

    Next, we   turn to analyze the remaining parts of $I_{N}$ and $J_{N}$ by the method of steepest descent.  At this time, we need to  establish some  estimates   of $\Re{f(x;z)}$ when $z$ belongs to some chosen curves.  We will use the same notations $ \eps_{1}$ and  $\eps_{2}$ to denote different constants,   for simplicity. Also, we will  only use     $ \Re{f(z_{-})}$ since
    $ \Re{f(z_{+})}=  \Re{f(z_{-})}$.

    The required  facts  are listed as  follows when the spectral variable $x\in (0,4)$.

    \begin{enumerate}
      \item By statement \ref{S1} in Lemma \ref{extr-f}, for some small $\eps_{1} > 0$ there exists $\delta_{1} > 0$ such that
          \begin{align}
            \min_{w \in ({\Sigma_{+}^{2} \cup \overline{\Sigma_{+}^{2}}})  \backslash \Sigma^{\mathrm{local}}} \Re{f(x;w)} - \Re{f(x;z_{-})} \geq \eps_{1}, \label{nocri-bulk-sig1}
          \end{align}
          where \begin{align}
            \Sigma^{\mathrm{local}} = \left\{w \in \Sigma_{+}^{2}\cup \overline{\Sigma_{+}^{2}}: \abs{w - z_{+}} < \delta_{1} \,\text{or}\, \abs{w - z_{-}} < \delta_{1}\right\} . \label{nocri-bulk-sig}
          \end{align}

      \item By statement \ref{S2} and   statement \ref{S3} in Lemma \ref{extr-f} in Case 1 for the choice of  $\Sigma_{+}^{1}$  in Step 1,  or by statement \ref{S3} in Lemma \ref{extr-f} in Case 2, $\Re{f(x;w)}$ is strictly decreasing along with $\Sigma_{+}^{1}$ as $w$ approaches  $w_2$ moving along  $\Sigma_{+}^{1}$. Thus
          \begin{align}
            \min_{w \in \Sigma_{+}^{1}\cup \overline{\Sigma_{+}^{1}}} \Re{f(x;w)} - \Re{f(x;z_{-})}= \Re{f(x;w_2)} - \Re{f(x;z_{-})}\geq \eps_{1}. \label{nocri-bulk-sig2}
          \end{align}
      \item By statement \ref{S2} in Lemma \ref{extr-f}, $\Re{f(x;w)}$ is strictly increasing along  $\Sigma_{+}^{3}$ as $\Re{w}$ increases, so
          \begin{align}
            \min_{w \in  \Sigma_{+}^{3}\cup \overline{\Sigma_{+}^{3}}} \Re{f(x;w)} - \Re{f(x;z_{-})} \geq \eps_{1}. \label{nocri-bulk-sig3}
          \end{align}
      \item  By statement \ref{S3} in Lemma \ref{extr-f}, $\Re{f(w)}$ is strictly increasing along  $\Sigma_{+}^{4}$ as $\Im{w}$ increases, so
          \begin{align}
            \min_{w \in \Sigma_{+}^{4}\cup \overline{\Sigma_{+}^{4}}} \Re{f(x;w)} - \Re{f(x;z_{-})} \geq \Re{f(x;w_{4})} - \Re{f(x;z_{-})} \geq \eps_{1}. \label{nocri-bulk-sig4}
          \end{align}
      \item By statement \ref{S1} in Lemma \ref{extr-f}, for some small $\eps_{2} > 0$ there exists $\delta_{2} > 0$ such that
          \begin{align}
            \max_{z \in \mathcal{C}^{2} \backslash \mathcal{C}^{\mathrm{local}}} \Re{f(x;z)} - \Re{f(x;z_{-})} \leq - \eps_{2}, \label{nocri-bulk-c1}
          \end{align}
          where
          \begin{align}
            \mathcal{C}^{\mathrm{local}} = \left\{z \in \mathcal{C}^{2}: \abs{z - z_{+}} < \delta_{2} \,\text{or}\, \abs{z - z_{-}} < \delta_{2}\right\}. \label{nocri-bulk-c}
          \end{align}
      \item By definition of  $\mathcal{C}^{3}$,  $\Re{f_{1}(x;z)}$ is a constant for all $z \in \mathcal{C}^{3}$,  so
          \begin{align}
            \max_{z \in \mathcal{C}^{3}} \Re{f_{1}(x;z)} - \Re{f(x;z_{-})} = \Re{f(x;z_2)} - \Re{f(x;z_{-})} \leq -\eps_{2}. \label{nocri-bulk-c3}
          \end{align}
    \end{enumerate}

    Take the Taylor expansion of  $f(x;z)$ at $z_{+}$ and $z_{-}$
    \begin{align*}
    f(x;z) &=f(x;z_{\pm}) + \frac{1}{2} f''(x;z_{\pm})(z - z_{\pm})^{2} + \bo((z - z_{\pm})^{3}),
    \end{align*}
    and note that   near the two points  $(z, w)=(z_{+},z_{+})$     or $(z_{-},z_{-})$
    \begin{align*}
     H(z,w)&=        \frac{1 }{(w-z_{\pm})-(z-z_{\pm})} \Big(1+\bo(z - z_{\pm})  +\bo(w - z_{\pm})\Big).
        \end{align*}
    we  see  that after shifting  the Cauchy principal integral
    \begin{align*}
      I_{N}^{2} &= \frac{1}{\sqrt{N}} \mathrm{P.V.} \int_{N^{\frac{1}{2}}(\mathcal{C}^{\mathrm{local}} - z_{\pm})} \frac{dz}{2\pi i}  \int_{N^{\frac{1}{2}}(\Sigma^{\mathrm{local}} - z_{\pm})} \frac{dw}{2\pi i} e^{\frac{1}{2}  f''(x;z_{\pm}) (z^{2} - w^{2}) }  \nonumber\\
      &\quad\times e^{\frac{\eta_{-}}{\varphi(x)}(v-u)(z_{\pm}-\Re{z_{-}})}   \frac{1 }{w-z} \frac{\eta_{-}}{\varphi(x)} \Big(1+\bo\big(\frac{1}{\sqrt{N}}\big) \Big)\nonumber\\
      &= \bo(N^{-\frac{1}{2}})
    \end{align*}
    uniformly for  $u, v$ in a compact subset of $\mathbb{R}$.
  Noting that $\Re\{1-z\}\leq 0$ when $z\in \mathcal{C}^{\mathrm{local}}$,  we can first  remove the factor $e^{x(1-z)\eta_{-}N}$ and then proceed in a similar manner as for $ I_{N}^{2}$ to obtain an upper bound for $J_{N}^{2}$
    \begin{align*} |J_{N}^{2}| =  \bo(N^{-1}).
    \end{align*}

    Together with  \eqref{nocri-bulk-sig1}, \eqref{nocri-bulk-sig2}, \eqref{nocri-bulk-sig3} and \eqref{nocri-bulk-sig4},  we obtain\begin{align*}
      \min_{w \in \Sigma^{\mathrm{global}}} \Re{f(x;w)} - \Re{f(x;z_{-})} \geq \eps_{1},
    \end{align*}
    from which   it is easily to find that
     \begin{align*}
      \abs{I_{N}^{3}} \leq e^{-\frac{1}{2} N \eps_{1}}, \quad
      \abs{J_{N}^{3}}  \leq e^{-\frac{1}{2} N \eps_{1}}
    \end{align*}
    when $N$ is sufficiently  large.
Similarly, combination of \eqref{nocri-bulk-c1} and \eqref{nocri-bulk-c3} yields
    \begin{align*}
      \max_{z \in \mathcal{C}^{\mathrm{global}}} \Re{f(x;z)} - \Re{f(x;z_{-})} \leq - \eps_{2},
    \end{align*}
  from which    one finds
    \begin{align*}
      \abs{I_{N}^{4}}  \leq e^{-\frac{1}{2} N \eps_{2}}, \quad
      \abs{J_{N}^{4}}  \leq e^{-\frac{1}{2} N \eps_{2}}
    \end{align*}
    for sufficiently large $N$.

    So, in summary,   we have obtained
    \begin{align*}
     e^{\frac{\eta_{-}}{\varphi(x)}(\Re{z_{-}}+1)( u-v)} & \frac{1}{\varphi(x)} S_{N}\Big(Nx+\frac{u}{\varphi(x)}, Nx+\frac{v}{\varphi(x)}\Big) \nonumber\\ &= \frac{\sin\pi(u-v)}{\pi(u-v)}\Big(1+\bo\big(\frac{1}{\sqrt{N}}\big) \Big)
    \end{align*}
    uniformly for  $u, v$ in a compact subset of $\mathbb{R}$.

        \textbf{Step 3:  Completion}.

    We can proceed as in  the estimates of  $I_N$ and $J_N$ for  $S_N$ to  obtain similar  approximations for  $DS_N$ and $IS_N$. However,  the crossing factor  between $z$ and $w$ in the integrand changes  from $(1-zw)/(z-w)$ to $(z-w)/(1-zw)$. Correspondingly, the dominant contribution changes  from the  the Taylor expansion     near the two points  $(z_{+},z_{+})$ and $(z_{-},z_{-})$ to  $(z_{+},z_{-})$ and $(z_{-},z_{+})$. After detailed analysis as in Step 2 we have
    \begin{align*}
     e^{\frac{\eta_{-}}{\varphi(x)}(\Re{z_{-}}+1)( u+v)} & \frac{1}{\varphi(x)} DS_{N}\Big(Nx+\frac{u}{\varphi(x)}, Nx+\frac{v}{\varphi(x)}\Big)  =  \bo\big(\frac{1}{N}\big)
    \end{align*}
    and
     \begin{align*}
     e^{-\frac{\eta_{-}}{\varphi(x)}(\Re{z_{-}}+1)( u+v)} & \frac{1}{\varphi(x)} IS_{N}\Big(Nx+\frac{u}{\varphi(x)}, Nx+\frac{v}{\varphi(x)}\Big)  =  \bo\big(\frac{1}{N}\big)
    \end{align*}
    uniformly for  $u, v$ in a compact subset of $\mathbb{R}$.

   As $N\to \infty$, the factors  associated with  $DS_N$ and $IS_N$ tend  to zero and thus the Pfaffian reduces to a determinant.  This completes  the proof of Theorem   \ref{bulk limit}.
  \end{proof}

  \begin{proof}[Proof of Theorem \ref{edge limit}:  Part (i)]
     At the soft edge $x = 4$, the phase function \eqref{phase} reads
    \begin{align*}
      f(z):=f(4;z)=4\eta_{-} (z+1)+ \log(\tau z - 1) -\log(z-\tau),
    \end{align*}
    from which the   saddle point is  given by $z_{0}$  as in \eqref{z0saddle}.  We proceed in a similar manner as in the proofs of Theorems  \ref{edgecritthm} and \ref{largestcritthm}.

    \textbf{Step 1:  Choice of contours.}

    Set  $\pi_{*} =  \min\{0,  \pi_{1}, \ldots,  \pi_{m} \}$, recalling $\eta_{-}=\tau/(1-\tau^2)$,  let
    \begin{align*}
      z_{1} &= z_{0} + 2^{-\frac{4}{3}} N^{-\frac{1}{3}} \eta_{-}^{-1} (\pi_{*} - 2), \quad w_{1} = z_{0} + 2^{-\frac{4}{3}} N^{-\frac{1}{3}} \eta_{-}^{-1} (\pi_{*} - 1),
          \end{align*}
          and
       \begin{align*} z_{2} &= z_{1} +\eta_{-}^{-1} \delta e^{i \frac{2}{3} \pi},  \quad w_{2} = w_{1} + \eta_{-}^{-1} \delta e^{i \frac{1}{3} \pi}
      \end{align*}
      where $\delta = N^{-\frac{1}{12}}$ can be  chosen such that \eqref{airy-w1} and \eqref{airy-z1} hold. With two curves $\mathcal{C}_{\tau, 0}$ and  $\mathcal{C}_{1/\tau, 0}$ in mind, which are   defined  in  \eqref{Ctau} and  \eqref{C/tau} respectively,  take   $z_{3}, z_{4} \in \mathcal{C}_{\tau, 0}$ and $w_{3}, w_{4} \in \mathcal{C}_{1/\tau, 0}$ such that
      $$\Re{z_{3}} = \Re{z_{2}}, \Re{w_{3}} = \Re{w_{2}},  \Im{z_{3}}>0,  \Im{w_{3}}>0,$$
      and
      $$   z_{4} =1+i \frac{1 - \tau}{2\tau} \sqrt{(1+3\tau) (1-\tau)}, \ w_{4} = \frac{1}{\tau} + \frac{1}{2} \big( \frac{1}{\tau}-\tau\big) e^{i\frac{\pi}{4}}.$$

  We are ready to construct $z$- and $w$- contours as follows.  Note that  $\mathcal{Q}_{\tau,z_4}$, which is defined in    \eqref{circletau},    stands on the left of   the saddle point $z_{0}$,  we can  choose  it  as     the $z$-contour  $ \mathcal{C}$  of integrals $I_N$ and $J_N$  in \eqref{SleadingI}  and \eqref{SleadingJ}. Then we need to deform the $z$-contour  $\{z \in \mathcal{Q}_{\tau,z_4}: \Re{z} > 1\}$ in  $I_N$ into $\mathcal{C}^{\mathrm{local}}\cup \mathcal{C}^{\mathrm{glocal}}$, where   $\mathcal{C}^{\mathrm{local}} = \mathcal{C}_{+}^{\mathrm{local}} \cup \overline{\mathcal{C}_{+}^{\mathrm{local}}}$ and $\mathcal{C}^{\mathrm{global}} =  \mathcal{C}_{+}^{1} \cup \mathcal{C}_{+}^{2} \cup \overline{\mathcal{C}_{+}^{2}} \cup \overline{\mathcal{C}_{+}^{1}}$ with
    \begin{align*}
      \mathcal{C}_{+}^{\mathrm{local}} &= \left\{z = z_{1} + r e^{i\frac{2\pi}{3}}: 0 \leq r < \eta_{-}^{-1}\delta\right\},\\
      \mathcal{C}_{+}^{1} &= \left\{z = \Re{z_{2}} + iy: \Im{z_{2}} < y < \Im{z_{3}}\right\},\\
      \mathcal{C}_{+}^{2} &= \left\{z \in \mathcal{C}_{\tau, 0} : 1 < \Re{z} < \Re{z_{2}}~\text{and}~\Im z \geq 0\right\}.
    \end{align*}
Also set
$$  w_{5} =  1+\max\Big\{\frac{3}{2\tau}-\frac{1}{2}\tau,  \frac{\sigma_{m+1}}{\tau}, \ldots,  \frac{\sigma_{n}}{\tau} \Big\},$$
   and  let
    \begin{align*}
      \Sigma_{+}^{\mathrm{local}} &= \left\{w = w_{1} + r e^{i\frac{\pi}{3}}: 0 \leq r \leq \eta_{-}^{-1} \delta\right\},\\
      \Sigma_{+}^{1} &= \left\{w = \Re{w_{2}} + iy: \Im{w_{2}} < y < \Im{w_{3} }\right\},\\
      \Sigma_{+}^{2} &= \left\{w \in \mathcal{C}_{1/\tau, 0}: \Re{w_{2}} < \Re{w} < \Re{w_{4}} ~\text{and}~ \Im{w} \geq 0\right\},\\
      \Sigma_{+}^{3} &= \left\{w = r + i \Im{w_{4}}: \Re{w_{4}} \leq r < w_{5}\right\},\\
      \Sigma_{+}^{4} &= \left\{w= w_{5} + iy: 0 \leq y \leq \Im{w_{4}}\right\},
    \end{align*}
we deform the $w$-contour of  $I_N$ and $J_N$  in \eqref{SleadingI}  and \eqref{SleadingJ} into the union of       $\Sigma^{\mathrm{local}} := \Sigma_{+}^{\mathrm{local}} \cup \overline{\Sigma_{+}^{\mathrm{local}}}$ and $\Sigma^{\mathrm{global}}: = \overline{\Sigma_{+}^{1}} \cup \overline{\Sigma_{+}^{2}} \cup \overline{\Sigma_{+}^{3}} \cup \overline{\Sigma_{+}^{4}} \cup \Sigma_{+}^{4} \cup \Sigma_{+}^{3} \cup \Sigma_{+}^{2} \cup \Sigma_{+}^{1}$.
 See Figure \ref{softfig} for an illustration.
    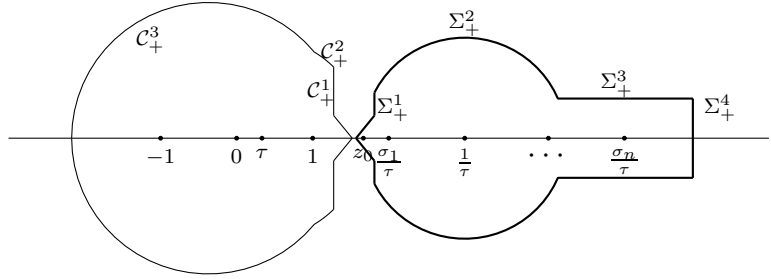
\begin{figure}[h]
      \centering
      \begin{tikzpicture}
        \draw (-3,0) -- (7,0);
        \draw (1.52, 0) -- ({1/3 + (4/3)*cos(pi/4 r)}, 0.3);
        \draw ({1/3 + (4/3)*cos(pi/4 r)} ,{(4/3)*sin(pi/4 r)}) -- ({1/3 + (4/3)*cos(pi/4 r)} ,0.3);
        \draw (1.52, 0) -- ({1/3 + (4/3)*cos(pi/4 r)}, -0.3);
        \draw ({1/3 + (4/3)*cos(pi/4 r)} ,{-(4/3)*sin(pi/4 r)}) -- ({1/3 + (4/3)*cos(pi/4 r)} ,-0.3);
        \draw (1.45,0.6) node[left] {$\scriptstyle \mathcal{C}_{+}^{1}$};
        \draw (1.3,0.8) node[above] {$\scriptstyle \mathcal{C}_{+}^{2}$};
        \draw[domain=pi/4:0.33*pi] plot({1/3 + (4/3)*cos(\x r)} ,{(4/3)*sin(\x r)});
        \draw[domain=-0.33*pi:-pi/4] plot({1/3 + (4/3)*cos(\x r)} ,{(4/3)*sin(\x r)});
        \draw[domain=0.22*pi:pi] plot({-0.37 + 1.8*cos(\x r)} ,{1.8*sin(\x r)});
        \draw[domain=-pi:-0.22*pi] plot({-0.37 + 1.8*cos(\x r)} ,{1.8*sin(\x r)});
        \draw (-0.8,1.3) node[left] {$\scriptstyle \mathcal{C}_{+}^{3}$};
        \draw[thick] (1.57, 0) -- ({3 - (4/3)*cos(0.15*pi r)}, 0.3);
        \draw[thick] ({3 - (4/3)*cos(0.15*pi r)} ,{(4/3)*sin(0.15*pi r)}) -- ({3 - (4/3)*cos(0.15*pi r)}, 0.3);
        \draw[thick] (1.57, 0) -- ({3 - (4/3)*cos(0.15*pi r)}, -0.3);
        \draw[thick] ({3 - (4/3)*cos(0.15*pi r)}, -0.3) -- ({3 - (4/3)*cos(0.15*pi r)} ,{-(4/3)*sin(0.15*pi r)});
        \draw (1.7,0.4) node[right] {$\scriptstyle \Sigma_{+}^{1}$};
        \draw[domain=0.15*pi:pi-0.4,thick] plot({3 - (4/3)*cos(\x r)} ,{(4/3)*sin(\x r)});
        \draw[domain=-pi+0.4:-0.15*pi,thick] plot({3 - (4/3)*cos(\x r)} ,{(4/3)*sin(\x r)});
        \draw (3,1.2) node[above] {$\scriptstyle \Sigma_{+}^{2}$};
        \draw[thick] ({3 + (4/3)*cos(pi-0.46 r)},{(4/3)*sin(pi-0.46 r)}) -- (6,{(4/3)*sin(pi-0.46 r)});
        \draw (5,0.4) node[above] {$\scriptstyle \Sigma_{+}^{3}$};
        \draw[thick] ({3 + (4/3)*cos(pi-0.46 r)},{-(4/3)*sin(pi-0.46 r)}) -- (6,{-(4/3)*sin(pi-0.46 r)});
        \draw (6,0.4) node[right] {$\scriptstyle \Sigma_{+}^{4}$};
        \draw[thick] (6,{(4/3)*sin(pi-0.46 r)}) -- (6,{-(4/3)*sin(pi-0.46 r)});
        \filldraw
          (-1,0) circle (0.7pt) node[below] {$\scriptstyle -1$}
          (0,0) circle (0.7pt) node[below] {$\scriptstyle 0$}
          (1/3,0) circle (0.7pt) node[below] {$\scriptstyle \tau$}
          (1,0) circle (0.7pt) node[below] {$\scriptstyle 1$}
          (5/3,0) circle (0.7pt) node[below] {$\scriptstyle z_{0}$}
          (3,0) circle (0.7pt) node[below] {$\scriptstyle \frac{1}{\tau}$}
          (2,0) circle (0.7pt) node[below] {$\scriptstyle \frac{\sigma_{1}}{\tau}$}
          (4.1,0) circle (0.7pt) node[below] {$\cdots$}
          (5.1,0) circle (0.7pt) node[below] {$\scriptstyle \frac{\sigma_{n}}{\tau}$};
      \end{tikzpicture}
      \caption{Contours of double integrals for $S_N$: soft}\label{softfig}
    \end{figure}

 \textbf{Step 2:  Estimates of $I_{N}$ and $J_{N}$.}

  Let   $\varphi(4)=2^{-\frac{4}{3}}N^{-\frac{1}{3}}$, rewrite $I_{N}$ and $J_{N}$ as
    \begin{align*}
      I_{N} &= \bigg(\int_{\mathcal{C}^{\mathrm{local}}} \frac{dz}{2 \pi i} \int_{\Sigma^{\mathrm{local}}} \frac{dw}{2 \pi i} + \int_{\mathcal{C}^{\mathrm{local}}} \frac{dz}{2 \pi i} \int_{\Sigma^{\mathrm{global}}} \frac{dw}{2 \pi i} + \int_{\mathcal{C}^{\mathrm{global}}} \frac{dz}{2 \pi i} \int_{\Sigma^{\mathrm{lobal}}\cup \Sigma^{\mathrm{global}}} \frac{dw}{2 \pi i}\bigg) \nonumber\\
      &\quad \frac{\eta_{-}}{\varphi(4)} e^{N(f(z) - f(w))}  e^{\frac{\eta_{-}}{\varphi(4)}( v(z - z_{0})-u(w - z_{0}))} H(z, w) \nonumber \\
      &=: I_{N}^{1} + I_{N}^{2} + I_{N}^{3},
    \end{align*}
    and
    \begin{align*}
      J_{N} &= \frac{1}{\sqrt{2\pi \eta_{-}(4N+\frac{v}{\varphi(4)})}} \frac{\eta_{-}}{\varphi(4)} \int_{\mathcal{Q}_{\tau,z_4}} \frac{dz}{2 \pi i} \int_{\Sigma} \frac{dw}{2 \pi i} e^{N(f_{1}(4;z) - f(w))} e^{\frac{\eta_{-}}{\varphi(4)}( v(1 - z_{0})-u(w - z_{0}))} \nonumber\\
      &\quad\times \frac{\sqrt{z^{2}-1}}{1-z^{2}} H(z, w) \Big(1 + \bo\big(\frac{1}{4\eta_{-}N+\frac{\eta_{-}v}{\varphi(4)}}\big)\Big) .
    \end{align*}
We need to tune external source parameters $\sigma_1, \ldots,\sigma_n$ near the saddle point $z_0$ in order to
   observe some nontrivial limits, as stated in Theorem \ref{edge limit}.  Take the Taylor expansion at $z_0$ and we get
   \begin{align*}
     I_{N}^{1} &= \frac{\eta_{-}}{\varphi(4)} \int_{\mathcal{C}^{\mathrm{local}}} \frac{dz}{2\pi i }  \int_{\Sigma^{\mathrm{local}}} \frac{dw}{2\pi i}  e^{\frac{N}{6}f'''(z_{0})((z-z_{0})^3-(w-z_{0})^3))}  e^{\frac{\eta_{-}}{\varphi(4)}( v(z-z_{0})- u(w-z_{0})) }  \nonumber\\
     &\quad\times  \frac{1}{z-w}  \prod_{k = 1}^{m} \frac{z-z_{0}-(16N)^{-\frac{1}{3}}\eta_{-}^{-1}\pi_{k}}{w-z_{0}-(16N)^{-\frac{1}{3}}\eta_{-}^{-1}\pi_{k}}
  \Big(1+\bo(z-z_{0})+\bo(w-z_{0})\Big).
   \end{align*}
 Noticing   $f'''(z_{0})=-32\eta_{-}^{3}$,    change variables $$z\to z_{0}+(16N)^{-\frac{1}{3}} \eta_{-}^{-1}z, \quad w\to z_{0}+ (16N)^{-\frac{1}{3}} \eta_{-}^{-1}w,$$ we  compute
   \begin{align}
     I_{N}^{1} &= \Big(1+\bo(N^{-\frac{1}{3}})\Big) K_{\mathrm{Airy}}(\pi;u,v), \label{airyIn1}
   \end{align}
    uniformly for  $u, v$ in a compact subset of $\mathbb{R}$. Moreover,
   for $u, v \geq 0$ there exists a constant $C>0$
   \begin{align}
     \abs{I_{N}^{1}}
     &\leq C e^{-av + bu},\label{airyIn1'}
   \end{align}
   where \begin{equation}a = \frac{\eta-}{\varphi(4)} (z_{0} - z_{1}) = 2- \pi_{*}, \quad b = \frac{\eta-}{\varphi(4)} (z_{0} - w_{1}) = 1 - \pi_{*}. \label{softab} \end{equation}

Next, we   turn to analysize the remaining parts of $I_{N}$ and $J_{N}$ by steepest-descent  method.  In order to   establish some  estimates   of $\Re{f(z)}$,    the required  facts  are listed as  follows.
    \begin{enumerate}
      \item
          For  $\delta=N^{-\frac{1}{12}}$, when $N$ is large sufficiently,    by definition of $z_2$ and $w_2$ simple manipulation yields
          \begin{align}
            \Re{f(w_{2})} - \Re{f(z_{0})}            &= 2\delta+ \bo(N^{-\frac{1}{3}})+ \frac{1}{2}\log\Big(1 - \frac{4\delta}{1+2\delta +4\delta^2}\big(1 + \bo(N^{-\frac{1}{3}})\big)\Big) \nonumber\\
            &= \frac{4\delta^{2}(1 + 2\delta)}{1 + 2\delta + 4\delta^{2}} \big(1 + \bo(N^{-\frac{1}{3}})\big)\nonumber\\
            &\geq 2N^{-\frac{1}{6}}. \label{airy-w1}
          \end{align}
         Similarly, one has
          \begin{align}
            \Re{f(z_{2})} - \Re{f(z_{0})}
             &= -\frac{4\delta^{2}(1 + 2\delta)}{1 + 2\delta + 4\delta^{2}} \big(1 + \bo(N^{-\frac{1}{3}})\big) \leq -2N^{-\frac{1}{6}}. \label{airy-z1}
          \end{align}
      \item By statement \ref{S3} of Lemma \ref{extr-f}, $\Re{f(w)}$ is strictly increasing along  $\Sigma_{+}^{1}$ as $\Im{w}$ increases. Together with \eqref{airy-w1}  one obtains
          \begin{align}
            \min_{w \in \Sigma_{+}^{1}} \Re{f(w)} - \Re{f(z_{0})} = \Re{f(w_{2})} - \Re{f(z_{0})}\geq 2N^{-\frac{1}{6}}. \label{nocri-airy-sig1}
          \end{align}
      \item By statement \ref{S1}  and statement \ref{S2} of Lemma \ref{extr-f},   $\Re{f(w)}$ is strictly increasing along $\Sigma_{+}^{2}\cup \Sigma_{+}^{3}$ as $\Re{w}$ increases, one thus obtains           \begin{align}
            \min_{w \in \Sigma_{+}^{2}\cup \Sigma_{+}^{3}}& \Re{f(w)} - \Re{f(z_{0})}= \Re{f(w_{3})} - \Re{f(z_{0})} \nonumber \\
            & \geq \Re{f(w_{2})} - \Re{f(z_{0})} \geq 2N^{-\frac{1}{6}},\label{nocri-airy-sig2}
          \end{align}
         where in the second line we have used  statement \ref{S3} of Lemma \ref{extr-f}.
      \item By  statement \ref{S3} in Lemma \ref{extr-f}, $\Re{f(w)}$ is strictly increasing along  $\Sigma_{+}^{4}$ as $\Im{w}$ increases. By choice of $w_5$,     $$w_5-\tau >\frac{3}{2}(\frac{1}{\tau}-\tau), \quad (w_5-z_{0})\eta_{-}>1, $$
       one thus derives   \begin{align}
            \min_{w \in \Sigma_{+}^{4}}& \Re{f(w)} - \Re{f(z_{0})} \geq \Re{f(w_{5})} - \Re{f(z_{0})}\nonumber \\
            &= 4 \eta_{-} (w_5-z_{0})+\log\big(1 - \frac{1}{w_5-\tau}\big( \frac{1}{\tau}-\tau\big)\big)\nonumber\\
            &>4-\log 3.
             \label{nocri-airy-sig4}
          \end{align}

      \item By  statement \ref{S3} of Lemma \ref{extr-f}, $\Re{f(z)}$ is strictly decreasing along  $\mathcal{C}_{+}^{1}$ as $\Im{z}$ increases, together with \eqref{airy-z1} which implies
          \begin{align}
            \max_{z \in \mathcal{C}_{+}^{1}} \Re{f(z)} - \Re{f(z_{0})} =\Re{f(z_{2})} - \Re{f(z_{0})}\leq -2N^{-\frac{1}{6}}. \label{nocri-airy-c1}
          \end{align}
      \item By statement \ref{S1} of Lemma \ref{extr-f},  $\Re{f(z)}$  attains its maximum over $\mathcal{C}_{+}^{2}$ at the right endpoint of $\mathcal{C}_{+}^{2}$. It is easy to see from     \eqref{nocri-airy-c1} that
          \begin{align}
            \max_{z \in \mathcal{C}_{+}^{2}} &\Re{f(z)} - \Re{f(z_{0})}=  \Re{f(z_3)} - \Re{f(z_{0})}\nonumber \\
            &\leq  \Re{f(z_2)} - \Re{f(z_{0})}\leq -2N^{-\frac{1}{6}}.\label{nocri-airy-c2}
          \end{align}

      \item By definition of  $\mathcal{Q}$,  $\Re{f_{1}(4;z)}$ is a constant for all $z \in \mathcal{Q}$.    Since
          \begin{align*}
          \Re{f(z_{4})} - \Re{f(z_{0})} =-\frac{2(1-\tau)}{1+\tau}+\frac{1}{2} \log\Big(1+\frac{4(1-\tau)}{1+\tau}\Big)<0,
          \end{align*}
          whenever $\tau\in [0,1)$,
    one has
          \begin{align}
            \max_{z \in \mathcal{Q}} \Re{f_{1}(z)} - \Re{f(z_{0})}
            = \Re{f(z_{4})} - \Re{f(z_{0})} \leq -2N^{-\frac{1}{6}}.  \label{nocri-airy-c3}
          \end{align}
    \end{enumerate}

    Combining \eqref{nocri-airy-sig1}, \eqref{nocri-airy-sig2} and \eqref{nocri-airy-sig4},  we obtain
    \begin{align*}
       \min_{w \in \Sigma^{\mathrm{global}}} \Re{f(w)} - \Re{f(z_{0})} \geq 2N^{-\frac{1}{6}}.
    \end{align*}
  Note that  when $z\in \mathcal{C}^{\mathrm{local}}$ and $w \in \Sigma^{\mathrm{global}}$
   \begin{align*}
    H(z,w)=  \bo(N^{\frac{m}{12}}) \prod_{k = 1}^{m} (z-z_{0}-(16N)^{-\frac{1}{3}}\eta_{-}^{-1}\pi_{k}),
   \end{align*}
    we proceed for $z$ variable just as  done in the estimate of $I^{1}_N$ and  know that for large $N$
    \begin{align}
      \abs{I_{N}^{2}} &\leq e^{-N^{\frac{5}{6}}},\label{airyIn2}
    \end{align}
    which holds uniformly for  $u, v$ in a compact subset of $\mathbb{R}$.
    Furthermore,  we see for  any $u, v\geq 0$ that
    \begin{align}
      \abs{I_{N}^{2}} &\leq e^{-N^{\frac{5}{6}}} e^{-av + bu},\label{airyIn2'}
    \end{align}
where $a,b$ are given in \eqref{softab}.

    Likewise, combining \eqref{nocri-airy-c1} and \eqref{nocri-airy-c2}, we obtain
    \begin{align}
      \max_{z \in \mathcal{C}^{\mathrm{global}}} \Re{f(z)} - \Re{f(z_{0})} \leq -2N^{-\frac{1}{6}},
    \end{align}
  which implies
    \begin{align}
      \abs{I_{N}^{3}} &\leq e^{-N^{\frac{5}{6}}}\label{airyIn3}
    \end{align}
   and  for $u, v \geq 0$,
    \begin{align}
      \abs{I_{N}^{3}} &\leq e^{-N^{\frac{5}{6}}} e^{-av + bu}.\label{airyIn3'}
    \end{align}

    For $J_N$, take a similar procedure as in the estimate of $I_{N}^{2}$  and we see from \eqref{nocri-airy-c3} that
    \begin{align}
      \abs{J_{N}} &\leq e^{-N^{\frac{5}{6}}}\label{airyJn}
    \end{align}
    and  for $u, v \geq 0$
    \begin{align}
      \abs{J_{N}} &\leq e^{-N^{\frac{5}{6}}} e^{-av + bu}.\label{airyJn'}
    \end{align}

    Finally, combining \eqref{airyIn1}, \eqref{airyIn2}, \eqref{airyIn3} and \eqref{airyJn}, we  obtain
    \begin{align}
     e^{\frac{\eta_{-}}{\varphi(4)}(z_{0}+1)( u-v)} & \frac{1}{\varphi(4)} S_{N}\Big(4N+\frac{u}{\varphi(4)}, 4N+\frac{v}{\varphi(4)}\Big) \nonumber\\ &=\big(1 + \bo(N^{-\frac{1}{3}})\big) K_{\mathrm{Airy}}(\pi;u,v) \label{softsl}
    \end{align}
    uniformly for  $u, v$ in a compact subset of $\mathbb{R}$.
    Combining \eqref{airyIn1'}, \eqref{airyIn2'}, \eqref{airyIn3'} and \eqref{airyJn'},
 we know that
    there exists a constant $C>0$
 such that  for $u, v \geq 0$,
    \begin{align}
      \abs{e^{\frac{\eta_{-}}{\varphi(4)}(z_{0}+1)( u-v)}  \frac{1}{\varphi(4)} S_{N}\Big(4N+\frac{u}{\varphi(4)}, 4N+\frac{v}{\varphi(4)}\Big) } \leq C e^{-av + bu}, \label{softsb}
    \end{align}
where $a = 2- \pi_{*}>0$ and $b = 1 - \pi_{*}>0$.

 \textbf{Step 3:  Completion}.

    We can proceed as in  the estimates of  $I_N$ and $J_N$ for  $S_N$ to  obtain similar  approximations for  $DS_N$ and $IS_N$. However,  the crossing factor  between $z$ and $w$ and the ratios which lead to nontrivial  local scaling factors in the integrand do change.    After detailed analysis as in Step 2,  we can obtain
    \begin{align}
     N^{-\frac{2m}{3}}\e^{\frac{\eta_{-}}{\varphi(4)}(z_{0}+1)( u+v)} & \frac{1}{\varphi(4)} DS_{N}\Big(4N+\frac{u}{\varphi(4)}, 4N+\frac{v}{\varphi(4)}\Big)  =  \bo\big(N^{-\frac{2}{3}}\big), \label{softdsl}\\
    N^{\frac{2m}{3}}\e^{-\frac{\eta_{-}}{\varphi(4)}(z_{0}+1)( u+v)} & \frac{1}{\varphi(4)} IS_{N}\Big(4N+\frac{u}{\varphi(4)}, 4N+\frac{v}{\varphi(4)}\Big)  =  \bo\big(N^{-\frac{2}{3}}\big), \label{softisl}
    \end{align}
    uniformly for  $u, v$ in a compact subset of $\mathbb{R}$, and
    there exists some constant $C>0$ such that  for $u, v\geq 0$
     \begin{align}
   N^{-\frac{2m}{3}}\e^{\frac{\eta_{-}}{\varphi(4)}(z_{0}+1)( u+v)}  \frac{1}{\varphi(4)}    \abs{ DS_{N}\Big(4N+\frac{u}{\varphi(4)}, 4N+\frac{v}{\varphi(4)}\Big)  }
    &\leq C e^{-au -av},  \label{softdsb}
    \\
   N^{\frac{2m}{3}}\e^{\frac{\eta_{-}}{\varphi(4)}(z_{0}+1)( u+v)}  \frac{1}{\varphi(4)}    \abs{ IS_{N}\Big(4N+\frac{u}{\varphi(4)}, 4N+\frac{v}{\varphi(4)}\Big)  }
    &\leq C e^{bu +bv}.  \label{softisb}  \end{align}
       Note that when taking the Pfaffian, the  factors  $N^{-{2m/3}}$ and  $\e^{\eta_{-}(z_{0}+1)( u+v)/\varphi(4)}$
 cancel out each other. Moreover,  the factors  associated with  $DS_N$ and $IS_N$ tend  to zero and thus the Pfaffian reduces to a determinant.  Combination of  \eqref{softsl}, \eqref{softdsl} and \eqref{softisl}   completes  the proof of limiting correlation functions.

    By  \cite[Lemma 2.6]{BBCS},  we obtain from   \eqref{softsb},   \eqref{softdsb} and   \eqref{softisb}  that for  $k\geq 1$
    \begin{equation}
  \left| \frac{1}{(\varphi(4))^k} \mathrm{Pf} \Big[
        K_{N}\Big(4N+\frac{u_i}{\varphi(4)}, 4N+\frac{u_j}{\varphi(4)}\Big) \Big]_{i,j = 1}^{k} \right| \leq (2k)^{\frac{k}{2}} C^{k} \prod_{j=1}^{k} e^{-(a-v)u_{j}}.
  \end{equation}
  Applying  dominated convergence theorem to   the finite $N$  Pfaffian  series expansion,
       we thus   conclude the convergence of the largest eigenvalue.
\end{proof}

\begin{proof}[Proof of Theorem \ref{edge limit}: Part (ii)]
In this case the spectral variable $x>4$,  and  we thus choose a parametrisation representation  $x(\theta)$ defined in  \eqref{xtheta}.  Also,  in  the integrals $I_N$  and  $J_N$ given in   \eqref{SleadingI} and  \eqref{SleadingJ},  take  $$\varphi(x)= N^{-\frac{1}{2}}\frac{\tau}{1-\tau^2} h(\theta),$$
and  denote $f(x;z)$ in   \eqref{phase} by $f(z)$ for  short.  Then the two solutions  $f'(z)=0$  read
    \begin{align*}
      z_{-}=\theta, \quad z_{+}=\frac{1}{\tau}-\tau -\theta.
    \end{align*}
  However, we only need to  use the steepest decent method  to    consider  behavior near the point $z_{-}=\theta$.

\textbf{Step 1:  Choice of contours.}

To construct two contours for $z$ and $w$ variables,  let
    \begin{align*}
     \pi_{*} &= 1 + \max\{|\pi_{1}|, \ldots, |\pi_m|\},\\
      \sigma_{*} &= \frac{1}{2} \Big(\theta +  \min\Big\{\frac{ \sigma_{m+1}}{\tau}, \ldots, \frac{ \sigma_{n}}{\tau}, \frac{1 + \tau^{2}}{2\tau}\Big\}\Big),\\
      \Sigma_{R} &= \left\{w = \frac{1}{\tau} +\frac{R}{\eta_{-}} e^{i\phi}: \phi \in (-\pi, \pi]\right\},\\
      \mathcal{C}_{R} &= \left\{w = \tau + \frac{R}{\eta_{-}}  e^{i\phi}: \phi \in (-\pi, \pi]\right\},
    \end{align*}
    where  $$R = \frac{1}{2} + \max\Big\{\frac{3}{2},   \frac{\sigma_{m+1}}{1-\tau^2}, \ldots, \frac{\sigma_{n}}{1-\tau^2}\Big\}.$$
Take $z_{1} = \theta - (\pi_{*} + 1)h(\theta)/\sqrt{N}$ and two points on the circle $\mathcal{C}_{R}$, denoted by   $z_{2} = z_{1} + iy_{1}, z_{3} = 1 + iy_{2} $ with  $y_1, y_2>0$. Let $\delta$ be a small positive number such that   \eqref{asy-ol-dz1} holds,  introduce some curves
    \begin{align*}
      \mathcal{C}_{+}^{\mathrm{local}} &=\left\{z = z_{1} + iy: 0 \leq y \leq \delta\right\},\\
      \mathcal{C}_{+}^{1} &=\left\{z = z_{1} + iy: 0 \leq y \leq y_{1}\right\},\\
      \mathcal{C}_{+}^{2}&=\left\{z \in \mathcal{C}_{\tau}: 1 \leq \Re{z} \leq z_{1}~\text{and}~\Im{z} \geq 0\right\},\\
      \mathcal{C}_{+}^{3}&=\left\{z\in \mathbb{C}: \left|\frac{\tau z - 1}{z - \tau}\right| = \left|\frac{\tau z_{3} - 1}{z_{3} - \tau}\right|, \Re{z} \leq 1~\text{and}~\Im(z) \geq 0\right\}.
    \end{align*}
     Set $\mathcal{C}^{\mathrm{local}} = \mathcal{C}_{+}^{\mathrm{local}} \cup  \overline{\mathcal{C}_{+}^{\mathrm{local}}}$, and  $\mathcal{C}^{\mathrm{global}} = \big(\mathcal{C}_{+}^{1} \cup \mathcal{C}_{+}^{2} \cup   \overline{\mathcal{C}_{+}^{1}} \cup   \overline{\mathcal{C}_{+}^{2}}\big) \backslash \mathcal{C}^{\mathrm{local}}$. We choose the $z$-contour $\mathcal{C}$ in \eqref{SleadingI}  and \eqref{SleadingJ}  as the union $\mathcal{C}_{+}^{1} \cup \mathcal{C}_{+}^{2} \cup \mathcal{C}_{+}^{3} \cup \overline{\mathcal{C}_{+}^{3}} \cup \overline{\mathcal{C}_{+}^{2}} \cup \overline{\mathcal{C}_{+}^{1}}$.

     For $w$-contour, let  $y_{3}$ be a  small positive number satisfying the inequality  \eqref{asy-ol-dw1'} and  choose $y_{4}>0$ such that $ \frac{1}{2}\big(\tau+\frac{1}{\tau}\big) + iy_{4} \in \Sigma_{R}$, set
    \begin{align*}
      \Sigma^{\mathrm{local}} &= \left\{w = \theta + \frac{\pi_{*} h(\theta)}{\sqrt{N}} e^{i\phi}: -\pi < \phi \leq \pi\right\},\\
      \Sigma_{+}^{1} &= \left\{w = \sigma_{*} + iy: 0 \leq y < y_{3}\right\},\\
      \Sigma_{+}^{2} &=\left\{w = r + iy_{3}: \sigma_{*} \leq r <  \frac{1}{2}\big(\tau+\frac{1}{\tau}\big) \right\},\\
      \Sigma_{+}^{3} &= \left\{w = \frac{1}{2}\big(\tau+\frac{1}{\tau}\big) + iy: y_{3} \leq y < y_{4}\right\},\\
      \Sigma_{+}^{4} &= \left\{w \in \Sigma_{R}:  \frac{1}{2}\big(\tau+\frac{1}{\tau}\big) \leq \Re{w} ~\text{and}~ \Im{w} \geq 0\right\}.
    \end{align*}
      We choose the $w$-contour in \eqref{SleadingI}  and \eqref{SleadingJ} as  $\Sigma^{\mathrm{local}} \cup \Sigma^{\mathrm{global}}$, where $\Sigma^{\mathrm{global}} = \cup_{k=1}^{3}(\Sigma_{+}^{k} \cup \overline{ \Sigma_{+}^{k}})$.

    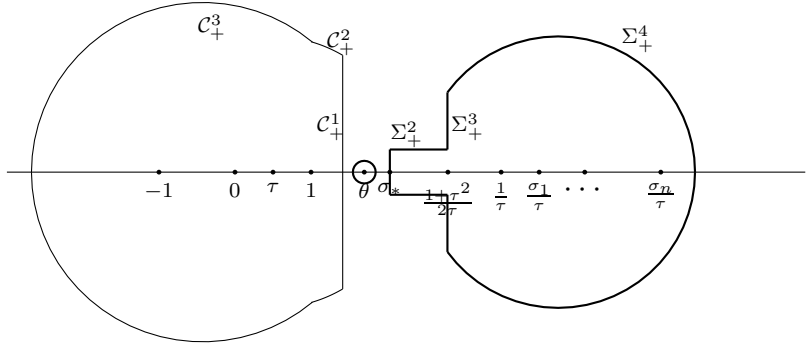
\begin{figure}[h]
      \centering
      \begin{tikzpicture}
        \draw (-3.5,0) -- (7,0);
        \draw ({1.8*cos(0.33*pi r)} ,{1.8*sin(0.33*pi r)}) -- ({1.8*cos(0.33*pi r)} ,{-1.8*sin(0.33*pi r)});
        \draw (1.1,0.6) node[left] {$\scriptstyle \mathcal{C}_{+}^{1}$};
        \draw (0.9,1.4) node[above] {$\scriptstyle \mathcal{C}_{+}^{2}$};
        \draw[domain=0.33*pi:0.41*pi] plot({1.8*cos(\x r)} ,{1.8*sin(\x r)});
        \draw[domain=-0.41*pi:-0.33*pi] plot({1.8*cos(\x r)} ,{1.8*sin(\x r)});
        \draw[domain=0.278*pi:pi] plot({-0.927+2.25*cos(\x r)} ,{2.25*sin(\x r)});
        \draw[domain=-pi:-0.278*pi] plot({-0.927+2.25*cos(\x r)} ,{2.25*sin(\x r)});
        \draw (-0.8,1.6) node[above] {$\scriptstyle \mathcal{C}_{+}^{3}$};

        \draw[thick] ({2.75 - 1.5*cos(pi/5 r)}, 0) -- ({2.75 - 1.5*cos(pi/5 r)}, 0.3);
        \draw[thick] ({2.75 - 1.5*cos(pi/5 r)}, 0.3) -- ({3.75 - 1.8*cos(pi/5 r)}, 0.3);
        \draw[thick] ({3.75 - 1.8*cos(pi/5 r)} ,{1.8*sin(pi/5 r)}) -- ({3.75 - 1.8*cos(pi/5 r)}, 0.3);
        \draw[thick] ({2.75 - 1.5*cos(pi/5 r)}, 0) -- ({2.75 - 1.5*cos(pi/5 r)}, -0.3);
        \draw[thick] ({2.75 - 1.5*cos(pi/5 r)}, -0.3) -- ({3.75 - 1.8*cos(pi/5 r)}, -0.3);
        \draw[thick] ({3.75 - 1.8*cos(pi/5 r)}, -0.3) -- ({3.75 - 1.8*cos(pi/5 r)} ,{-1.8*sin(pi/5 r)});
        \draw (1.4,0.5) node[right] {$\scriptstyle \Sigma_{+}^{2}$};
        \draw (2.2,0.6) node[right] {$\scriptstyle \Sigma_{+}^{3}$};
        \draw[domain=pi/5:pi,thick] plot({3.75 - 1.8*cos(\x r)} ,{1.8*sin(\x r)});
        \draw[domain=-pi:-pi/5,thick] plot({3.75 - 1.8*cos(\x r)} ,{1.8*sin(\x r)});
        \draw (4.8,1.45) node[above] {$\scriptstyle \Sigma_{+}^{4}$};

        \draw[domain=-pi:pi, thick] plot({1.2 + 0.15*cos(\x r)} ,{0.15*sin(\x r)});

        \filldraw
          (-1.5,0) circle (0.7pt) node[below] {$\scriptstyle -1$}
          (-0.5,0) circle (0.7pt) node[below] {$\scriptstyle 0$}
          (0,0) circle (0.7pt) node[below] {$\scriptstyle \tau$}
          (0.5,0) circle (0.7pt) node[below] {$\scriptstyle 1$}
          (1.2, 0) circle(0.7pt) node[below] {$\scriptstyle \theta$}
          ({2.75 - 1.5*cos(pi/5 r)}, 0) circle (0.7pt) node[below] {$\scriptstyle \sigma_{*}$}
          (2.3, 0) circle (0.7pt) node[below] {$\scriptscriptstyle \frac{1 + \tau^{2}}{2\tau}$}
          (3,0) circle (0.7pt) node[below] {$\scriptstyle \frac{1}{\tau}$}
          (3.5,0) circle (0.7pt) node[below] {$\scriptstyle \frac{\sigma_{1}}{\tau}$}
          (4.1,0) circle (0.7pt) node[below] {$\cdots$}
          (5.1,0) circle (0.7pt) node[below] {$\scriptstyle \frac{\sigma_{n}}{\tau}$};
      \end{tikzpicture}
      \caption{Contours of double integrals for $S_N$: outlier}
    \end{figure}

 \textbf{Step 2:  Estimates of $I_{N}$ and $J_{N}$.}

    First, we rewrite $I_{N}$ and $J_{N}$ as
    \begin{align*}
      I_{N} &= \Big(\int_{\mathcal{C}^{\mathrm{local}}} \frac{dz}{2 \pi i} \int_{\Sigma^{\mathrm{local}}} \frac{dw}{2 \pi i} + \int_{\mathcal{C}^{\mathrm{local}}} \frac{dz}{2 \pi i} \int_{\Sigma^{\mathrm{global}}} \frac{dw}{2 \pi i} + \int_{\mathcal{C}^{\mathrm{global}}} \frac{dz}{2 \pi i} \int_{\Sigma^{\mathrm{lobal}}\cup \Sigma^{\mathrm{global}}} \frac{dw}{2 \pi i}\Big) \nonumber\\
      &\quad \frac{\eta_{-}}{\varphi(x)}  e^{\frac{\eta_{-}}{\varphi(x)}( v(z - z_{-})-u(w - z_{-}))}e^{N(f(z) - f(w))}   H(z, w) \nonumber \\
      &=: I_{N}^{1} + I_{N}^{2} + I_{N}^{3},
    \end{align*}
    and
    \begin{align*}
      J_{N} &= \frac{1}{\sqrt{2\pi \eta_{-}(xN+\frac{v}{\varphi(x)})}} \frac{\eta_{-}}{\varphi(x)} \int_{\mathcal{C}} \frac{dz}{2 \pi i} \int_{\Sigma} \frac{dw}{2 \pi i} e^{N(f_{1}(x;z) - f(w))} e^{\frac{\eta_{-}}{\varphi(x)}( v(1 - z_{-})-u(w - z_{-}))} \nonumber\\
      &\quad\times \frac{\sqrt{z^{2}-1}}{1-z^{2}} H(z, w) \Big(1 + \bo\big(\frac{1}{\eta_{-}xN}\big)\Big) .
    \end{align*}

  Take the Taylor expansion at $z_{-}$ and we get
    \begin{align*}
      I_{N}^{1}
      &= \frac{\eta_{-}}{\varphi(x)} \int_{\mathcal{C}^{\mathrm{local}}} \frac{dz}{2\pi i }  \int_{\Sigma^{\mathrm{local}}} \frac{dw}{2\pi i}  e^{\frac{N}{2}f''(z_{-})((z-z_{-})^2-(w-z_{-})^2))}  e^{\frac{\eta_{-}}{\varphi(x)}( v(z-z_{-})- u(w-z_{-})) }  \nonumber\\
      &\quad\times \frac{1}{z-w}  \prod_{k = 1}^{m} \frac{ z-z_{-}-N^{-\frac{1}{2}}h(\theta)\pi_{k}}{ w-z_{-}-N^{-\frac{1}{2}}h(\theta)\pi_{k}}
  \Big(1+\bo(z-z_{-})+\bo(w-z_{-})\Big).
    \end{align*}
   Note that $f''(z_{-})=h^{-2}(\theta)$, after changing variables $z\to z_{-}+N^{-\frac{1}{2}} h(\theta) z$, $w\to  z_{-}+N^{-\frac{1}{2}} h(\theta) w$, we  calculate that
    \begin{align}
      I_{N}^{1} &= \Big(1+\bo(N^{-\frac{1}{2}})\Big) K_{\mathrm{GUE}}(\pi;u,v), \label{olIn1}
    \end{align}
      uniformly for  $u, v$ in a compact subset of $\mathbb{R}$.
    Besides, for $u, v \geq  0$, there exists some constant $C > 0$ such that
    \begin{align}
      \abs{I_{N}^{1}} \leq C e^{-av + bu},\label{olIn1'}
    \end{align}
    where \begin{align}a = \frac{\eta_{-}}{\varphi(x)} (z_{-} - z_{1})
    = \pi_{*}+1,  \quad  b = \frac{\eta_{-}}{\varphi(x)} \frac{\pi_{*}  h(\theta)}{\sqrt{N}} = \pi_{*}.  \label{outlierab}\end{align}

Next, we   turn to analysize the remaining parts of $I_{N}$ and $J_{N}$ by steepest-descent  method.  In order to   establish some  estimates   of $\Re{f(z)}$,    the required  facts  are listed as  follows.  We will use the same notation $ \eps$ to denote different constants  for simplicity.
         \begin{enumerate}
      \item By the parametrization  representation of $x=x(\theta)$,    we see from
          \begin{align*}
            f'(z) &=
            \Big(\frac{1}{\tau} - \tau\Big) \Big(\frac{1}{(\frac{1}{\tau} - \theta) (\theta - \tau)} - \frac{1}{(\frac{1}{\tau} - z) (z - \tau)}\Big)
          \end{align*}
          that
           $f'(z) < 0$ for $z \in (\tau, \theta)$ while $f'(z) > 0$ for $z \in (\theta, \frac{1}{2}(\tau + \frac{1}{\tau}))$. Thus there exists some some $\eps >  0$ such that
          \begin{align*}
            \Re{f(\sigma_{*})} - \Re{f(z_{-})} \geq 2\eps.
          \end{align*}

      \item By statement \ref{S4} of Lemma \ref{extr-f},  $\Re{f(\sigma_{*} + i y)}$ is strictly decreasing for $y > 0$.  Combine the above  estimate and we see that  there exists a small number  $y_{3} > 0$ such that
          \begin{align}
            \Re{f(\sigma_{*} + iy_{3})} - \Re{f(z_{-})} \geq \eps. \label{asy-ol-dw1'}
          \end{align}
          Hence,
          \begin{align}
            \min_{y \in \Sigma_{+}^{1}} \Re{f(\sigma_{*}+iy)} - \Re{f(z_{-})} = \Re{f(\sigma_{*} + iy_{3})} - \Re{f(z_{-})} \geq \eps.\label{asy-ol-dw1}
          \end{align}
      \item When  $r    \in (\theta, \frac{1}{2}(\tau + \frac{1}{\tau}))$,   it follows from
          \begin{align*}
            \frac{\partial}{\partial r} \Re{f(r + iy_{3})}
            &= \frac{1}{\eta_{-}} \frac{1}{(\frac{1}{\tau} - \theta) (\theta - \tau)} + \frac{r - \frac{1}{\tau}}{(\frac{1}{\tau} - r)^{2} + y_{3}^{2}} - \frac{r - \tau}{(r - \tau)^{2} + y_{3}^{2}}\\
            &> \frac{1}{\eta_{-}}  \Big(\frac{1}{(\frac{1}{\tau} - \theta) (\theta - \tau)} + \frac{1}{(\frac{1}{\tau} - r)(r - \tau)} \Big)> 0
          \end{align*}
  that
          \begin{align}
            \min_{w \in \Sigma_{+}^{2}} (\Re{f(w)} - \Re{f(z_{-})}) \geq \Re{f(\sigma_{*} + iy_{3})} - \Re{f(z_{-})} \geq \eps. \label{asy-ol-dw2}
          \end{align}
      \item When  $R> \frac{3}{2}$ and $x>4$,  simple calculation shows
          \begin{align*}
            \frac{\partial}{\partial \phi} \Re{f\Big(\frac{1}{\tau}+ \frac{R}{\eta_{-}} e^{i\phi}\Big)}
            &=- \Big(x- \frac{1}{1+R^2+ 2R \cos\phi}\Big) R \sin\phi\\
            &< -\Big(x- \frac{1}{(R-1)^2}\Big) R \sin\phi <0.
          \end{align*}
         So   $\Re{f\big(\frac{1}{\tau}+ \frac{R}{\eta_{-}} e^{i\phi}\big)}  $  is strictly decreasing in $\phi \in (0, \pi)$. Together with the fact that $\Re{f(w)}$ is constant when $w \in \Sigma_{+}^{3}$,  we obtain
          \begin{align}
           & \min_{w  \in \Sigma_{+}^{3} \cup \Sigma_{+}^{4}} (\Re{f(w)} - \Re{f(z_{-})}) \geq
           \Re{f\big(\frac{1}{2}(\tau + \frac{1}{\tau}) + iy_{3}\big)} - \Re{f(z_{-})}\nonumber\\
            &\quad \geq \Re{f(\pi_{*} + iy_{3})} - \Re{f(z_{-})} \geq \eps. \label{asy-ol-dw3}
          \end{align}

      \item By statement \ref{S4} of Lemma \ref{extr-f},  $\Re{f(z)}$ is strictly decreasing along  the vertical line with real part $z_{-}$ or $z_1$ in the  upper half plane.        So  there exists a small number   $\delta>0$ such that for large $N$
          \begin{align*}
            \Re{f(z_1 +i\delta)} - \Re{f(z_{-})} =\bo\big(\frac{1}{\sqrt{N}}\big)+  \Re{f(z_{-}+i \delta)} - \Re{f(z_{-})}\leq -\eps.           \end{align*}
     Furthermore, on the vertical line with real part $z_1$   one has
          \begin{align}
            \max_{z \in \mathcal{C}_{+}^{1} \backslash \mathcal{C}_{+}^{\mathrm{local}}} \Re{f(z)} - \Re{f(z_{-})} =\Re{f(z_1 +i\delta)} - \Re{f(z_{-})}\leq -\eps. \label{asy-ol-dz1}
          \end{align}
      \item   Recalling  $z_2=z_1+i y_{1}\in \mathcal{C}_{R}$,  by the  calculation one obtains
        \begin{align*}
            \Re{f(z_2)} - \Re{f(z_{-})} &=-\frac{x\eta_{-}}{\sqrt{N}}(\pi_{*}+1)h(\theta)-\log\frac{\frac{1}{\tau}-\theta}{\theta-\tau}\\
          &  \quad +  \frac{1}{2}
          \log  \Big(1+\frac{\eta_{-}}{R^2} \big(\tau +\frac{1}{\tau}-2z_{1}\big)\Big)\\
          &= \bo\big(\frac{1}{\sqrt{N}}\big)-\frac{1}{2}T,
                 \end{align*}
      where         \begin{align*}
      T&=\log\Big(1+\frac{\frac{1}{\tau}+\tau-2\theta}{\theta-\tau}\Big)^2
      -\log\Big(1+\frac{\eta_{-}}{R^2} \big(\tau +\frac{1}{\tau}-2\theta\big)\Big).
                 \end{align*}
Note that  $R\geq 1/2$ and $4\eta_{-}<2/(\theta-\tau)$,   one has \begin{align*}
      T&>\log\Big(1+\frac{2}{\theta-\tau}\big(\frac{1}{\tau}+\tau-2\theta\big)\Big)
      -\log\Big(1+4\eta_{-}\big(\tau +\frac{1}{\tau}-2\theta\big)\Big)>0.                          \end{align*}
      Since \begin{align*}
            \frac{\partial}{\partial \phi} \Re{f\Big(\tau+ \frac{R}{\eta_{-}} e^{i\phi}\Big)}
            &=\Big(-x+\frac{1}{1+R^2+ 2R \cos\phi}\Big) R \sin\phi\\
            &< -\Big(x- \frac{1}{(R-1)^2}\Big) R \sin\phi <0,
          \end{align*}
      whenever $R>3/2$,    so   $\Re{f\big(\tau+ \frac{R}{\eta_{-}} e^{i\phi}\big)}  $  is strictly decreasing in $\phi \in (0, \pi)$.  This implies           \begin{align}
            \max_{z \in \mathcal{C}_{+}^{2}} \Re{f(z)} - \Re{f(z_{-})} \leq \Re{f(z_{2})} - \Re{f(z_{-})} \leq -\eps  \label{asy-ol-dz2}
          \end{align}
          for some small $\eps>0$.
      \item Combine the above estimate and   the fact that $\Re{f_{1}(x;z)}$ is a constant for all $z \in \mathcal{C}_{+}^{3}$,  one obtains
          \begin{align}
            \max_{z \in \mathcal{C}_{+}^{3}} \Re{f_{1}(x;z)} - \Re{f(z_{-})}= \Re{f(z_{3})} - \Re{f(z_{-})}<\Re{f(z_{2})} - \Re{f(z_{-})}\leq -\eps.  \label{asy-ol-dz3}
          \end{align}
      \item  For $N$ large and  $z \in \mathcal{C}_{+}^{1}$,
        \begin{align*}
           \Re{f_{1}(x;z)} - \Re{f(z_{-})}&= x\eta_{-}(1-z_{-})+\frac{1}{2}\log\frac{(\frac{1}{\tau}-z_{1})^2+    (\Im{z})^2}{(z_{1}-\tau)^2+    (\Im{z})^2}-\log\frac{\frac{1}{\tau}-z_{-}}{z_{1}-\tau}\\
     & \leq x\eta_{-}(1-\theta)+ \log\frac{\frac{1}{\tau}-z_{1}}{z_{1}-\tau}-\log\frac{\frac{1}{\tau}-z_{-}}{z_{1}-\tau}\\
     &\leq -\frac{1}{2}x\eta_{-}(\theta-1)    \leq -\eps.
          \end{align*}
      On the other hand,    combining  \eqref{asy-ol-dz2} and the fact that  $$\Re{f_{1}(x;z)} - \Re{f(z)} \leq \eta_{-}x(1 - \Re{z})\leq 0, \quad \forall z \in \mathcal{C}_{+}^{2},$$ we get
          \begin{align*}
            \max_{z \in \mathcal{C}_{+}^{2}}\Re{f_{1}(x;z)} - \Re{f(z_{-})} \leq \max_{z \in \mathcal{C}_{+}^{2}}\Re{f(z)} - \Re{f(z_{-})} \leq -\eps.
          \end{align*}
         In short,  we have
          \begin{align}
            \max_{z \in \mathcal{C}_{+}^{1} \cup \mathcal{C}_{+}^{2}} \Re{f_{1}(z)} - \Re{f(z_{-})} \leq -\eps.  \label{asy-ol-dz4}
          \end{align}
    \end{enumerate}

    Combining \eqref{asy-ol-dw1}, \eqref{asy-ol-dw2} and \eqref{asy-ol-dw3},  we obtain
    \begin{align}
     \min_{w\in \Sigma^{\mathrm{global}}} \Re{f(w)} - \Re{f(z_{-})} \geq \eps.
    \end{align}
         Note that  when $z\in \mathcal{C}^{\mathrm{local}}$ and $w \in \Sigma^{\mathrm{global}}$ \begin{align}
    H(z,w)=  \bo(1) \prod_{k = 1}^{m} (z-\theta-N^{-\frac{1}{2}} h(\theta)\pi_{k}),
   \end{align}
    we proceed for variable $z$  just as  in the estimate of $I^{1}_N$ and  know that for large $N$
    \begin{align}
      \abs{I_{N}^{2}} \leq e^{-\frac{1}{2} \eps N},\label{olIn2}
    \end{align}
    which holds uniformly for  $u, v$ in a compact subset of $\mathbb{R}$.
    Furthermore,  we see for  any $u, v\geq 0$ that
    \begin{align}
     \abs{I_{N}^{2}} \leq e^{-\frac{1}{2} \eps N} e^{-av + bu}, \label{olIn2'}
    \end{align}
where $a,b$ are given in \eqref{outlierab}.

Likewise, combining \eqref{asy-ol-dz1} and \eqref{asy-ol-dz2} we obtain
    \begin{align*}
      \inf_{z \in \mathcal{C}^{\mathrm{global}}} \Re{f(z)} - \Re{f(z_{-})} \leq - \eps,
    \end{align*}
    which  implies
    \begin{align}
      \abs{I_{N}^{3}} \leq e^{-\frac{1}{2} \eps N}.\label{olIn3}
    \end{align}
     and  for $u, v \geq 0$
    \begin{align}
      \abs{I_{N}^{3}} \leq e^{-\frac{1}{2} \eps N} e^{-av + bu}.\label{olIn3'}
    \end{align}

    For $J_N$, take a similar procedure as in the estimate of $I_{N}^{2}$  and we see from   \eqref{asy-ol-dz3} and \eqref{asy-ol-dz4}  that
    \begin{align}
      \abs{J_{N}} \leq e^{-\frac{1}{2} \eps N}.\label{olJn}
    \end{align}
    and for $u,v \geq 0$
    \begin{align}
      \abs{J_{N}} \leq e^{-\frac{1}{2} \eps N} e^{-av + bu}.\label{olJn'}
    \end{align}

    Finally,  combining  \eqref{olIn1}, \eqref{olIn2}, \eqref{olIn3} and \eqref{olJn} we obtain
    \begin{align}
       e^{\frac{\eta_{-}}{\varphi(x)}(z_{-}+1)( u-v)} & \frac{1}{\varphi(x)}  S_{N}\Big(Nx(\theta) +   \frac{u}{\varphi(x)}, Nx(\theta) +   \frac{v}{\varphi(x)}\Big)\nonumber \\
        &= \Big(1+\bo(N^{-\frac{1}{2}})\Big) K_{\mathrm{GUE}}(\pi;u,v) \label{outsl}
    \end{align} uniformly for  $u, v$ in a compact subset of $\mathbb{R}$.
    Also, combining   \eqref{olIn1'}, \eqref{olIn2'}, \eqref{olIn3'} and \eqref{olJn'}, we know that
    there exists a constant $C>0$
 such that  for $u, v \geq 0$
    \begin{align}
    e^{\frac{\eta_{-}}{\varphi(x)}(z_{-}+1)( u-v)}  \frac{1}{\varphi(x)}     \abs{S_{N}\Big(Nx(\theta) +   \frac{u}{\varphi(x)}, Nx(\theta) +   \frac{v}{\varphi(x)}\Big)}
      \leq C e^{-av + bu}, \label{outsb}
    \end{align}
      where $a = 1+\pi_{*}$ and $b = \pi_{*}>0$.

\textbf{Step 3:  Completion}.

    We can proceed as in  the estimates of  $I_N$ and $J_N$ for  $S_N$ to  obtain similar  approximations for  $DS_N$ and $IS_N$. However,  the crossing factor  between $z$ and $w$ and the ratios which lead to nontrivial  local scaling factors in the integrand do change.    After detailed analysis as in Step 2,  we can obtain
    \begin{align}
     N^{-m}\e^{\frac{\eta_{-}}{\varphi(x)}(z_{-}+1)( u+v)} & \frac{1}{\varphi(x)} DS_{N}\Big(xN+\frac{u}{\varphi(x)}, xN+\frac{v}{\varphi(x)}\Big)  =  \bo\big(\frac{1}{N}\big), \label{outdsl}\\
    N^{m}\e^{-\frac{\eta_{-}}{\varphi(x)}(z_{-}+1)( u+v)} & \frac{1}{\varphi(x)} IS_{N}\Big(xN+\frac{u}{\varphi(x)}, xN+\frac{v}{\varphi(x)}\Big)  =  \bo\big(\frac{1}{N}\big), \label{outisl}
    \end{align}
    uniformly for  $u, v$ in a compact subset of $\mathbb{R}$, and
    there exists some constant $C>0$ such that  for $  \forall u, v \geq 0$
        \begin{align}
     N^{-m}\e^{\frac{\eta_{-}}{\varphi(x)}(z_{-}+1)( u+v)} & \frac{1}{\varphi(x)} \abs{DS_{N}\Big(xN+\frac{u}{\varphi(x)}, xN+\frac{v}{\varphi(x)}\Big) } \leq   C e^{-au -av},  \label{outdsb}  \\
    N^{m} e^{-\frac{\eta_{-}}{\varphi(x)}(z_{-}+1)( u+v)} & \frac{1}{\varphi(x)} \abs{IS_{N}\Big(xN+\frac{u}{\varphi(x)}, xN+\frac{v}{\varphi(x)}\Big)}  \leq  C e^{bu +bv}.  \label{outisb}
    \end{align}

       Note that when taking the Pfaffian, the  factors  $N^{-m}$ and  $\e^{\eta_{-}(z_{-}+1)( u+v)/\varphi(x)}$
 cancel out each other. Moreover,  the factors  associated with  $DS_N$ and $IS_N$ tend  to zero and thus the Pfaffian reduces to a determinant.  Combination of  \eqref{outsl}, \eqref{outdsl} and \eqref{outisl}   completes  the proof of limiting correlation functions.

    By  \cite[Lemma 2.6]{BBCS},  we obtain from   \eqref{outsb},   \eqref{outdsb} and   \eqref{outisb}  that for  $k\geq 1$
    \begin{equation}
  \left| \frac{1}{(\varphi(x))^k} \mathrm{Pf} \Big[
        K_{N}\Big(xN+\frac{u_i}{\varphi(x)}, xN+\frac{u_j}{\varphi(x)}\Big) \Big]_{i,j = 1}^{k} \right| \leq (2k)^{\frac{k}{2}} C^{k} \prod_{j=1}^{k} e^{-(a-v)u_{j}}.
  \end{equation}
  Applying  dominated convergence theorem to   the finite $N$  Pfaffian  series expansion,
       we thus   conclude the convergence of the largest eigenvalue.
       \end{proof}

  \section{Hard edge  limits }\label{sectlimits}
   In this section  we investigate hard edge limits when $\tau$ goes to 1 at a certain critical rate, with the help of   the double integral representations for correlation kernels (cf. Theorem \ref{newrep}).

    Given two nonnegative integers $n$ and $\alpha$,  for    $\sigma_1,  \ldots, \sigma_{n}\in [1,\infty)$ and $\tau \in (0, 1]$,  let
    \begin{align}
      \psi_{n, \alpha, \tau}(z, w)
      &= \frac{1-zw}{w-z} \Big(\frac{w}{z}\Big)^{\alpha} \prod_{k = 1}^{n} \frac{\tau z-\sigma_{k}}{\sigma_{k}z-\tau} \frac{\sigma_{k}w-\tau}{\tau w-\sigma_{k}}. \label{temp-f}
    \end{align}
   Choose  two   contours $\mathcal{C}$ and $\Sigma$ as  illustrated in Figure \ref{intcontour2} with  two intersections $e_{2} = 1+i \delta$ and $\bar{e}_{2}= 1-i \delta$ for  some $\delta > 0$, we define   for $\kappa > 0$
    \begin{align}
      DS^{(\mathrm{hard})}(u, v)
      &= - \frac{1}{4 \pi \kappa^{2}} \int_{0}^{\delta} \frac{dt}{ \pi} \Im\left\{\frac{1}{1+it}e^{-\frac{u}{\kappa} (2+it) - \frac{v}{\kappa} \frac{2+it}{1+it}}\right\}
      \nonumber\\
      &\quad  +\frac{1}{4 \pi \kappa^{2}} \mathrm{P.V.} \int_{\mathcal{C}}\frac{dz}{2\pi i} \int_{\Sigma} \frac{dw}{2\pi i} \frac{e^{-\frac{u}{\kappa} (w+1) - \frac{v}{\kappa}(1+ \frac{1}{z}) - \frac{2 \kappa}{z - 1} + \frac{2 \kappa}{w - 1}}}{z \sqrt{w^{2}-1}\sqrt{1 - z^{2}} } \psi_{n, \alpha, 1}(z, w), \label{DSlimit}
    \end{align}
    \begin{align}
      S^{(\mathrm{hard})}(u, v) &=- \frac{1}{\kappa} \int_{0}^{\delta} \frac{dt}{ \pi} \Im\left\{\frac{e^{i\frac{\pi}{4}}}{\sqrt{(2+it)t} } e^{-\frac{u}{\kappa}(2 + it)} g\big(1 + it, \frac{v}{\kappa}\big)\right\}
     \nonumber\\
      &\quad +   \frac{1}{\kappa} \mathrm{P.V.} \int_{\mathcal{C}} \frac{dw}{2\pi i} \int_{\Sigma} \frac{dz}{2\pi i} \frac{e^{- \frac{u(w+1)}{\kappa}}}{\sqrt{w^{2} - 1}} \frac{g(z, \frac{v}{\kappa})}{1 - z^{2}} e^{- \frac{2 \kappa}{z - 1} + \frac{2 \kappa}{w - 1}} \psi_{n, \alpha, 1}(z, w) , \label{Slimit}\\
      IS^{(\mathrm{hard})}(u, v) &= -\Epsilon\big(\frac{u}{\kappa}, \frac{v}{\kappa}\big)
      -4\pi \int_{0}^{\delta} \frac{dt}{ \pi} \Im\left\{\frac{1}{t(2+it)}g\big(\frac{1}{1 + it},\frac{u}{\kappa}\big) g\big(1 + it,\frac{v}{\kappa}\big)\right\}
      \nonumber\\
      &\quad +4\pi \mathrm{P.V.} \int_{\mathcal{C}} \frac{dz}{2\pi i} \int_{\Sigma} \frac{dw}{2\pi i} \frac{g(\frac{1}{w},\frac{u}{\kappa})}{w^{2} - 1} \frac{g(z,\frac{v}{\kappa})}{1-z^2} e^{- \frac{2 \kappa}{z - 1} + \frac{2 \kappa}{w - 1}} \psi_{n, \alpha, 1}(z, w) ,\label{ISlimit}
    \end{align}
   where P.V. denotes the Cauchy principal value.

    \begin{figure}[h]
      \centering
      \begin{tikzpicture}[scale=3]
        \draw[->] (-0.5,0) -- (2.5,0);

        \draw[thick] (0.6,0.4) -- (0.1,0.4);
        \draw[thick] (0.1,0.4) -- (0.1,-0.4);
        \draw[thick] (0.1,-0.4) -- (0.6,-0.4);
        \draw[dashed,thick] (1,0) -- (1.1,0.2);
        \draw[dashed,thick] (1.1,0.2) -- (0.9,0.4);
        \draw[dashed,thick] (0.9, 0.4) -- (0.6, 0.4);
        \draw[dashed,thick] (1,0) -- (1.1,-0.2);
        \draw[dashed,thick] (1.1,-0.2) -- (0.9,-0.4);
        \draw[dashed,thick] (0.9, -0.4) -- (0.6, -0.4);
        \draw (0.1,0.3) node[left] {$\scriptstyle \mathcal{C}$};

        \draw (1,0) -- (0.85,0.2);
        \draw (0.85,0.2) -- (1.15,0.4);
        \draw (1.15,0.4) -- (2.2,0.4);
        \draw (2.2,0.4) -- (2.2,-0.4);
        \draw (2.2,-0.4) -- (1.15,-0.4);
        \draw (1,0) -- (0.85,-0.2);
        \draw (0.85,-0.2) -- (1.15,-0.4);
        \draw (2.2,0.3) node[right] {$\scriptstyle \Sigma$};

        \filldraw
          (0,0) circle (0.3pt) node[below] {$\scriptstyle 0$}
          (0.25,0) circle (0.3pt) node[below] {$\scriptstyle \frac{1}{\sigma_{n}}$}
          (0.5,0) node[below] {$\scriptstyle \cdots$}
          (0.8,0) circle (0.3pt) node[below] {$\scriptstyle \frac{1}{\sigma_{1}}$}
          (1,0) circle (0.3pt) node[below] {$\scriptstyle 1$}
          (1.25,0) circle (0.3pt) node[below] {$\scriptstyle \sigma_{1}$}
          (1.7,0) node[below] {$\scriptstyle \cdots$}
          (2.1,0) circle (0.3pt) node[below] {$\scriptstyle \sigma_{n}$}
          (1,0.3) circle (0.3pt) node[above] {$\scriptstyle e_{2}$};
      \end{tikzpicture}
      \caption{Contours of double integrals for kernels:   hard}\label{intcontour2}
    \end{figure}
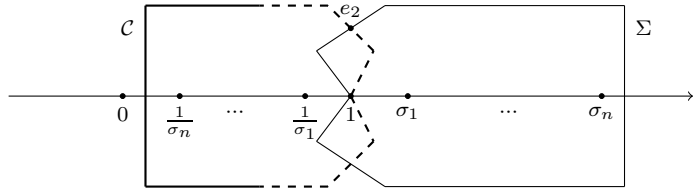

With these preparations out of the way, we  are now ready to   state the  main result  at the hard edge when $\tau$ approaches 1 at a critical rate.

  \begin{thm}\label{originlimit}
    With the same notations as in Theorem \ref{thmkernel}, for   fixed  nonnegative integers $n$ and $\alpha := M-N$, assume that
    \begin{equation}  \sigma_{n+1} = \cdots = \sigma_{N} = 1  \  \mathrm{and} \
     \sigma_i \in (1,\infty),  \quad  i= 1, \ldots, n.
    \end{equation}
    If  $N(1 - \tau) \rightarrow \kappa \in (0,\infty)$ as $N\to \infty$, then
    \begin{equation}
   \lim_{N \to \infty}  \left(\frac{2}{N}\right)^k R_{N}^{(k)}\left(\frac{2u_1}{N}, \ldots, \frac{2u_k}{N}\right)=  \mathrm{Pf}
       {  \begin{bmatrix}
        DS^{(\mathrm{hard})}(u_i, u_j) & S^{(\mathrm{hard})}(u_i, u_j) \\
        -S^{(\mathrm{hard})}(u_j, u_i) & IS^{(\mathrm{hard})}(u_i, u_j)
      \end{bmatrix}}_{i,j = 1}^{k}
  \end{equation}
 hold   uniformly for  $u_1, \ldots, u_k$ in a compact subset of $\mathbb{R}$.
  \end{thm}

  \begin{proof}
    By the Pfaffian formulas for correlation functions  \eqref{kernelexpression}, it is sufficient to obtain  scaling limits of the  sub-kernels for $K_{N}(u, v)$ given in \eqref{matrixkernel}.
Comparing  integral representations of sub-kernels $DS_N, S_N$ and $IS_N$ in Theorems \ref{thmkernel} and  \ref{newrep},  we just focus on asymptotic behavior  of $DS_{N}$ since  the proof of $S_{N}$ and $IS_{N}$ can be obtained in a   similar way.

     Change $z$ to $1/z$ in \eqref{DSneq-2}  and we rewrite $DS_{N}$ as
    \begin{align}
      DS_{N}(u, v)
      &=   \int_{\mathcal{C}_{\{\tau/\sigma_{l}\}}}\frac{dz}{2\pi i} \int_{\mathcal{C}_{\{\sigma_{l}/\tau\}}} \frac{dw}{2\pi i}  \frac{\eta_{-}^2}{4\pi}    \frac{e^{-\eta_{-} v(1+\frac{1}{z})}}{z\sqrt{1 - z^{2}}}
      \frac{e^{-\eta_{-} u(w+1)}}{\sqrt{w^{2}-1}}\psi_{N, \alpha, \tau}(z, w), \label{DSlast}
    \end{align}
    where $\psi_{N, \alpha, \tau}$ is defined  in \eqref{temp-f}. We need to deform the contours  into new contours $\widetilde{\mathcal{C}}$ and $\widetilde{\Sigma}$, which are  illustrated in Figure \ref{intcontour3}  and  will be described  in more detail below.

    \begin{figure}
      \centering
      \begin{tikzpicture}[scale=3]
        \draw (-0.5,0) -- (2.5,0);

        \draw[thick] (0.9, 0) -- (0.6,0.4);
        \draw[thick] (0.6,0.4) -- (0.1,0.4);
        \draw[thick] (0.1,0.4) -- (0.1,-0.4);
        \draw[thick] (0.1,-0.4) -- (0.6,-0.4);
        \draw[thick] (0.6,-0.4) -- (0.9,0);
        \draw[dashed,thick] (0.9,0) -- (1.1,0.2);
        \draw[dashed,thick] (1.1,0.2) -- (0.9,0.4);
        \draw[dashed,thick] (0.9, 0.4) -- (0.6, 0.4);
        \draw[dashed,thick] (0.9,0) -- (1.1,-0.2);
        \draw[dashed,thick] (1.1,-0.2) -- (0.9,-0.4);
        \draw[dashed,thick] (0.9, -0.4) -- (0.6, -0.4);
        \draw (0.1,0.3) node[left] {$\scriptstyle \tilde{\mathcal{C}}$};

        \draw (1.15,0) -- (0.85,0.2);
        \draw (0.85,0.2) -- (1.15,0.4);
        \draw (1.15,0.4) -- (2.2,0.4);
        \draw (2.2,0.4) -- (2.2,-0.4);
        \draw (2.2,-0.4) -- (1.15,-0.4);
        \draw (2.2,0.3) node[right] {$\scriptstyle \tilde{\Sigma}$};

        \draw (1.15,0) -- (0.85,-0.2);
        \draw (0.85,-0.2) -- (1.15,-0.4);

        \filldraw
          (0,0) circle (0.3pt) node[below] {$\scriptstyle 0$}
          (0.8,0) circle (0.3pt) node[below] {$\scriptstyle \tau$}
          (0.9,0) circle (0.3pt) node[below] {$\scriptstyle e_{0}$}
          (0.85,0.2) circle (0.3pt) node[above] {$\scriptstyle e_{3}'$}
          (1.1,0.2) circle (0.3pt) node[above] {$\scriptstyle e_{3}$}
          (0.6,0.4) circle (0.3pt) node[above] {$\scriptstyle e_{4}$}
          (1,0) circle (0.3pt) node[below] {$\scriptstyle 1$}
          (1.25,0) circle (0.3pt) node[below] {$\scriptstyle \frac{1}{\tau}$}
          (1,0.1) circle (0.3pt) node[above] {$\scriptstyle e_{1}$}
          (1,0.3) circle (0.3pt) node[above] {$\scriptstyle e_{2}$}
          (1.15,0) circle (0.3pt) node[below] {$\scriptstyle e_{0}'$};
      \end{tikzpicture}
      \caption{Contours of double integrals for $DS_{N}$: hard}\label{intcontour3}
    \end{figure}

    Let $e_{0}$ and $e_{0}'$  be two  points, which are  respectively   located between $\tau$ and $1$ and between $1$ and $1/\tau$.  Take $e_{1} = 1 + i\eps_{N}$ with $0 < \eps_{N} < \delta$ and $e_{2} = 1 + i\delta$  with a fixed $\delta>0$. Also let $e_{3}$ and $e_{3}'$ be two points in the upper complex half-plane with $\Re{e_{3}} > 1$ and $\Re{e_{3}'} < 1$.  Set
   $$r_1=\frac{1}{2}\min\big\{\frac{1}{\sigma_1}, \ldots, \frac{1}{\sigma_n}\big\}, \quad r_2=1+\max\{\sigma_1, \ldots, \sigma_n\}.$$ We construct requested  integration  contours as follows.
    Let $\widetilde{\mathcal{C}}$ be a contour staring at $e_{0}$, passing through $e_{1}, e_{3}, e_{2},
    r_1,$ $\bar{e}_{2}, \bar{e}_{3}, \bar{e}_{1}$ in turn and finally returning to $e_{0}$. Likewise,
    let $\widetilde{\Sigma}$ be a  simple contour, starting at $e_{0}'$, moving anticlockwise across $\bar{e}_{1}$, $\bar{e}_{3}'$, $\bar{e}_{2}$, $r_2$,  $e_{2}$, $e_{3}'$, $e_{1}$ in turn and finally returning to $e_{0}'$.  Besides, we choose $\widetilde{\mathcal{C}}_{+}$ to be a  closed contour,  which is obtained by connecting  $e_{0}, e_{1}, e_{3}, e_{2}, e_{4}$ and $e_{0}$ in sequence by line   segments. Let $\widetilde{\mathcal{C}}_{-}$ be the complex conjugate curve of $\widetilde{\mathcal{C}}_{+}$.

    To obtain limiting correlation kernels we  split $DS_{N}$ in \eqref{DSlast} into three parts
    \begin{align}
      DS_{N}(u, v)
      & =\text{P.V.}  \Big(\int_{\widetilde{\mathcal{C}}} dz  \int_{\widetilde{\Sigma}} dw- \int_{\widetilde{\mathcal{C}}_{+}} dz \int_{\widetilde{\Sigma}} dw -  \int_{\widetilde{\mathcal{C}}_{-}} dz \int_{\widetilde{\Sigma}} dw\Big)\Big(\cdot\Big) \nonumber \\&=:I_{1} (u, v) - I_{2}(u, v) - I_{3}(u, v).\nonumber
    \end{align}
   By the assumption on $\tau$,  we see that  both $\tau$ and $1/\tau$ approach $1$ as $N \to \infty$, from which the two points $e_{0}$ and $e_{0}'$ tend to $1$ respectively from the left-hand side  and the right-hand side. Therefore, both $\widetilde{\mathcal{C}}$ and $\widetilde{\Sigma}$ are  forced to  change to the depicted curves as in Figure \ref{intcontour2},  by letting  $e_{1}$ and $\bar{e}_{1}$ approach $1$ (that is,  $\eps_{N}$ tends to $0$) and noting that they encircle  $\tau$ and $1/\tau$ respectively.
    On the other hand, we get from the assumptions of $\sigma_{n+1}=\cdots=\sigma_N=1$ and $N(1 - \tau) \rightarrow \kappa \in (0,\infty)$ that
    \begin{align}
      \Big(\frac{\tau z - 1}{z - \tau} \frac{w - \tau}{\tau w - 1}\Big)^{N-n}
      & = \Big(1 - \frac{\frac{1}{\tau} - \tau}{z - \tau}  \Big)^{N-n} \Big(1- \frac{\frac{1}{\tau}-\tau }{w - \tau}  \Big)^{-N+n} \to e^{- \frac{2 \kappa}{z - 1} + \frac{2 \kappa}{w - 1}}.\nonumber
    \end{align}
    Hence, as $N \to \infty$ the rescaled integral
    $$ \frac{4}{N^{2}} I_{1}\Big(\frac{2u}{N}, \frac{2v}{N}\Big)$$
    converges to  the principal value integral on the right-hand side of \eqref{DSlimit}.

    To compute $I_{2}$ and  $I_{3}$, we first integrate out variable  $z$. It is easy to see that only the line segments from  $\widetilde{\Sigma}$ lying inside $\widetilde{\mathcal{C}}_{\pm}$ lead to nonzero contribution, so  noting by convention
    $$\sqrt{w-1}=\begin{cases} \quad i\sqrt{1-w}, \quad &\Re w\leq 1 \& \Im w>0, \\
  - i\sqrt{1-w} , \quad  &\Re w\leq 1 \& \Im w<0, \end{cases}$$
    and
    applying Cauchy's residue theorem gives us
        \begin{align*}
    &  \frac{4}{N^{2}} \Big(I_{2}\big(\frac{2u}{N}, \frac{2v}{N}\big) + I_{3}\big(\frac{2u}{N}, \frac{2v}{N}\big)\Big)\\
      &= \bigg(\int_{1 + i\delta}^{1 + i\eps_{N}}   \frac{dw}{2 \pi i}-\int_{1 - i\eps_{N}}^{1 - i\delta} \frac{dw}{2 \pi i} \bigg) \frac{\eta_{-}^{2}}{\pi N^{2}} \frac{ i}{w  }    e^{-\frac{2u\eta_{-}}{N} (w+1) - \frac{2v\eta_{-}}{N} \frac{w+1}{w}}
     \nonumber\\
      &= \frac{\eta_{-}^{2}}{\pi N^{2}} \int_{\eps_{N}}^{\delta} \frac{dt}{ \pi} \Im\left\{\frac{1}{1+it}e^{-\frac{2u\eta_{-}}{N} (2+it) - \frac{2v\eta_{-}}{N} \frac{2+it}{1+it}}\right\} .
    \end{align*}
Since $\eps_{N} \to 0$ converges to  $0$ as $N\to \infty$, we get
    \begin{align*}
      &\lim_{N \to \infty} \frac{4}{N^{2}} \Big(I_{2}\big(\frac{2u}{N}, \frac{2v}{N}\big) + I_{3}\big(\frac{2u}{N}, \frac{2v}{N}\big)\Big)
      = \frac{1}{4 \pi \kappa^{2}} \int_{0}^{\delta} \frac{dt}{ \pi} \Im\left\{\frac{1}{1+it}e^{-\frac{u}{\kappa} (2+it) - \frac{v}{\kappa} \frac{2+it}{1+it}}\right\}.
    \end{align*}

    Together with the limit of  $I_1$,  this gives asymptotic behavior  of $DS_N$.

   For   $IS_{N}$,  change of variable   $w \mapsto 1/w$ in \eqref{ISneqII} gives
    \begin{align*}
      IS_{N}(u, v)
      &=  -\Epsilon(\eta_{-}u, \eta_{-}v) + 4\pi \int_{\mathcal{C}_{\{\tau/\sigma_{l}\}}} \frac{dz}{2\pi i} \int_{\mathcal{C}_{\{\sigma_{l}/\tau\}}} \frac{dw}{2\pi i} \frac{g(\frac{1}{w},\eta_{-}u)}{w^{2} - 1} \frac{g(z,\eta_{-}v)}{1-z^2} \psi_{N, \alpha, \tau}(z, w).
    \end{align*}
    So we can proceed in the same way as  in  obtaining  scaling limit of $DS_{N}$ and prove that
    $$\frac{2}{N}S_{N} \big(\frac{2u}{N}, \frac{2v}{N}\big) \ \mathrm{and} \    IS_{N} \big(\frac{2u}{N}, \frac{2v}{N}\big) $$
    converge to the desired limits respectively.  This thus completes the proof.
  \end{proof}

\section{Miscellany} \label{lastsect}

\subsection{Integrable form}
In order to give new expressions of double integrals \eqref{Ssoft}, \eqref{DSsoft} and \eqref{ISsoft},  for $p=0,1,\ldots,$ let's define two families  of functions
\begin{align}
 \phi_{1,p}(\kappa, \pi;u)=& \int_{\mathcal{C}_{<}}\frac{dw}{2\pi i }     e^{\frac{1}{3}(w-\kappa)^3-u(w-\kappa)}   \frac{1}{\sqrt{2w}}
   \prod_{k = 1}^{p} \frac{w+\pi_{k}}{ w- \pi_{k}}.
  \end{align}
and
  \begin{align}
 \phi_{2,p}(\kappa, \pi;v)=& \int_{\mathcal{C}_{>}}\frac{dz}{2\pi i }     e^{-\frac{1}{3}(z-\kappa)^3+ v(z-\kappa)}   \frac{1}{\sqrt{2z}}
   \prod_{k = 1}^{p} \frac{z-\pi_{k}}{ z+ \pi_{k}}.
  \end{align}

 In the special case of  $m=0$,   $DS^{(\mathrm{soft})}(\kappa, \pi;u,v)$, $S^{(\mathrm{soft})}(\kappa, \pi;u,v)$ and $IS^{(\mathrm{soft})}(\kappa, \pi;u,v)$ are denoted by $DS^{(\mathrm{soft})}(\kappa;u,v)$, $S^{(\mathrm{soft})}(\kappa;u,v)$ and $IS^{(\mathrm{soft})}(\kappa;u,v)$ respectively, for short.  Also,  $\phi_{1,0}(\kappa, \pi;u)$ and    $\phi_{2,0}(\kappa, \pi;v)$ are denoted   by  $  \phi_{1}(\kappa;u)$ and $  \phi_{2}(\kappa;v)$. Compared to contour integral representations of Airy functions,   there is an additional square root  factor   in the integrand in   both $  \phi_{1}(\kappa;u)$ and $  \phi_{2}(\kappa;v)$. They  can be treated as generalized Airy functions and satisfy certain differential equations of order 3 as follows
   \begin{align}   \phi_{1}'''(\kappa;u)-\kappa\big(\phi_{1}''(\kappa;u)-u\phi_{1}(\kappa;u)\big) -u\phi_{1}'(\kappa;u)  -\frac{1}{2}\phi_{1}(\kappa;u)&=0, \label{3rdeqn1} \\
    \phi_{2}'''(\kappa;v)+\kappa\big(\phi_{2}''(\kappa;v)-v\phi_{2}(\kappa;v)\big) -v\phi_{2}'(\kappa;v)  -\frac{1}{2}\phi_{2}(\kappa;v)&=0.   \label{3rdeqn2}  \end{align}
    Actually,  \eqref{3rdeqn1} can be easily derived  from
  \begin{align*}
 \phi_{1}(\kappa;u)=&  - \int_{\mathcal{C}_{<}}\frac{dw}{2\pi i }  \big( (w-\kappa)^2 -u\big) \sqrt{2w}  e^{\frac{1}{3}(w-\kappa)^3-u(w-\kappa)}
  \end{align*}
  by integration by parts and
 \begin{align*}
 \phi_{1}'(\kappa;u)=\kappa  \phi_{1}(\kappa;u) -\frac{1}{2} \int_{\mathcal{C}_{<}}\frac{dw}{2\pi i }   \sqrt{2w}  e^{\frac{1}{3}(w-\kappa)^3-u(w-\kappa)},\\
 \phi_{1}'''(\kappa;u)=\kappa  \phi_{1}''(\kappa;u) -\frac{1}{2} \int_{\mathcal{C}_{<}}\frac{dw}{2\pi i }   \sqrt{2w} (w-\kappa)^2 e^{\frac{1}{3}(w-\kappa)^3-u(w-\kappa)}.
  \end{align*}
    \eqref{3rdeqn2} can be  obtained in a similar way.

   \begin{prop} With the above notations, the following hold
   \begin{align} DS^{\mathrm{soft}} (\kappa, \pi;u,v)&=DS^{(\mathrm{soft})}(\kappa;u,v) \nonumber\\
   &\quad +\sum_{p=1}^{m}\Big( \phi_{1,p-1}(\kappa, \pi;u) \phi_{1,p}(\kappa, \pi;v)-\phi_{1,p}(\kappa, \pi;u)\phi_{1,p-1}(\kappa, \pi;v) \Big), \label{repDS}
   \end{align}
   \begin{align}
   S^{\mathrm{soft}} (\kappa, \pi;u,v)&=S^{(\mathrm{soft})}(\kappa;u,v) \nonumber\\
   &\quad +\sum_{p=1}^{m}\Big( \phi_{1,p}(\kappa, \pi;u)\phi_{2,p-1}(\kappa, \pi;v) - \phi_{1,p-1}(\kappa, \pi;u)  \phi_{2,p}(\kappa, \pi;v)\Big), \label{repS}
   \end{align}
   \begin{align}
   IS^{\mathrm{soft}} (\kappa, \pi;u,v)&=IS^{(\mathrm{soft})}(\kappa;u,v) \nonumber\\
   &\quad +\sum_{p=1}^{m}\Big( \phi_{2,p-1}(\kappa, \pi;u) \phi_{2,p}(\kappa, \pi;v)-\phi_{2,p}(\kappa, \pi;u) \phi_{2,p-1}(\kappa, \pi;v)\Big). \label{repIS}
     \end{align}\end{prop}
 \begin{proof}
 Using the identity
   \begin{align}
  \frac{z+w}{z-w} \prod_{k = 1}^{m} \frac{z-\pi_{k}}{ z+ \pi_{k}}\frac{w+\pi_{k}}{ w-\pi_{k}}=\frac{z+w}{z-w}+\sum_{p=1}^{m}\Big(\frac{w+\pi_{p}}{ w-\pi_{p}}-\frac{z-\pi_{p}}{ z+\pi_{p}}\Big) \prod_{k = 1}^{p-1} \frac{z-\pi_{k}}{ z+ \pi_{k}}\frac{w+\pi_{k}}{ w-\pi_{k}}, \label{ide}
  \end{align}
  integrate term by term and we obtain \eqref{repS}.  Verification of  the above identity \eqref{ide} is a straight   substitution of   the following facts
   \begin{align*}
\mathrm{ LHS \ of\ }  \eqref{ide}&=\frac{z+w}{z-w}+\frac{z+w}{z-w}\sum_{p=1}^{m}\Big(\frac{z-\pi_{p}}{ z+\pi_{p}}\frac{w+\pi_{p}}{ w-\pi_{p}}-1\Big) \prod_{k = 1}^{p-1} \frac{z-\pi_{k}}{ z+ \pi_{k}}\frac{w+\pi_{k}}{ w-\pi_{k}}  \end{align*}
  and
   \begin{align*}
\frac{z+w}{z-w}\Big(\frac{z-\pi_{p}}{ z+\pi_{p}}\frac{w+\pi_{p}}{ w-\pi_{p}}-1\Big) =\frac{w+\pi_{p}}{ w-\pi_{p}}-\frac{z-\pi_{p}}{ z+\pi_{p}}.
  \end{align*}

  Similarly, we can prove   \eqref{repDS} and \eqref{repIS}.
 \end{proof}

 There exist certain relations between limiting  sub-kernels at the soft edge  for the GOE and GSE ensembles, see e.g.  \cite{tw1996}. Lemmas 2.6 and 2.7  in  \cite{BBCS} show that these sub-kernels admit alternative double integrals, in which the same crossing factors  appear in the integrands as in \eqref{Ssoft},  \eqref{DSsoft} and \eqref{ISsoft}.  It is an interesting question to find possible interrelations between $DS^{(\mathrm{soft})}(\kappa, \pi;u,v)$, $S^{(\mathrm{soft})}(\kappa, \pi;u,v)$ and $IS^{(\mathrm{soft})}(\kappa, \pi;u,v)$.   However,  $S^{(\mathrm{soft})}(\kappa;u,v)$ admits  an integrable form.

\begin{prop}  When $m=0$, the sub-kernel  \eqref{Ssoft}  can be rewrited as
   \begin{align*}  S^{(\mathrm{soft})}(\kappa;u,v) &=\frac{1}{u-v}
  \Big( -2 \phi_{1}''(\kappa;u) \phi_{2}(\kappa;v)-2\phi_{1}(\kappa;u) \phi_{2}''(\kappa;v)+2\phi_{1}'(\kappa;u) \phi_{2}'(\kappa;v)
  \nonumber\\
  &+2\kappa \big( \phi_{1}'(\kappa;u) \phi_{2}(\kappa;v)-\phi_{1}(\kappa;u) \phi_{2}'(\kappa;v)\big) +(u+v) \phi_{1}(\kappa;u) \phi_{2}(\kappa;v)\Big).
     \end{align*}\end{prop}
 \begin{proof}  Apply the integration by parts and we get
 \begin{align*}
  \int_{\mathcal{C}_{<}}\frac{dw}{2\pi i }  (2w)^{-\frac{3}{2}}  e^{\frac{1}{3}(w-\kappa)^3-u(w-\kappa)} = \phi_{1}''(\kappa;u)-u \phi_{1}(\kappa;u),
  \end{align*}
 and
  \begin{align*}
  \int_{\mathcal{C}_{>}}\frac{dz}{2\pi i }  (2z)^{-\frac{3}{2}}  e^{-\frac{1}{3}(z-\kappa)^3+v(z-\kappa)} = -\phi_{1}''(\kappa;v)+ v\phi_{1}(\kappa;v).
  \end{align*}
  Thus,
   \begin{align*}
(u-v) S^{(\mathrm{soft})}&(\kappa;u,v)= \int_{\mathcal{C}_{>}} \frac{dz}{2\pi i }    \int_{\mathcal{C}_{<}}   \frac{dw}{2\pi i }   e^{\frac{1}{3}(w-\kappa)^3-\frac{1}{3}(z-\kappa)^3}    \frac{1}{\sqrt{4zw}}
\frac{z+w}{z-w}
   \nonumber\\
& \quad \times    \big(-\frac{\partial}{\partial w}-\frac{\partial}{\partial z}\big)  e^{-u(w-\kappa)+ v(z-\kappa)}\\
&=\int_{\mathcal{C}_{>}} \frac{dz}{2\pi i }    \int_{\mathcal{C}_{<}} \frac{dw}{2\pi i }
 \frac{1}{\sqrt{4zw}} e^{\frac{1}{3}(w-\kappa)^3-u(w-\kappa)}   e^{- \frac{1}{3}(z-\kappa)^3+v(z-\kappa)}
   \nonumber\\
& \quad \times  \big(-(z-\kappa+w-\kappa) ^2 -2\kappa(z-\kappa+w-\kappa)  +\frac{1}{2z}-\frac{1}{2w}\big),
& \end{align*}
 from which the required result follows.
 \end{proof}

\subsection{Singular value of GUE}
As mentioned earlier, when $\tau=1$  $X$ defined in  \eqref{ellipticPDF} reduces to an $N\times N$ GUE  matrix
$H$ with density
\begin{equation} 2^{-\frac{1}{2}N} \pi^{-\frac{1}{2}N^2} e^{-\frac{1}{2}\Tr(H^2)}. \label{GUEdensity}\end{equation}
The  eigenvalues of GUE have been studied quite well,  however, only  recently have people began to investigate singular values of GUE (or GOE and even more general matrix ensembles with even weights), see   relevant  works of Forrester \cite{forrester06},  Edelman and La Croix \cite{Edelman},   Bornemann and La Croix \cite{BL},  Bornemann and Forrester  \cite{BF16}.
There exists  an unexpected decomposition that   the singular values of the GUE are distributed identically
to the union of the distinct non-zero singular values of two independent anti-GUE ensembles,    one of order $N$ and the other of order $N + 1$.   This result was first  observed by Forrester  \cite[eq.(2.6)]{forrester06}  and then by Edelman and La Croix \cite[Thm. 1]{Edelman}.     Interestingly,
Edelman and La Croix commented in \cite{Edelman} that `the singular values of the GUE play an unpredictably important role that had gone unnoticed for decades even though, in hindsight, so many clues had been around'.

 \begin{thm} \label{SVGUE}
 With  a Hermitian matrix  $H$ distributed as   in  \eqref{GUEdensity},
 let $\lambda_{\mathrm{max}}(H^2)$ be the largest eigenvalue of $H^2$,  then   for any  $x\in \mathbb{R}$          \begin{align}
  \lim_{N \to \infty}\mathbb{P}\Big(2^{-\frac{4}{3}}N^{-\frac{1}{3}}\big(\lambda_{\mathrm{max}}(H^2)-4N\big) \leq   x\Big)=  F_{\mathrm{GUE}}(2^{-\frac{2}{3}}x)F_{\mathrm{GUE}}(2^{-\frac{2}{3}}x).
  \end{align}
     \end{thm}

\begin{proof}

 Let $\lambda_{\mathrm{max}}(H)$  and $\lambda_{\mathrm{min}}(H)$ be the largest and smallest eigenvalues of $H$,  then
\begin{align*}
  &\mathbb{P}\Big(2^{-\frac{4}{3}}N^{-\frac{1}{3}}\big(\lambda_{\mathrm{max}}(H^2)-4N\big) \leq   x
   \Big)\\
 &= \mathbb{P}\Big(\frac{\lambda_{\mathrm{max}}(H)}{\sqrt{N}}\leq 2\big(1+2^{-\frac{2}{3}}N^{-\frac{2}{3}}x\big)^{\frac{1}{2}}, \frac{\lambda_{\mathrm{min}}(H)}{\sqrt{N}}\geq -2\big(1+2^{-\frac{2}{3}}N^{-\frac{2}{3}}x\big)^{\frac{1}{2}}\Big)\\
 &=\mathbb{P}\Big( N^{\frac{2}{3}} \big(\frac{\lambda_{\mathrm{max}}(H)}{\sqrt{N}}-2\big) \leq 2^{-\frac{2}{3}}x+\bo(N^{-\frac{2}{3}}),
  N^{\frac{2}{3}} \big(\frac{\lambda_{\mathrm{min}}(H)}{\sqrt{N}}+2\big) \geq -2^{-\frac{2}{3}}x-\bo(N^{-\frac{2}{3}})\Big).
  \end{align*}
Since both the centered and rescaled largest and smallest eigenvalues of the GUE  asymptotically  follow the same GUE Tracy-Widom distribution,   by asymptotic independence  of largest  and smallest eigenvalues of the GUE (see   \cite{BDN} or \cite{Bo10,BCZ}), we see that  the required result immediately from  asymptotic independence  in \cite[Corollary 1]{BDN}.

 We can also give  another different proof by using  the superposition representation of singular values of  the GUE as  the union of the singular values of two independent anti-GUE matrices,  which is proved  by
Forrester  \cite[eq.(2.6)]{forrester06}  and Edelman and La Croix \cite[Thm. 1]{Edelman}, we see
that the largest singular value of the GUE  is distributed as maximum of two  independent largest eigenvalues of  anti-GUE.   Since the largest eigenvalue for each anti-GUE, after the same rescalings,  follows the same GUE Tracy-Widom distribution, we can complete the proof.
\end{proof}

We remark that Theorem \ref{SVGUE}  should hold for  more general random matrix ensembles with even weights, see  \cite{forrester06} and \cite{BF16}.

\subsection{Crossovers}
At the critical value of $\tau=1-2^{\frac{2}{3}}N^{-\frac{1}{3}}\kappa$, for any fixed $\kappa>0$, the limiting  correlation functions at the soft edge are described by a Pfaffian point process with crossover kernel  \eqref{mkernel}.   And intuitively, it seems plausible to us that \eqref{mkernel} approaches the Airy kernel of the GUE as $\kappa \to \infty$ , while    $\kappa \to  0$  it could correspond  to that of the largest squared singular value of   GUE.

\begin{prop}  \label{crosslimit}
With a $2\times 2$ matrix kernel  \eqref{mkernel},  we have
 \begin{equation}
  \lim_{\kappa \to 0}K^{(\mathrm{soft})}(\kappa, \pi;u, v) = K^{(\mathrm{soft})}(0, \pi;u, v), \label{0mkernel}
\end{equation}
and
 \begin{equation}
  \lim_{\kappa \to \infty} \diag((2\kappa)^{-m},1)K^{(\mathrm{soft})}(\kappa, \pi+\kappa;u, v)\diag(1,(2\kappa)^{m}) =K^{(\mathrm{soft})}(\infty, \pi;u, v),\label{infmkernel}
\end{equation}
where
\begin{equation*}
K^{(\mathrm{soft})}(\infty, \pi;u, v)   =  \begin{bmatrix}
       0 &  K_{\mathrm{Airy}}(\pi;u,v) \\
       - K_{\mathrm{Airy}}(\pi;v,u)  & 0
        \end{bmatrix}.
\end{equation*}
 \end{prop}

\begin{proof}
Obviously, \eqref{0mkernel} follows from the definition of  $K^{(\mathrm{soft})}(\kappa, \pi;u, v)$.  In \eqref{Ssoft},  \eqref{DSsoft} and  \eqref{ISsoft}, replace   $\pi_{k}$ by  $\pi_{k}+\kappa$ and  make change of variables $z,w \to z+\kappa, w+\kappa$,  as $\kappa \to \infty$ we have
  \begin{align*}
 (2\kappa)^{-m} DS^{(\mathrm{soft})}(\kappa, \pi+\kappa;u,v)=\bo(\kappa^{-2}),\\
  (2\kappa)^{m} IS^{(\mathrm{soft})}(\kappa, \pi+\kappa;u,v)=\bo(\kappa^{-2}),
\end{align*}
and
 \begin{align*}
 S^{(\mathrm{soft})}(\kappa, \pi+\kappa;u,v) \to K_{\mathrm{Airy}}(\pi;u,v).
\end{align*}
This completes the proof of \eqref{infmkernel}.
\end{proof}

Furthermore,
as to  the distribution $F^{(\mathrm{soft})}(\kappa,\pi;x)$ defined  in \eqref{pfaffianseries}, we can prove that
it tends to the GUE Tracy-Widom distribution $F_{\mathrm{GUE}}(x)$ as $\kappa \to \infty$,
by following the key steps in  the proof of Theorem \ref{largestcritthm}.
However, as  $\kappa \to 0$ it should tend to
  another distribution   $F_{\mathrm{GUE}}(2^{-\frac{2}{3}}x)F_{\mathrm{GUE}}(2^{-\frac{2}{3}}x)$ as in Theorem \ref{SVGUE}.
 Besides,     as some of parameters $\kappa, \pi_1, \ldots, \pi_m$  change in a certain way,  some distributions    can be obtained from  $F^{(\mathrm{soft})}(\kappa,\pi;x)$, for instance,  the distribution  $(F_{\mathrm{GOE}}(x))^2$.  This is because  it is equal to
$F_{1}(0,x)$ as a  special case of $m=1$ and $\pi_1=0$ in  \eqref{detAiryseries}; see \cite[eq.(3.34)]{BF03} or \cite[Sect. 1.2]{BBP05}.

\subsection{Open questions}
We have investigated the complex elliptic Ginibre ensemble with a spike, and derived its singular value PDF as a Pfaffian point process. In the large $N$ limit, we have obtained crossover phenomena at the soft and hard edges  when $\tau$ changes   at the critical rate; see Theorems  \ref{edgecritthm}, \ref{largestcritthm} and \ref{originlimit}. We have also proved the sine kernel in the bulk and the BBP transition of the largest eigenvalue when $\tau$ is fixed; see Theorems \ref{bulk limit} and \ref{edge limit}.   Our  first question is to prove the analogous  results about the real elliptic Ginibre ensemble.

{\bf Question 1}. With the same scalings,  verify the analogous  results  about the real elliptic Ginibre ensemble (with a spike). In particular,  find the  similar  crossover phenomena at the soft and hard edges as in Theorems  \ref{edgecritthm}, \ref{largestcritthm} and \ref{originlimit}.

The second question is to prove  universality of singular values of complex  elliptic random matrices with general entries;  see e.g.  \cite{girko85} and \cite{NO15} for definition.

{\bf Question 2}. Verify local universality of singular values for complex elliptic random matrices under certain higher moment assumptions. In particular, prove the four moment theorem for   complex elliptic random matrices (cf. \cite{TV2010, TV2011}).

Compare  Proposition \ref{crosslimit} and Theorem \ref{SVGUE},  when $m=0$  we can ask

{\bf Question 3}.
Can  $F^{(\mathrm{soft})}(0,0;x)$ given in Definition \ref{pflargest} be simplified  to  the form $F_{2}(2^{-\frac{2}{3}}x)F_{2}(2^{-\frac{2}{3}}x)$?

The last question is about the existence of crossover   phenomenon in the bulk when $\tau$ tends to 1 at a certain rate.

{\bf Question 4}.
Is there a crossover   phenomenon in the bulk for $K_{N}(u,v)$ in  \eqref{matrixkernel} as   $\tau$ tends to 1 at a certain rate?

\
\
\

{\it Acknowledgements:}
We are very grateful to P.J.  Forrester  for his valuable  comments.
D.-Z. Liu was   supported   by the Natural Science Foundation of China \# 11771417, the Youth Innovation Promotion Association CAS  \#2017491,  Anhui Provincial Natural Science Foundation \#1708085QA03 and the Fundamental Research Funds for the Central Universities \#WK0010450002.



\end{document}